\documentclass[11pt,a4paper]{article}

\usepackage[left=2.20cm,right=2.20cm,top=2.50cm,bottom=2.50cm]{geometry}

\usepackage{amsmath,amssymb,amsthm,mathtools,bbm,mathrsfs}
\usepackage{graphicx}
\usepackage{xcolor}
\usepackage{enumitem}
\usepackage{hyperref}
\usepackage{cleveref}
\usepackage{float,caption} 

\hypersetup{colorlinks=true,linkcolor=blue,citecolor=blue,urlcolor=blue}

\newcommand*\samethanks[1][\value{footnote}]{\footnotemark[#1]}
\newcommand{\R}{\mathbb{R}}
\newcommand{\Z}{\mathbb{Z}}
\newcommand{\NN}{\mathbb{N}}
\newcommand{\bS}{\mathbb{S}}
\newcommand{\D}{\mathbb{D}}
\newcommand{\T}{\mathbb{T}}
\newcommand{\sL}{\mathscr{L}}
\newcommand{\cX}{\mathcal{X}}
\newcommand{\cD}{\mathcal{D}}
\newcommand{\cF}{\mathcal{F}}
\newcommand{\cG}{\mathcal{G}}
\newcommand{\cL}{\mathcal{L}}
\newcommand{\cW}{\mathcal{W}}

\newcommand{\abs}[1]{\left\lvert#1\right\rvert}
\newcommand{\norm}[1]{\left\lVert#1\right\rVert}
\newcommand{\ip}[2]{\left\langle #1,#2\right\rangle}

\DeclareMathOperator{\Tr}{Tr}

\DeclareMathOperator{\id}{id}
\DeclareMathOperator{\conv}{conv} 
\DeclareMathOperator{\Fix}{Fix} 
\DeclareMathOperator{\Span}{span} 
\DeclareMathOperator{\diag}{diag}  
\DeclareMathOperator{\diam}{diam}

\DeclareMathOperator{\Lip}{Lip}  
  
\DeclareMathOperator{\per}{per}  
\DeclareMathOperator{\pyr}{pyr}  

\renewcommand{\div}{\operatorname{div}}

\numberwithin{equation}{section}

\newtheorem{theorem}{Theorem}[section]
\newtheorem{lemma}[theorem]{Lemma}

\newtheorem{proposition}[theorem]{Proposition}
\newtheorem{remark}[theorem]{Remark}

\newtheorem{hypothesis}{Hypothesis}

\setlist[enumerate]{itemsep=2pt,topsep=2pt}

\title{Local crystallization in three-dimensional quantization}

\author{
Ata Deniz Ayd{\i}n
\thanks{ETH Z\"urich, Department of Mathematics, R\"amistrasse 101, 8092 Z\"urich, Switzerland.\\
Email: \textsf{deniz.aydin@math.ethz.ch}, \textsf{mikaela.iacobelli@math.ethz.ch}}
  \quad
Mikaela Iacobelli\samethanks
}

\date{\vspace{-8ex}}

\begin{document}


\maketitle

\bigskip

\begin{abstract}
A long-standing conjecture on the quantization problem due to Gersho, which remains open in dimensions higher than two, posits that optimal quantizers crystallize into a lattice structure; in dimension three, the natural candidate is the body-centered cubic (BCC) lattice, known to be optimal among other lattices. 
In this paper, we establish the asymptotic local optimality of the BCC lattice under small periodic perturbations, extending its known stability beyond the class of lattice configurations. Our approach connects the discrete geometry of Voronoi tessellations with a continuum variational problem, showing that the rescaled quadratic quantization energies converge to a continuum functional describing locally deformed BCC lattices. A central ingredient is a geometric and algebraic framework that combines the stability of the BCC Voronoi structure, a discrete calculus on its adjacency graph and an analysis of the lattice symmetries via representation theory, allowing us to establish the positivity of the second variation at the BCC configuration and thereby prove the local minimality of the BCC lattice for the limiting energy. 
\end{abstract}
\medskip

\section{Introduction}

This paper is devoted to the study of the three-dimensional quantization problem, with the aim of understanding the local optimality and stability of the body-centered cubic (BCC) lattice beyond the class of lattice configurations. The quantization problem, whose precise formulation we recall below, seeks to approximate a given probability distribution by a prescribed number of representative points with minimal average error, a question that arises in many areas of mathematics and its applications, ranging from signal processing and data compression to approximation theory and optimal location problems.

The structure of the problem depends strongly on the dimension, and the extent of our understanding differs substantially between the one-, two-, and three-dimensional settings. The three-dimensional case, which is the natural one to describe configurations in physical space, remains significantly less understood than its lower-dimensional counterparts.

For a uniform distribution, the problem becomes a geometric question about how points should be arranged in space, with equally spaced points being optimal in one dimension and regular triangular patterns achieving asymptotic optimality in two dimensions. In three dimensions, the conjectured optimal arrangement is the body-centered cubic (BCC) lattice, which can be constructed by placing points at the vertices and centers of the cubes in a regular cubic tiling.
The BCC lattice is known to be optimal among lattices, but it remains an open question whether this optimality carries over to arbitrary non-lattice configurations. 

We address this question from a local perspective by studying the stability of the BCC lattice under small periodic deformations of space.
This local setting allows us to exhibit the discrete quantization problem, under appropriate rescaling, as a graph discretization of a continuous variational problem, from which we obtain the convergence of the quantization error to an integral functional which depends on the gradient of the deformation.
In the absence of a tractable closed-form formula for the limiting functional, we exploit the representation theory of the BCC symmetry group to compute the second variation at the unperturbed configuration and thus conclude that the BCC lattice is a strict local minimizer of the limiting energy.

\subsection{Historical survey of quantization}

The quantization problem has been introduced in various formulations across disparate fields, from signal processing and operations research to convex geometry and calculus of variations. 
The first explicit formulation of the quantization problem on general Euclidean spaces dates to Steinhaus \cite{steinhaus}, and asks to partition a bounded domain $\Omega$ into $N$ regions with minimal total moment of inertia:
\begin{equation}\label{eq:steinhaus} 
\begin{aligned}
\min_{A_i,\, x_i} \quad & \sum_{i=1}^N \int_{A_i} \abs{x-x_i}^2 \,dx \\ \mathrm{s.t.} \quad & \bigsqcup_{i=1}^N A_i = \Omega; \quad x_1, \ldots, x_N \in \R^d. \end{aligned} 
\end{equation}
More generally, one can replace bounded domains with probability measures $\mu$, and the quadratic power with any $r$th power for $r \in (0,\infty)$; in this case, the value of the minimum is denoted by $V_{N,r}(\mu)$ \cite{pagesintro}. 
Our results will be restricted to the above formulation, with $r = 2$ and the choice of the Lebesgue measure (more precisely, the volume form of the flat torus), but we will also review for completeness the known results for the more general case. For the review and equivalence of different formulations of the quantization problem, we refer to \cite[Sect.~3]{quantbook}.

This \emph{sum of moments} formulation of the quantization problem exhibits a double minimization structure, with two sets of variables $\{A_i\}_i$ and $\{x_i\}_i$ to optimize, such that when one set is fixed, there is a unique tractable minimizer for the other. Fixing the sets $\{A_i\}_i$, each $x_i$ is optimally taken to be the centroid of the set $A_i$, and fixing the points $\{x_i\}_i$, each set $A_i$ is optimally taken to be (up to zero measure) the \emph{Voronoi cell} $V_i$ of $x_i$, comprised of all points in space closer to $x_i$ than any other $x_j$:
\[ V_i := \{ x \in \Omega \mid \abs{x-x_i} = \min_{j=1,\ldots,N} \abs{x-x_j} \}. \]
The problem of optimal quantization is hence also referred to as \emph{centroidal Voronoi tessellations} \cite{cvt}: stationary Voronoi diagrams such that each point $x_i$ is the centroid of its Voronoi cell. The well-known \emph{Lloyd--Max algorithm} \cite{lloyd, max} iteratively improves the quantization error by coordinate descent, alternately optimizing $\{A_i\}_i$ given $\{x_i\}_i$ and $\{x_i\}_i$ given $\{A_i\}_i$. The Lloyd--Max algorithm admits various other interpretations, whether as gradient descent or as a Newton's method \cite{Bottou:1994aa}, and is guaranteed to converge to a local minimizer \cite{cvt_lloyd, Bourne:2015, Portales:2025aa}. 

However in general, the quantization problem admits many local minimizers. This is due to the nonconvexity of the cost, which can be observed by optimizing for $\{A_i\}_i$ and expressing the quantization error solely in terms of the points $\{x_i\}_i$:
\[ \min_{x_i} \quad \int_\Omega \min_{i=1,\ldots,N} \abs{x-x_i}^2 \,dx; \quad x_1, \ldots, x_N \in \R^d. \]
While compactness arguments yield the existence of global minimizers for any fixed $N$ (cf.~\cite[Thm.~4.12]{quantbook}, \cite[Sect.~3.2]{cvt}), it is generally intractable to find or characterize them. This has motivated both the asymptotic analysis of the quantization error as $N \to \infty$, and the use of heuristics for the construction of quantizers, in particular random samples of points (\emph{random quantizers}) and rescaled lattice points (\emph{lattice quantizers}) \cite[Sect.~8--9]{quantbook}. 

\paragraph{Asymptotic quantization.} The main theorem describing the asymptotics of the quantization error on Euclidean spaces is \emph{Zador's theorem} \cite{zador}, which has been proven gradually in broader and broader settings:

\begin{theorem}[Zador]
Let $\mu$ be a probability measure on $\R^d$ such that $\int \abs{x}^{r+\delta} \,d\mu(x) < \infty$ for some $\delta > 0$. 
Then denoting by $\rho$ the absolutely continuous component of $\mu$,
\[ \lim_{N \to \infty} N^{r/d} V_{N,r}(\mu) = C_{r,d} \left( \int \rho(x)^{\frac{d}{r+d}} \,dx \right)^{\frac{r+d}{d}}, \]
where $C_{r,d} > 0$ is a dimensional constant, corresponding to the quantization rate of the uniform measure on the unit cube $\mu = \sL^d|_{[0,1]^d}$. 
\end{theorem}

The study of asymptotic quantization in the univariate case ($d=1$) in the signal processing literature dates back to the late 40s, with the above optimal rate first proven for densities on bounded intervals \cite{bennett, pandit, Smith:1957aa, Algazi:1966aa} and generalized to the unbounded case with the above integral condition \cite{pierce}. 
In one dimension, \emph{uniform quantizers} given by centers of uniform-length interval partitions are optimal quantizers for the uniform measure for all $N$, yielding the explicit value $C_{r,1} = \frac{1}{2^r (r+1)}$. 
Key to the asymptotic quantization of non-uniform measures is the so-called \emph{companding} (\emph{com}press-ex\emph{pand}) strategy, which models variable-length partitions by uniformly quantizing a nonlinear distortion of the non-uniform data. 
The limiting error is given by \emph{Bennett's formula} \cite{bennett}, an integral expression in terms of the derivative of the distortion, and choosing the optimal distortion attains the minimal asymptotic rate for the quantization error.
We refer to \cite{gersho} and \cite{grayneuhoff} for a more detailed survey of these results.

The original proofs of Zador's theorem by Zador \cite{zador} and Bucklew and Wise \cite{buckwise} are also based on a companding argument. However, unlike the univariate case, for generic densities the compander approach fails to attain the optimal rate, hence can only provide an upper bound \cite{gersho}. The first complete proof of Zador's theorem as stated above is due to Graf and Luschgy \cite[Thm.~6.2]{quantbook}. In recent years, the theorem has also been extended to Riemannian manifolds \cite{Gruber2004, discapprox, iacasym, Aydn:2025aa}, Carnot groups \cite{carnotquant}, and rectifiable measures and in a weaker sense to general metric measure settings \cite{Aydn:2026aa}. 

\paragraph{Fejes T\'oth's moment theorem: hexagonal optimality in 2D.}
L.~Fejes T\'oth, following the Steinhaus formulation, independently proved the above theorem for densities on bounded two-dimensional domains, in the case $r = 2$, with again an explicit value for the constant $C_{2,2}$ \cite{Fejes-Toth:1959vj}: 
\[ \lim_{N \to \infty} N V_{N,2}( f \sL^2|_\Omega) = \frac{5}{18\sqrt3} \left( \int_\Omega \sqrt{f(x,y)} \,dx \,dy \right)^2. \]
We also cite Newman's independent proof \cite{newman} of the above result for the uniform measure on the unit square, and Gruber's simplified proof \cite{Gruber:1999tm} based on convexity arguments. 
The constant $C_{2,2} = \frac{5}{18\sqrt3}$ is the normalized second moment of the regular hexagon, and the heuristic implication is that in the limit of large $N$, Voronoi cells approach regular hexagons with areas varying with the density. This was later made rigorous by Gruber's stability theorem \cite{Gruber2001, Gruber2004}, showing that any asymptotically optimal sequence of quantizers is \emph{asymptotically regular hexagonal} in a precise sense, on both Euclidean and Riemannian settings. 

In fact, in the uniform case, hexagonal configurations are universally optimal: replacing the quadratic cost with any monotone function $\phi$ of the distance, the cost of any set of $N$ points admits the lower bound
\[ \int_\Omega \min_{i=1,\ldots, N} \phi(\abs{x-x_i}) \,dx \geq N \int_{H_N} \phi(\abs{x}) \,dx \]
where $H_N$ is a regular hexagon centered at the origin with area $\abs{\Omega}/N$ \cite[(3.8.4)]{Fejes-Toth:2023aa}, \cite{Gruber:1999tm}, \cite[App.~B]{Osting:2015aa}. 
The right-hand side corresponds to the total cost of a tiling of $N$ regular hexagons, covering the same area as $\Omega$, but not necessarily the exact shape of $\Omega$ itself. 
With the particular choice $\phi(t) = t^r$, this implies that $C_{r,2}$ is the normalized $r$th moment of the regular hexagon.
We also refer to \cite[Sect.~3.8]{Fejes-Toth:2023aa} and \cite{Gruber2004} for various applications of Fejes T\'oth's general moment theorem.

\paragraph{Gersho's conjecture and lattice quantizers.}
In the asymptotic analysis of quantization on general Euclidean spaces (\emph{vector quantization}), Gersho's seminal paper \cite{gersho} brought together the companding approach in one dimension and Fejes T\'oth's theorem in two dimensions. 

Gersho conjectured that the optimal asymptotic quantization rate for the uniform measure is attained by tilings consisting of congruent copies of a single polytope $P$, as in the case of dimensions 1 and 2, so that the optimal constant $C_{r,d}$ becomes the normalized moment of $P$. As well as other candidate polytopes in higher dimensions, Gersho proposed in dimension 3 the truncated octahedron, the Voronoi cell of the BCC lattice. 

The original purpose of Gersho's conjecture was to derive a higher-dimension generalization of Bennett's formula, considering nonlinear distortions of uniform tilings as quantizers for non-uniform measures. As Gersho observed however, in higher dimensions such an approach cannot attain the optimal quantization rate for generic densities: already in dimension two, the logarithm of the density
must be harmonic, which rules out natural examples such as the normal distribution \cite[IX]{gersho}. Nevertheless, Gersho's heuristic formula provides a rigorous upper bound on the optimal rate. We note also that despite this obstacle to attaining the optimal rate using a single distortion, optimal non-uniform quantizers can be constructed from uniform ones in a piecewise manner, at least in dimension two using hexagonal quantizers \cite{Su:1997aa}.

In general, any admissible space-filling polytope $P$ yields an upper bound on $C_{r,d}$ by the normalized $r$th moment of $P$ \cite[Thm.~8.5]{quantbook}. 
This holds in particular for the Voronoi cell of any lattice: one can construct \emph{lattice quantizers} for any Jordan measurable region by rescaling a given lattice and taking the points that lie within the region as the set of quantizers. As the scale of the lattice tends to zero, the contribution of boundary artifacts vanishes and the asymptotic quantization rate approaches the normalized $r$th moment of the Voronoi cell of the lattice. For more details, we refer to \cite[Sect.~8.3]{quantbook}. 

The \emph{lattice quantizer problem} then seeks to find a lattice Voronoi cell with minimal normalized moment:
\begin{equation}\label{eq:latquant} 
\begin{aligned}
\min_{\Lambda} \quad M_{r,d}(\Lambda) &:= \abs{\det \Lambda}^{-\frac{r+d}{d}} \int_{V(\Lambda)} \abs{x}^r \,dx \\ \mathrm{s.t.} \,\qquad  V(\Lambda) &:= \bigl\{ x \mid \abs{x} = \min_{l \in \Lambda} \abs{x-l} \bigr\}, \quad \Lambda \subset \R^d \text{ lattice.} 
\end{aligned} 
\end{equation}
Apart from dimensions 1 and 2, the only case where the minimizer to the lattice quantizer problem is known is dimension 3 with $r = 2$.
In this setting, the pivotal theorem of Barnes and Sloane \cite{barslo} proved not only that the BCC lattice is a global minimizer of the lattice quantizer problem, but that it is the unique minimizer and no other lattice is even a local minimizer of the normalized moment. Their proof relied on a representation of the normalized moment in terms of the quadratic form induced by the lattice, proving that the quadratic form of a locally minimizing lattice must have equal coefficients \cite[Thm.~3]{barslo}.

Another noteworthy result regarding the lattice quantization problem is the Zamir--Feder theorem \cite[Thm.~1]{Zamir1996}, valid for $r = 2$ and $d$ arbitrary: the Voronoi cell of any minimizing lattice has a \emph{white} covariance matrix (i.e.~$cI_d$ for some $c > 0$). This holds in particular for the hexagonal lattice and the BCC lattice. For more background and numerics on lattice quantizers, we refer to \cite{Conway1982}, \cite{Conway1984}, \cite[Sect.~21]{Conway:1999aa}, \cite[Sect.~8]{quantbook}, \cite{10168914}, \cite{10979772}.

Note that these results concern the restricted minimization problem over lattices, while Gersho's conjecture asks whether the minimizing lattice is also optimal among all non-lattice configurations, which remains open for $d \geq 3$. In dimension 3, numerical evidence indicates that the normalized second moment of the BCC lattice is close to the theoretical lower bound on $C_{2,3}$ \cite[Ex.~8.14]{quantbook}. Simulations based on Lloyd iteration, as well as other methods from a statistical mechanics perspective, also indicate the local stability of the BCC lattice among non-lattice perturbations \cite{cvt:2005, Klatt:2019aa, Hain:2020aa}. 
Based on Gruber's proof of Fejes T\'oth's theorem in 2D, Choksi and Lu \cite{Choksi:2020aa} have proposed a possible strategy for proving Gersho's conjecture in 3D and obtained partial results restricting the geometric complexity of optimal quantizers. We cite also \cite[7.2]{discapprox}, \cite{Zhu2019pre} for weaker forms of Gersho's conjecture proven in more general settings.

While finalizing this manuscript, we became aware of the concurrent preprint by Luo and Xie \cite{LuoXie2026}, which establishes sharp coercivity and nonlinear stability estimates for the quadratic quantization energy of the BCC lattice under periodic perturbations. Although their second variation results appear analogous to ours, their subsequent approaches and objectives diverge. Luo and Xie work directly with the discrete quantization energy, relying on Fourier analysis and explicit algebraic inequalities to obtain stability, while our approach establishes a discrete-to-continuum variational framework in terms of discrete calculus on the BCC adjacency graph, and relies on the representation theory of its symmetry group in the computation of the second variation, leading to a continuum local optimality result and a nonlinear gradient-flow analysis establishing exponential relaxation towards the BCC lattice. The two works have been developed independently.

\paragraph{Comparison with related problems.}
Gersho's conjecture places quantization in the extensive list of arrangement problems where lattice configurations are known or conjectured to be optimal, whether among other lattices or in general. 

The \emph{sphere-packing problem} is likely the best-known example of such a problem, and looks for the largest density of space that can be covered by non-intersecting balls of equal radius. 
The global optimality of the integer lattice in one dimension and the hexagonal lattice in two dimensions (the latter known as \emph{Thue's theorem}) is rather straightforward to obtain; cf.~\cite[Sect.~3]{Fejes-Toth:2023aa}, \cite[Sect.~4.1]{Gruber2000}, \cite{Chang2010} for the 2D case.

In three dimensions, Kepler conjectured that the FCC lattice, the dual to the BCC lattice, provided the optimal sphere packing. Gauss proved the optimality of the FCC lattice among other lattices, but proving optimality also in general required the enumeration of a lengthy yet finite list of configurations, as outlined by Fejes T\'oth \cite{Fejes:1942aa}. Hales' computer-assisted proof implementing Fejes T\'oth's strategy demonstrated that the FCC lattice does indeed provide the optimal sphere packing among all configurations, and was only fully formalized in the last decade \cite{Hales:2005, Hales:2006aa, Hales:2017aa}. 

In contrast with the 3D case, Viazovska \cite{Viazovska:2017my} in dimension 8, and Cohn, Kumar, Miller, Radchenko and Viazovska \cite{Cohn:2017wf} in dimension 24, proved that the $E_8$ root lattice and the Leech lattice, respectively, are globally optimal for the sphere packing problem among both lattice and non-lattice
configurations. These methods subsequently led to universal
optimality results \cite{Cohn:2022aa}. Both lattices are also conjectured to be optimal lattice quantizers based on numerical evidence \cite{Conway1984}.

Heuristically, sphere packing can be thought of as a dual problem to quantization: areas not covered by a packing present an obstacle to efficient quantization. 
Indeed the sphere covering problem, the more direct dual to the sphere packing problem, is a limiting case of the quantization problem with $p = \infty$ \cite[Sect.~10]{quantbook}, and the BCC lattice is known to be optimal for the covering problem among lattices \cite{bambah1954lattice, Barnes_1956, Few_1956}. 
In light of this, as well as other numerical evidence in dimensions up to 24, Conway and Sloane \cite{Conway1984} conjectured that the optimal lattice for quantization is dual to that of the sphere packing problem, while noting that duality did not necessarily hold between the best known lattices for packing and covering problems. Recently constructed lattice quantizers in higher dimensions also perform better than the duals of best known packing lattices, casting doubt on a strict duality relation between quantization and sphere packing \cite{10168914, 10979772}.

We also cite existing literature on Coulomb/Riesz interaction systems, in which case the question whether optimal configurations are lattices is referred to as the \emph{crystallization conjecture}; see \cite{Blanc:2015aa} for a general review. Such interaction systems contain sphere packing as an extremal case, in which particles behave as hard spheres with equal radius \cite[Sect.~2.3]{Blanc:2015aa}. However, less is known for general interaction potentials than the sphere packing problem: the hexagonal lattice in 2D has only been conjectured to be universally optimal \cite{Sandier:2015aa}, and recently shown to be locally optimal \cite{Leble2025pre}. 
In dimension 3, either the BCC or the FCC lattice can be an energy minimizer depending on the potential \cite{Betermin:2023aa}.
On the other hand, the Viazovska approach has also shown the universal optimality of the $\mathrm{E}_8$ resp.~Leech lattices in dimensions 8 resp.~24
\cite{Petrache:2020aa}. 
See also \cite{Klatt:2019aa} for a treatment of the quantization problem in the perspective of interacting particle systems, and \cite{Huesmann:2026aa} for connections to semi-discrete optimal transport.

We lastly mention the \emph{Kelvin problem}, concerning the arrangement of soap bubbles of equal volume with minimal surface area. Kelvin's conjecture that the BCC lattice (up to appropriate modifications to satisfy Plateau's laws) provided the optimal arrangement was refuted by the counterexample of Weaire and Phelan \cite{WeairePhelan}. 
A more restricted conjecture (\emph{truncated octahedron conjecture}) asks if the BCC tessellation is optimal among tilings by translation \cite{Cesaroni:2026aa}, which is the inverse of what is known in quantization: the non-lattice optimality of the BCC lattice is disproven while the lattice optimality has only recently been proven in the affirmative \cite{Hales2026, Cesaroni2026}. 
The two-dimensional analogue of the Kelvin problem is nevertheless universally minimized by the hexagonal lattice, per the \emph{honeycomb theorem} of Hales \cite{Hales:2001aa}.
We note finally that the Weaire--Phelan counterexample has been tested for 3D quantization as well, but even after further optimization, remains inferior to the BCC lattice \cite{Kashyap:2001aa, Sloane2002-gs}.

\paragraph{Gradient flow approach.} 
In recent years, a gradient flow approach to the quantization problem has been introduced by Caglioti, Golse, and Iacobelli, first in dimension one \cite{Caglioti_2015} and subsequently in dimension two \cite{CGI_2D}. The idea is to let the points evolve according to the gradient flow of the quantization energy and to study whether, as their number tends to infinity, this discrete evolution can be described by the gradient flow of a continuous functional.

In dimension one, for quadratic cost and a density $\rho$ sufficiently close to a constant in $C^2$, the authors proved that, under suitable assumptions on the initial data and after rescaling time, the discrete and continuous gradient flows remain uniformly close in $L^2$ for all times. Moreover, entropy-dissipation estimates for the continuous equation give exponential convergence to equilibrium, whose density in Eulerian variables is proportional to $\rho^{1/3}$, in agreement with the static theory.

The Eulerian formulation of this one-dimensional evolution also led to the study of weighted ultrafast diffusion equations arising from quantization. Their Wasserstein gradient flow structure was investigated by Iacobelli \cite{Iacobelli:2019aa}, who proved
exponential convergence to equilibrium and gave sufficient conditions for $W_2$-stability. Iacobelli, Patacchini, and Santambrogio \cite{Iacobelli:2019wz} subsequently established well-posedness, regularity estimates, and exponential convergence in arbitrary dimensions under suitable assumptions, with periodic or no-flux boundary conditions. More recently, Fathi and Iacobelli
\cite{FathiIacobelli2026} extended the analysis beyond the compact setting, proving exponential convergence on the whole real line for a class of log-concave, log-Lipschitz weights.

In dimension two, the asymptotic optimality of the hexagonal
lattice was already known, which suggested using its points
as the reference configuration. The limiting energy obtained
in \cite{CGI_2D} depends singularly on the determinant of the deformation gradient and has a nonconvex density, so the authors considered small periodic deformations and studied the linearized equation. Combining this analysis with $\varepsilon$-regularity for parabolic systems, they controlled the nonlinear evolution and proved exponential convergence to the hexagonal configuration, up to translation, as well as local minimality for the limiting
energy.

For a more detailed presentation of this approach and the results obtained in dimensions one and two, we refer to the surveys \cite{Iacobelli:2018tz,Golse2018}.

\subsection{Setting, theorem statements and outline}

\subsubsection{The BCC lattice}

We denote by $\sL \subset \R^3$ the unscaled BCC lattice, defined to be the set of all elements of $\Z^3$ whose coordinates are all congruent mod $2$:
\[ \sL = \{(0,0,0),(1,1,1)\} + (2\Z)^3 \subset \Z^3. \]
One choice of basis for $\sL$ is the set $\{(-1,1,1),(1,-1,1),(1,1,-1)\}$. We denote by $N_0 \subset \sL$ the neighbors of $0$ in the Voronoi diagram of $\sL$, and by $G$ the symmetry group of $\sL$. 

One property of the BCC lattice which will be crucial to our analysis is that its incidence relations are stable under small Lipschitz perturbations. Geometrically, this relates to the fact that the fundamental Voronoi cell of the BCC lattice has no degenerate faces with empty interior, and every vertex lies at the intersection of exactly three faces. We prove this property in Section \ref{sect:nbrstab}.

\subsubsection{Setting of the quantization problem}

In order to eliminate any transient contributions from boundary effects, we pose the quantization problem on the flat torus $\T := \R^3/\sL$. Since $\T$ is a compact manifold without boundary, Zador's theorem applies to the optimal quantization error of its volume form:
\[ \lim_{N\to\infty} N^{2/3} V_{N,2}(\T) = Q_{2,3}([0,1]^3) \abs{\T}^{5/3}. \]
The object of our interest is the analogue of the limit
on the left-hand side for configurations obtained as small
perturbations of the BCC lattice.

For ease of construction, we restrict to $N = n^3$ for $n \in \NN$; note that this does not affect the limit of the optimal quantization error. With the scaling factor $\eta := 1/n$, the rescaled BCC lattice $\sL_\eta := \eta \sL/\sL$ defines a set of quantizers on $\T$ with $N$ elements. This is the set for which we will consider small perturbations of the BCC lattice, defined in the following sense:

\begin{hypothesis}\label{hyp1}
Let $X \in C^1(\T; \T)$ be a diffeomorphism of $\T$ which is $C^1$-close to the identity map: $\norm{\nabla X - I}_\infty \leq \rho_0$ for a constant $\rho_0 > 0$ to be determined. 
For each $\eta = 1/n$, we then take the set of quantizers $\cX_\eta := X(\sL_\eta)$.
\end{hypothesis}

The admissibility radius $\rho_0$ will be determined by various constraints throughout the paper, and will ensure in particular that the perturbed and unperturbed quantizers $\cX_\eta$ and $\sL_\eta$ have the same incidence graph.

In practice, we will represent $X$ as $\id_\T + u$ for some $u \in C^1(\T; \R^3)$ with $\norm{\nabla u}_\infty \leq \rho_0$, in the sense that
\[ X(x+\sL) = x + u(x) + \sL. \]
The existence of $u$ given $X$ follows by taking a lift of $X-\id$ from $\T$ to $\R^3$. 

In order to control the size of $u$ and eliminate the translation-invariance of the quantization problem of $\cX_\eta$ along any shift $X \mapsto X + a$, we can also impose wlog the following zero-mean condition:

\begin{hypothesis}\label{hyp2}
The diffeomorphism $X$ is given by $X = \id + u$ for some $u \in C^1(\T; \R^3)$ such that $\norm{\nabla u}_\infty < \rho_0$ and $\int_\T u(x) \,dx = 0$.
\end{hypothesis}

Note that even though the translation-invariance of the quantization error holds directly on the flat torus $\T$, this choice of a privileged solution requires the lifts $u$ rather than the diffeomorphisms $X$, since the integral of a $\T$-valued function is ill-defined.

We then define the functionals
\[ \cF_N(X) := N^{2/3} \int_\T d(x, \cX_\eta)^2 \,dx, \quad N = n^3, \, \eta = n^{-1}, \, n \in \NN. \]
Our first result is the following:

\begin{theorem}\label{thm1}
Under the assumptions of Hypothesis \ref{hyp1}, 
\[ \lim_{N \to \infty} \cF_N(X) = \cF(X) := \frac1{\abs{\T}} \int_\T F(\nabla X) \,dx,  \]
where for $M \in \R^{3\times3}$, $F(M)$ is the moment of the Voronoi cell of $0$ with respect to the linearly perturbed lattice $M \sL$:
\[ F(M) := \int_{V(M)} \abs{x}^2 \,dx; \quad V(M) := \{x \mid \abs{x} = d(x,M\sL)\}. \]
\end{theorem}

This statement is also quantitative: the rate of convergence is controlled by the modulus of continuity of $\nabla X$. We note that thanks to the choice of $\T$ as the domain of quantization, each $\cF_N(\id) = \cF(\id) = F(I)$ for the unperturbed BCC lattice.

Note that unlike the Barnes--Sloane result, our limiting functional involves the \emph{unnormalized} moment induced by $M$, which is not invariant under scaling. In the absence of normalization, the functional nevertheless remains meaningful thanks to the rigidity of the constraint that $X$ is a diffeomorphism of $\T$.

\subsubsection{Finite-difference representation}

Even for small perturbations of the BCC lattice, deriving explicit formulas for $\cF_N$ and $\cF$ and differentiating them by hand would require lengthy and cumbersome calculations involving the geometry of the deformed Voronoi cells, where the properties needed for stability are difficult to read. We therefore develop a new approach based on the observation that, once the variations are expressed in geometric terms, the cubic symmetry can be exploited via representation theory to determine the form of the second variation, leaving only a small number of coefficients to be computed explicitly.

We prove Theorem \ref{thm1} by exhibiting $\cF_N(X)$ as a finite-difference approximation of $\cF(X)$. Defining the discrete derivatives
\[ D_\eta^{(l)} X(x_0) := \frac{X(x_0+\eta l) - X(x_0)}{\eta} = l + \frac{u(x_0+\eta l) - u(x_0)}{\eta}, \quad x_0 \in \sL_\eta, \, l \in \sL, \]
the set $\cD_\eta X(x_0) := \{D_\eta^{(l)} X(x_0)\}_{l \in \sL}$ defines a locally finite set of quantizers on $\R^3$. Denoting by
\[ V(\cD_\eta X(x_0)) := \{ x \in \R^3 \mid \abs{x} = d(x, \cD_\eta X(x_0)) \} \]
the Voronoi cell of $0$ with respect to $\cD_\eta X(x_0)$, we can represent $\cF_N(X)$ as
\[ \cF_N(X) = \frac1N \sum_{x_0 \in \sL_\eta} F(\cD_\eta X(x_0)); \quad  F(\cD_\eta X(x_0)) := \int_{V(\cD_\eta X(x_0))} \abs{x}^2 \,dx. \]
Note that these definitions of $F$ and $V$ are consistent with the above definitions, up to identifying the matrix $M$ with the linearly perturbed lattice $M \sL$. We will also employ this convention below. 

This correspondence also carries over to the first variations of $\cF_N$ and $\cF$:

\begin{theorem}\label{thmdiff}
Let $X$ satisfy Hypothesis \ref{hyp1}, $v \in C^1(\T; \R^3)$. Then the first variations of $\cF_N$ and $\cF$ at $X$ in the direction $v$ are given by 
\begin{align} \label{eq:dfnul}
\delta \cF_N(X)[v] & = \frac1N \sum_{x_0 \in \sL_\eta} \sum_{l \in N_0} U_l(\cD_\eta X(x_0)) \cdot D^{(l)}_\eta v(x_0); \\ \label{eq:dful}
\delta \cF(X)[v] &= \frac1{\abs\T} \int_\T \sum_{l \in N_0} U_l(\nabla X(x)) \cdot (\nabla v(x) l) \,dx. 
\end{align}
\end{theorem}

The vectors $U_l(\cD_\eta X_t(x_0))$ resp.~$U_l(\nabla X_t(x))$ are functions of the Voronoi cell $V(\cD_\eta X_t(x_0))$ resp.~$V(\nabla X_t(x))$, whose precise definitions are given in \eqref{eq:uldef}. For now, we note that the terms $U_l$ define a directed edge function on the adjacency graph of $\sL_\eta$, in the sense that
\[ U_{-l}(\cD_\eta X_t(x_0+\eta l)) = - U_l(\cD_\eta X_t(x_0)). \]
One can then think of $U_l(\cD_\eta X_t(x_0))$ as the flux from vertex $x_0$ to its neighbor $x_0+\eta l$ in the direction of $l$.
This identity also yields the following alternative expression for the discrete first variation:
\begin{equation}\label{eq:dfndivform} 
\delta \cF_N(X)[v] = \frac1N \sum_{x_0 \in \sL_\eta} v(x_0) \cdot \biggl( -\frac2\eta \sum_{l \in N_0} U_l(\cD_\eta X(x_0)) \biggr). \end{equation}
This is the discrete analogue of the integration by parts
\begin{equation}\label{eq:dfdivform} 
\delta \cF(X)[v] = \frac1{\abs\T} \int_\T v(x) \cdot \Bigl( - \sum_{l \in N_0} (l \cdot \nabla) U_l(\nabla X(x)) \Bigr) \,dx; \end{equation}
in fact, the expression
\[ \frac2\eta \sum_{l \in N_0} U_l(\cD_\eta X(x_0)) = \sum_{l \in N_0} D^{(l)}_\eta U_l(\cD_\eta X(x_0-\eta l)) \]
is the appropriate notion of divergence on the adjacency graph of $\sL_\eta$, as detailed in Section \ref{sect:graphcalc}.  
In comparison, in the continuous setting, we have
\[ \div \nabla F(\nabla X(x)) = \div \sum_{l \in N_0} U_l(\nabla X(x)) l^\top = \sum_{l \in N_0} (l \cdot \nabla) U_l(\nabla X(x)). \]
This integration by parts shows that in both settings, the first variation vanishes for the unperturbed BCC lattice $X = \id$, as expected by the stationarity of lattice quantizers.

Consequently, the functionals $\cF_N$ and $\cF$ induce the following gradient flows on $u$:
\begin{align}\label{eq:discflow}
    \dot u_t^{(N)}(x_0) = \frac2\eta \sum_{l \in N_0} U_l(\cD_\eta X_t^{(N)}(x_0)) & = \sum_{l \in N_0} D^{(l)}_\eta U_l(\cD_\eta X_t^{(N)}(x_0-\eta l)); \\ \label{eq:contflow}
    \dot u_t(x) = \div_x \sum_{l \in N_0} U_l(\nabla X_t(x)) l^\top & = \sum_{l \in N_0} (l \cdot \nabla) U_l(\nabla X_t(x)).
\end{align}
The discrete flow \eqref{eq:discflow} is equivalent to continuous-time Lloyd iteration with velocities weighted by the Voronoi cell volumes, but having now the form of a nonlinear interacting particle system. We can then think of the continuous flow \eqref{eq:contflow} formally as the continuous limit of Lloyd's algorithm. 

We cite also \cite[Lem.~2.4]{Bourne:2015} as a prior example of weighted graph Laplacian structure observed in Lloyd's algorithm, for the related problem of semi-discrete optimal transport.

\subsubsection{Perturbative expansion around identity}

Having obtained exact formulas for the first variations of the functionals $\cF_N$ and $\cF$, we expand both functionals up to second order around the unperturbed BCC lattice.
We consider small perturbations of the form $X_\varepsilon = \id + \varepsilon u$ for $\norm{\nabla u} \leq \rho_0$. 
Differentiating $U_l$, we first obtain the expansion
\[ U_l(\cD_\eta X_\varepsilon(x_0)) = U_l(I) + \varepsilon \Bigl( A_l D^{(l)}_\eta u(x_0) + \sum_{l' \sim l} B_{ll'} D^{(l')}_\eta u(x_0)\Bigr) + O(\varepsilon^2 \norm{\nabla u}^2) \]
and likewise for $U_l(\nabla X(x))$, for fixed matrices $A_l$ and $B_{ll'}$ depending on the unperturbed BCC lattice. Applying Theorem \ref{thmdiff} then yields the following expansion:

\begin{theorem}\label{thmpert}
    Around $\varepsilon=0$, the functionals $\cF_N$ and $\cF$ admit the perturbative expansions
    \begin{align*}
        \cF_N(\id+\varepsilon u) & = \underbrace{\cF_N(\id)}_{\abs{}} + \frac12 \varepsilon^2 G_N(u) + O(\varepsilon^3), \\
        \cF(\id+\varepsilon u) & = \ \overbrace{\cF(\id)} + \frac12 \varepsilon^2 G(u) + O(\varepsilon^3);
    \end{align*}
    where
    \begin{align}\label{eq:gndef}
        G_N(u) & := \frac1N \sum_{x_0 \in \sL_\eta} \sum_{l \in N_0} D^{(l)}_\eta u(x_0) \cdot \Bigl( A_l D^{(l)}_\eta u(x_0) + \sum_{l' \sim l} B_{ll'} D^{(l')}_\eta u(x_0)\Bigr); \\ \label{eq:gdef}
        G(u) & := \frac1{\abs\T} \int_\T \sum_{l \in N_0} (\nabla u(x) l) \cdot \Bigl( A_l \nabla u(x) l + \sum_{l' \sim l} B_{ll'} \nabla u(x) l'\Bigr) \,dx.
    \end{align}
    Moreover, $G_N(u) \to G(u)$ as $N \to \infty$, and both admit the equivalent representations
    \begin{align}\label{eq:gnalt}
        G_N(u) & = \frac1N \sum_{x_0 \in \sL_\eta} \sum_{l \in N_0} D^{(l)}_\eta u(x_0)^\top C_l D^{(l)}_\eta u(x_0); \\ \label{eq:galt}
        G(u) & = \frac1{\abs\T} \int_\T \sum_{l \in N_0} (\nabla u(x) l)^\top C_l (\nabla u(x) l) \,dx
    \end{align}
    for $C_l := A_l + \sum_{l' \sim l} B_{l'l}$.
\end{theorem}

We obtain this equivalent form first in the discrete case, using \eqref{eq:dfndivform}, and then conclude the continuous case using $G_N(u) \to G(u)$. Note that while the integrals \eqref{eq:gdef} and \eqref{eq:galt} coincide, the integrands themselves are not identical, but differ a priori by a null-Lagrangian term whose integral over $\T$ is identically zero. 

Note also that the alternate representation \eqref{eq:gnalt} of the discrete Hessian is in the form of an interaction energy with matrix-valued potential
\[ \frac1N \sum_{x_0 \in \sL_\eta} \sum_{l \in N_0} \frac1{\eta^2}\left(u(x_0+\eta l) - u(x_0) \right)^\top  C_l \left(u(x_0+\eta l) - u(x_0)\right), \]
and the $\eta^{-2}$ scaling is what drives the convergence to the continuous Hessian. 

\subsubsection{Symmetry arguments}
We now exploit the invariance of the functionals $\cF_N$ and $\cF$ under the action of the symmetry group $G$ of the BCC lattice. From the representation theory of the conjugate action of $G$ on matrices, we obtain that the matrices $C_l$ are necessarily of the form
\begin{equation}\label{eq:clform} C_l = \frac{\alpha_l}{\abs{l}^2} I + \frac{\beta_l}{\abs{l}^4} l l^\top, \end{equation}
for $\alpha_l$ and $\beta_l$ constant along the $G$-orbits of $N_0$. We then compute explicitly
\begin{equation}\label{eq:alblcalc} \alpha_l = \left\{ \begin{matrix} -\frac{1}{24}, & l \in N_{ax}, \\ -\frac{15}{64}, & l \in N_{diag};\end{matrix} \right. \qquad \beta_l = \left\{ \begin{matrix} \frac{25}{24}, & l \in N_{ax}, \\ \frac{123}{64}, & l \in N_{diag}.\end{matrix} \right. \end{equation}
Note that since the coefficients $\alpha_l$ are negative, no individual matrix $C_l$ is positive definite. Nevertheless, we show in Section \ref{sect:hesscalc} that the integrand
\[ H(\nabla u(x)) := \sum_{l \in N_0} (\nabla u(x) l)^\top C_l (\nabla u(x) l) \]
is indeed strictly convex, up to the addition of a null-Lagrangian term integrating identically to zero. We then conclude the following:

\begin{theorem}
    The quadratic form $G(u)$ satisfies
    \[ \frac13 \frac1{\abs\T} \int_\T \norm{\nabla u}^2 \,dx \leq G(u) \leq \frac9{\abs\T} \int_\T \norm{\nabla u}^2 \,dx. \]
\end{theorem}

From this, we conclude that for a sufficiently small choice of $\rho_0$, the continuous functional $\cF$ is strictly convex and minimized at $u = 0$ among all perturbations satisfying Hypothesis \ref{hyp2}.

\subsubsection{Dynamical point of view}

We finally investigate the gradient flow \eqref{eq:contflow} induced by the limiting continuous functional $\cF(\id+u)$. Restricting to zero-mean perturbations, we show that the gradient flow exists, is unique and converges exponentially to the unperturbed lattice, provided the smallness of the initial perturbation with respect to a sufficiently strong Sobolev norm. The precise statements we prove are the following:

\begin{theorem}
On the space of diffeomorphisms on $\T$ satisfying Hypothesis \ref{hyp2}, the functional $\cF(X)$ is strongly convex on a small $W^{1,\infty}$-neighborhood of the identity map.

Let $3 < p < \infty$ and $1+3/p < \sigma < 2$. 
There exists $\delta_0 \ll 1$ such that, for all $u_0 \in C^1(\T; \R^3)$ satisfying Hypothesis \ref{hyp2} such that $\norm{u_0}_{W^{\sigma,p}} \leq \delta_0$, the Cauchy problem
\[ \partial_t u_t = \div_x \nabla F(I+\nabla_x u_t), \quad u(0,\cdot) = u_0 \]
admits a unique weak solution for all $t > 0$, which is globally bounded in $C^{1,\alpha}$ for some $\alpha \in (0,1)$ hence also a classical solution, and satisfies the estimates
\begin{align*}
    \norm{u_t}_{L^2(\T)} & \leq e^{- \lambda t} \delta_0; \\
    \norm{\nabla u_t}_{L^\infty(\T)} & \leq C \phi(\delta_0) \ll 1
\end{align*}
for some $\lambda, C > 0$, $\phi \colon [0,\infty) \to [0,\infty)$ continuous and increasing with $\phi(0) = 0$.
\end{theorem}

The results above give a precise meaning to the asymptotic local optimality of the BCC lattice beyond the class of lattice configurations. Indeed, for every fixed sufficiently small periodic deformation that is not a translation, the limiting quantization energy is strictly larger than that of the BCC lattice. 
We derive the coercivity needed for this conclusion from the aforementioned strict convexity of the second variation at zero, together with periodicity constraints implied by Hypothesis \ref{hyp2}.
Together with parabolic regularity, this coercivity also allows us to prove nonlinear stability and exponential convergence of the continuum gradient flow to the BCC configuration, for the initial data considered above.

\paragraph{Acknowledgments.}
The authors are partially supported by SNSF Starting Grant \linebreak TMSGI2\_226018.

\paragraph{AI declaration.}
No LLM tools have been used at any stage of the preparation of this paper.

\section{The BCC lattice and its symmetry group}\label{sect:bcc}

In this first section, we present the unperturbed BCC lattice and its discrete group of symmetries. For reference, see \cite[Sect.~5.9]{Serre:1977aa}, \cite[Sect.~8.2]{Fassler:1992aa}, \cite[Sect.~23]{Conway2016}, \cite[Sect.~11.4]{Johnson_2018}.

We denote by $\sL \subset \R^3$ the unscaled BCC lattice, defined to be the set of all elements of $\Z^3$ whose coordinates are all congruent mod $2$:
\[ \sL = \{(0,0,0),(1,1,1)\} + (2\Z)^3 \subset \Z^3. \]
One choice of basis for $\sL$ is the set $\{(-1,1,1),(1,-1,1),(1,1,-1)\}$. 

We will study the BCC lattice in light of two structures: a finite symmetry group $G \subset O(3)$ which keeps the lattice invariant, and a Cayley graph structure on $\sL$ which will characterize the neighborhood relations of the Voronoi diagram of $\sL$.

The symmetries of the unperturbed BCC lattice allow us to reduce many of the calculations below to the evaluation of a few scalar coefficients. We first recall the group actions and invariant subspaces used in these calculations.

\subsection{Basic representation-theoretic notions}

A \emph{linear (left) action} of a group $G$ on a real vector space $W$ is a map
\[
G \times W \to W, \qquad (g,w) \mapsto g\cdot w,
\]
such that $e\cdot w=w$, $g\cdot (w_1+w_2)=g\cdot w_1+g\cdot w_2$, $g \cdot (h \cdot w) = (gh) \cdot w$, and $g\cdot(\alpha w)=\alpha(g\cdot w)$ for all $\alpha\in\R$, where $e$ denotes the identity element of $G$.
In practice, we will only use:
\begin{itemize}
\item the action on vectors $x\in\R^3$, given by $x\mapsto g x$;
\item the induced action on matrices $A\in\R^{3\times 3}$, given by conjugation
\[
A \mapsto gAg^{\top},
\]
which is compatible with the above action because rank-one tensors transform as
\[
( gx)(gx)^{\top} = g(xx^{\top})g^{\top}.
\]
\end{itemize}

Given such an action, two notions will be crucial:
\begin{itemize}
\item The \emph{orbit} of an element $w\in W$ is the set $G\cdot w=\{g\cdot w:\ g\in G\}$.
In our context, orbits encode ``symmetry classes'' (e.g.\ axial vs diagonal neighbors).
\item A quantity is \emph{$G$-invariant} if it does not change under the action.
For example, a matrix $A$ is invariant under conjugation if $gAg^{\top}=A$ for all $g\in G$.
\end{itemize}

We will repeatedly use the following elementary observation:
if a finite set is a union of a few orbits, then any orbit-invariant scalar quantity takes only one value on each orbit.

\subsection{The symmetry group $G$}

The BCC lattice admits a discrete group $G \subset O(3)$ of symmetries, so that $g \sL = \sL$ for each $g \in G$.
Precisely, $G$ is the group of \emph{signed permutation matrices} in $\R^3$, i.e.\ matrices $g$ such that $g$ permutes the coordinate axes and possibly flips signs:
\[
g e_i \in \{\pm e_1,\pm e_2,\pm e_3\}, \qquad i=1,2,3.
\]
Equivalently, $g$ has exactly one nonzero entry in each row and column, and each nonzero entry is $\pm 1$.
In particular, $G\subset O(3)$, hence $g^{-1}=g^{\top}$ for all $g\in G$.
We denote by $G^{+}=G\cap SO(3)$ the orientation-preserving subgroup.

The reason for introducing $G$ is simple: it preserves the unperturbed BCC neighbor configuration and the Voronoi cell of the origin.
Thus, many objects attached to the Voronoi cell will be constant on $G$-orbits.

\paragraph{Alternative characterizations of $G$.}
We note also that $G \cong S_3 \ltimes \Z_2^3 = \Z_2 \wr S_3 \cong \Z_2 \times S_4$, with $G^+\cong \{0\} \times S_4 \cong S_4$. In literature, $G$ and $G^+$ are also called the full resp.~chiral octahedral group, and correspond to the symmetries of the cube and the octahedron.

Any element $g \in G$ acts on the basis vectors of $\R^3$ by
\[ g e_i = (-1)^{a_i} e_{\sigma_i}; \quad a = (a_1,a_2,a_3) \in \Z_2^3, \quad \sigma \in S_3. \]
Denoting by $\hat a$ the matrix with diagonal entries $(-1)^{a_i}$, and by $\hat \sigma$ the permutation matrix that sends each $e_i$ to $e_{\sigma_i}$, we have that
\[ g e_i = (-1)^{a_i} \hat \sigma e_i = \hat \sigma \hat a e_i \quad \implies \quad g = \hat \sigma \hat a. \]
The semidirect product structure is then apparent by observing that
\[ (\hat \sigma \hat a) (\hat \tau \hat b) = (\hat \sigma \hat \tau) (({}^{\hat\tau^{-1}} \hat a) \, \hat b) \]
where ${}^{\hat\tau} \hat a = \hat\tau \hat a \hat \tau^{-1}$, which is given by
\[ {}^{\hat\tau} \hat a  e_i = \hat\tau \hat a \hat \tau^{-1} e_i = \hat\tau \hat a e_{\tau^{-1}_i} = \hat \tau (-1)^{a_{\tau^{-1}_i}} e_{\tau^{-1}_i} = (-1)^{a_{\tau^{-1}_i}} e_i. \]
That is, ${}^{\hat\tau} \hat a = \widehat{({}^\tau a)}$ where $({}^\tau a)_i = a_{\tau^{-1}_i}$.

\noindent\textbf{Identification with $\Z_2 \times S_4$.}
It is already known that $\Z_2 \wr S_3 \cong \Z_2 \times S_4$; here we can deduce this more concretely from how the elements of $G^+$ (which is also the orientation-preserving symmetry group of the cube) permute the four diagonals of the cube.

First observe that $G \cong \Z_2 \times G^+$ along $g \mapsto (\det(g), \det(g)^{-1} g)$. So it suffices to show that $G^+ \cong S_4$.

Consider the action of $G^+$ on diagonals, i.e. equivalence classes of vectors $(-1)^{l_1} e_1 + (-1)^{l_2} e_2 + (-1)^{l_3} e_3$ along the pairing $v \sim -v$. Since each element $g \in G$ acts on basis vectors as $g e_i = (-1)^{a_i} e_{\sigma_i}$, we have
\[ g ((-1)^{l_1} e_1 + (-1)^{l_2} e_2 + (-1)^{l_3} e_3) = (-1)^{a_1+l_1} e_{\sigma_1} + (-1)^{a_2+l_2} e_{\sigma_2} + (-1)^{a_3+l_3} e_{\sigma_3}. \]
And since $g$ commutes with negation by linearity, it is compatible with the equivalence relation. $G^+$ therefore induces an action on the set of diagonals
\[ D = \{ [e_1+e_2+e_3], [-e_1+e_2+e_3], [e_1-e_2+e_3], [e_1+e_2-e_3]\}. \]
This action yields a homomorphism $\Phi \colon G^+ \to S_D \cong S_4$. Since both $G^+$ and $S_D$ have order $24$, $\Phi$ is injective iff it is surjective, so it suffices to show that the subgroup $\Phi(G^+) \leq S_D$ actually contains all of $S_D$. Since transpositions generate $S_D$, we only need to show that $\Phi(G^+)$ contains all transpositions of $D$.

To generate the transposition between $[e_1+e_2+e_3]$ and $[-e_1+e_2+e_3]$, consider the element $g \in G$ that negates $e_1$ and transposes $e_2$ and $e_3$, i.e. $g = \widehat{(23)}\widehat{e_1}$. Evidently 
\[ g [e_1+e_2+e_3] = [-e_1+e_2+e_3] \]
and vice versa, while 
\[ g [e_1-e_2+e_3] = [-e_1-e_3+e_2] = [e_1-e_2+e_3], \quad g [e_1+e_2-e_3] = [-e_1-e_2+e_3] = [e_1+e_2-e_3]. \]
Thus $g$ indeed induces the transposition we are looking for. The same argument applies to all transpositions between $[e_1+e_2+e_3]$ and another element.

For the remaining three transpositions, we follow a slightly different construction. In order to transpose e.g. $[e_1-e_2+e_3]$ and $[e_1+e_2-e_3]$, we likewise swap $e_2$ and $e_3$, but now negating all basis vectors, i.e. $g = \widehat{(23)}\widehat{(1,1,1)}$. This group element acts as follows:
\[ g[e_1+e_2+e_3] = [-e_1-e_3-e_2] = [e_1+e_2+e_3], \]
\[ g[-e_1+e_2+e_3] = [e_1-e_3-e_2] = [-e_1+e_2+e_3], \]
\[ g[e_1-e_2+e_3] = [-e_1+e_3-e_2] = [e_1+e_2-e_3], \]
\[ g[e_1+e_2-e_3] = [-e_1-e_3+e_2] = [e_1-e_2+e_3]. \]
Hence $g$ is exactly the transposition we are looking for.

We can interpret this construction as follows: $S_4=K_4\rtimes S_3$, where $S_3$ permutes the last three elements of $D$ while fixing the first, and $K_4$ consists of the identity and the three double transpositions. The $S_3$ subgroup is generated by the three transpositions constructed above, while $K_4$ is generated by changes of sign of two coordinates.

\subsection{The Voronoi diagram and graph structure of the BCC lattice}

The \emph{Voronoi diagram} of the BCC lattice $\sL$ is comprised of translations $l + V(\sL)$ of a single cell $V(\sL)$, which geometrically is a \emph{truncated octahedron}, having 6 square and 8 regular hexagonal faces. The induced tiling is also known as the \emph{bitruncated cubic honeycomb}.

We say that two distinct points $l, l' \in \sL$ are \emph{neighbors}, denoted by $l \sim l'$, if their Voronoi cells intersect. By the symmetries of the Voronoi diagram, to be explored in more detail in the next section, this relation is invariant under translations by any fixed $l_0 \in \sL$, and by the action of $G$: $gl \sim gl'$ iff $l \sim l'$. 

This graph structure on $\sL$ is a \emph{Cayley graph}, induced by a generating set $N_0$ so that $l \sim l'$ iff $l' - l \in N_0$. The generating set $N_0$ is comprised of the 14 neighbors of the origin, namely
\[
(\pm 2,0,0),\ (0,\pm 2,0),\ (0,0,\pm 2),\ (\pm 1,\pm 1,\pm 1).
\]
Observe that $N_0$ is invariant under the action of $G$, and naturally splits into two $G$-orbits:
\begin{itemize}
\item the \emph{axial orbit} $N_{ax}$ (6 points): $P=\pm 2e_i$;
\item the \emph{diagonal orbit} $N_{diag}$ (8 points): $Q=\pm e_1\pm e_2\pm e_3$.
\end{itemize}
These two families cannot mix under $G$ because they have different lengths, and $G\subset O(3)$ preserves lengths.
Later, this orbit decomposition will allow us to treat all axial neighbors at once, and all diagonal neighbors at once.

\begin{figure}[H]
\centering
\includegraphics[width=0.5\textwidth]{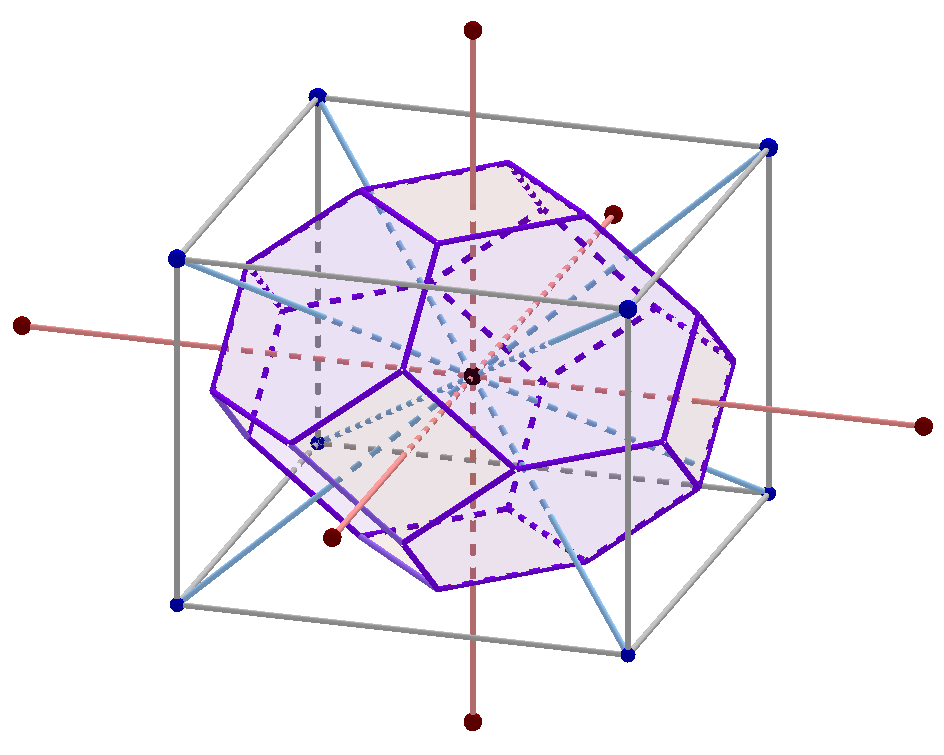}
\caption{Voronoi cell of the BCC lattice around the origin (black), elements of $N_{ax}$ (dark red) and $N_{diag}$ (dark blue) associated respectively to square and hexagonal faces.}
\end{figure}

\subsubsection{Vertices and local frames}

The Voronoi cell $V(\sL)$ is a \emph{permutohedron}: the 24 vertices of $V(\sL)$ can be uniquely identified with $G^+ \cong S_4$. 

Explicitly, each $g \in G^+$ defines a positively oriented orthonormal basis $(f_1, f_2, f_3)$ with each $f_i \in \{\pm e_1, \pm e_2, \pm e_3\}$, which is associated to the vertex $v = f_1 + \frac12 f_2$, lying at the intersection of the one square face and two hexagonal faces corresponding to the neighbors
\[
P = 2 f_1, \qquad Q_1 = f_1+f_2-f_3, \qquad Q_2 = f_1+f_2+f_3.
\]
In other words, fixing one canonical vertex and acting by $G^{+}$ generates all vertices: $V(\sL)$ is vertex-transitive along the action of $G^{+}$. Two vertices $v$ and $v'$ are adjacent iff they share two of the three neighbors $P, Q_1, Q_2$.

In particular, all vertices have the same distance to the origin and all edges the same length:
\begin{equation}\label{eq:bccvertexedge} |v|^2 = 1^2 + \left(\frac12\right)^2 + 0^2 = \frac54; \quad |v-v'|^2 = \left(\frac12\right)^2 + \left(\frac12\right)^2 = \frac12. \end{equation}

\begin{figure}[H]
\centering
\includegraphics[width=0.4\textwidth]{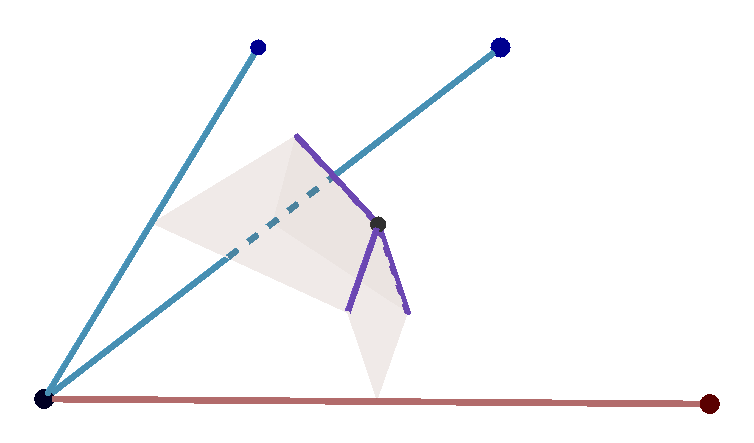}
\caption{A vertex of the BCC Voronoi cell, identified by the intersection of one square and two hexagonal faces.}
\end{figure}

\subsubsection{Faces and their symmetries}

Each face of the Voronoi cell $V(\sL)$ can be identified with the neighboring element of $N_0$ whose cell it intersects. The above orbits then correspond respectively to the square and hexagonal faces of the Voronoi cell.

The neighborhood relation also restricts to a graph structure on $N_0$: given $l, l' \in N_0$, $l \sim l'$ iff the faces adjacent to $l$ and $l'$ share an edge. This graph corresponds to the vertices and edges of a \emph{tetrakis hexahedron}, the dual polyhedron to the Voronoi cell $V(\sL)$. This edge relation on $N_0$ is evidently also conserved by $G$, and will be very important in our subsequent analysis.

\begin{figure}[H]
\centering
\includegraphics[width=0.4\textwidth]{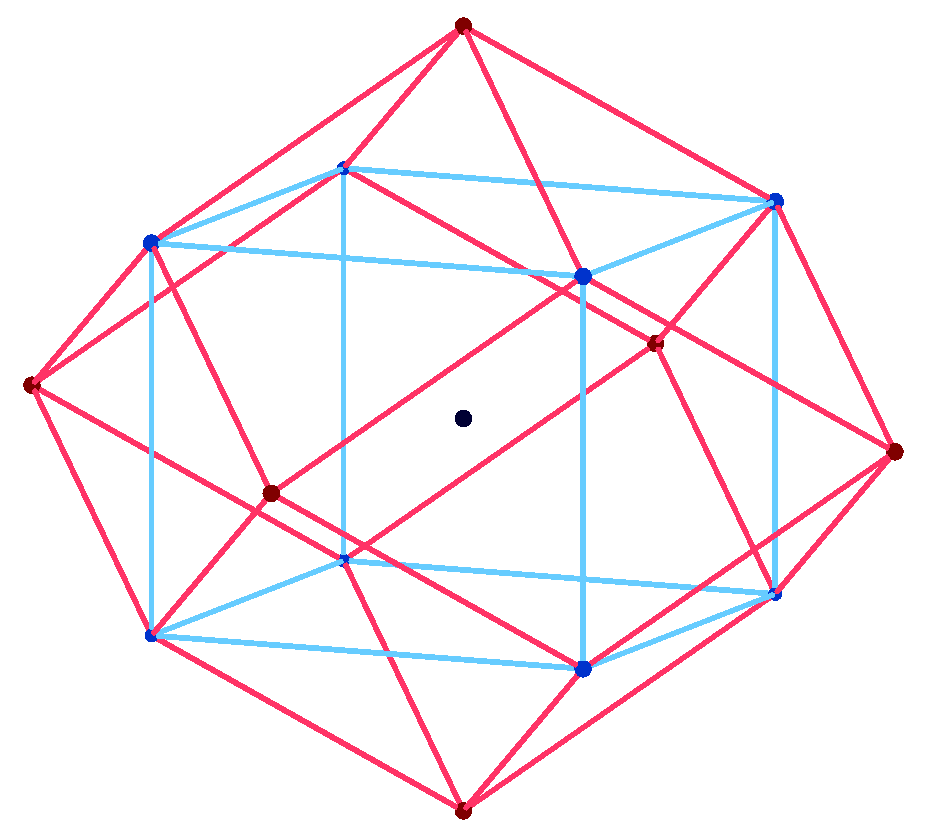}
\caption{Neighbors of the origin (black) and their adjacency graph: $N_{ax}$ colored dark red, $N_{diag}$ colored dark blue, edges $N_{ax} \sim N_{diag}$ colored light red, $N_{diag} \sim N_{diag}$ colored light blue. \\
Note that the actual polyhedron spanned by these vertices is a \emph{rhombic dodecahedron}, the Voronoi cell of the FCC lattice: every two triangular faces sharing a light blue edge are actually coplanar, forming a single rhombus. }
\end{figure}

Given $l \in N_{ax}$, its neighbors within $N_0$ are the four points $l' \in N_{diag}$ satisfying $l' \cdot l > 0$. For example, the neighbors of $l = (2,0,0)$ are
\[ (1,1,1), (1,-1,1), (1,-1,-1), (1,1,-1). \]
For $l \in N_{diag}$, its neighbors within $N_0$ are the points $l' \in N_0\setminus\{l\}$ satisfying $l' \cdot l > 0$: three belong to $N_{ax}$ and three to $N_{diag}$. For example, the neighbors of $l = (1,1,1)$ are
\[ (2,0,0),(1,1,-1),(0,2,0),(-1,1,1),(0,0,2),(1,-1,1). \]

Consider now the stabilizer $G_l \leq G$ of each $l \in N_0$. In light of the above orbit decomposition of $N_0$, the orbit-stabilizer theorem implies that $G_l$ has 8 elements for $l \in N_{ax}$, and 6 elements for $l \in N_{diag}$, and 4 resp.~3 orientation-preserving ones. These elements are given as follows:
\begin{itemize}
\item $l \in N_{ax}$, say $l = \pm 2 e_i$: $G_l^+$ is generated by the rotation that sends $e_j$ to $e_k$ and $e_k$ to $-e_j$, and $G_l$ includes four more reflections.
\item $l \in N_{diag}$, say $l = \sum_{i=1}^3 (-1)^{a_i} e_i$: $G_l$ is generated by permutations of $(-1)^{a_i} e_i$, and $G_l^+$ consists of the even permutations i.e. cyclic rotations.
\end{itemize}
That is, $G_l^+$ consists of rotations around a common axis in the direction of $l$, 90 degree rotations for $l \in N_{ax}$, 120 degree rotations for $l \in N_{diag}$.

Note that for $l \in N_{ax}$, $G_l$ is the full symmetry group of the square face, i.e. the dihedral group $D_4$, but the same is not true for $l \in N_{diag}$. In fact $G_l \cong D_3 = S_3$: the three square neighbors of a hexagonal face form a triangle, which can be permuted at will, but given any two square faces, there is only one hexagonal face that can lie between them so the hexagonal faces are completely fixed by such a permutation.

Both types of faces possess in particular the symmetry $l' \sim l \iff l-l' \sim l$.

\subsection{Conjugate action of $G$ on matrices}

It will be useful for us to consider the left action of $G$ on $3 \times 3$-matrices by conjugation:
\[ {}^g M := g M g^{-1} = g M g^\top, \quad M \in \R^{3 \times 3}, \ g \in G, \]
where we note that $G < O(3)$. Note also that since $(-g) M (-g)^\top = g M g^\top$, this reduces to an action of $G^+ \cong S_4$.

Evidently this action is linear, i.e.~a representation of $G$. $G$ acts on each basis element $e_{ij} = e_i e_j^\top$ of $\R^{3 \times 3}$ as follows: 
\begin{itemize}
\item $S_3$: $\widehat{(ij)} \ \colon \  e_{ii} \leftrightarrow e_{jj}, \ e_{kk} \mapsto e_{kk}; \quad e_{ij} \leftrightarrow e_{ji}, \ e_{ik} \leftrightarrow e_{jk}, \ e_{ki} \leftrightarrow e_{kj}$
\item $\Z_2^3$: $\widehat{e_i} \ \colon \ e_{ii} \mapsto e_{ii}, \ e_{jj} \mapsto e_{jj}; \quad e_{ij} \mapsto - e_{ij}, \ e_{ji} \mapsto - e_{ji}, \ e_{jk} \mapsto e_{jk}$
\end{itemize}
From this, we immediately observe that $G$ acts separately on diagonal and off-diagonal basis elements. 

On diagonal entries, the action of $\Z_2^3$ is trivial while $S_3$ acts by permutations $\hat \sigma \colon e_{ii} \mapsto e_{\sigma_i \sigma_i}$. 

On off-diagonal entries, the action of $\Z_2^3$ (as well as transpositions $(ij)$) fixes or negates $e_{ij} \pm e_{ji}$, while the remaining elements of $S_3$ permute their indices.

We can then conclude the existence of four orthogonal $G$-invariant subspaces:
\begin{itemize}
\item $W_1 = \Span\{I\} = \Fix(G)$, which is fixed by the group action;
\item $W_2 = \Span\{e_{11} - e_{22}, e_{22} - e_{33}\}$, the space of traceless diagonal matrices;
\item $W_3 = \Span\{e_{12} + e_{21}, e_{13} + e_{31}, e_{23} + e_{32}\}$, the space of symmetric diagonal-free matrices;
\item $W_4 = \Span\{e_{12} - e_{21}, e_{13} - e_{31}, e_{23} - e_{32}\}$, the space of antisymmetric diagonal-free matrices.
\end{itemize}
We will later show that each of these subspaces correspond to irreducible subrepresentations of $G$, by characterizing the $G$-invariant quadratic forms on each $W_i$. 

For now, we note the following: given $l \in N_0$, the rank-one matrix $l l^\top$ lies in $W_1 \oplus W_2$ for $l \in N_{ax}$ and in $W_1 \oplus W_3$ for $l \in N_{diag}$. For example, for $l = 2 e_1 \in N_{ax}$, we have
\[ l l^\top = \left[ \begin{matrix}
    4 & 0 & 0 \\ 0 & 0 & 0 \\ 0 & 0 & 0
\end{matrix} \right] = \frac43 I + \left[ \begin{matrix}
    \frac83 & 0 & 0 \\ 0 & -\frac43 & 0 \\ 0 & 0 & -\frac43
\end{matrix} \right] \in \frac43 I + W_2, \]
and for $l = e_1+e_2+e_3 \in N_{diag}$, we have
\[ l l^\top = \left[ \begin{matrix}
    1 & 1 & 1 \\ 1 & 1 & 1 \\ 1 & 1 & 1
\end{matrix} \right] = I + \left[ \begin{matrix}
    0 & 1 & 1 \\ 1 & 0 & 1 \\ 1 & 1 & 0
\end{matrix} \right] \in I + W_3. \]
The $I$ term of each $ll^\top$ will in fact coincide with $\frac{\abs{l}^2}{3} I$, as can be seen by taking inner products with $I$. This correspondence between invariant subspaces and the two orbits of neighbors will also appear at various points in later calculations.

\subsubsection{Subspaces fixed by stabilizers}

Let now $l \in N_0$ and consider the stabilizer $G_l$. It will be useful to characterize the subspace $\Fix(G_l)$ of matrices fixed by all elements of $G_l$. For this, we use the above direct sum decomposition of $\R^{3\times3}$:
\[ \Fix(G_l) = \Fix(G) \oplus (\Fix(G_l) \cap W_2) \oplus (\Fix(G_l) \cap W_3) \oplus (\Fix(G_l) \cap W_4). \]

\begin{lemma}\label{lem:stabfix}
For any $l \in N_0$, $\Fix(G_l) = \Span\{I, l l^\top\}$.
\end{lemma}
\begin{proof}
We handle the axial and diagonal orbits separately.
\begin{itemize}
    \item $l \in N_{ax}$, say $l = \pm 2 e_i$. Fixing an ordered basis $(e_i,e_j,e_k)$, note that the invariant subspace decomposition (in particular, the distinction between diagonal and zero-diagonal matrices) is preserved since the change of basis is induced by an element of $G$, so suppose wlog that $(i,j,k) = (1,2,3)$. 

    Suppose $M \in \R^{3\times3}$ such that $gMg^\top = M$ for all $g \in G_l$. 

    Taking the element $g \in G_l$ that sends $(e_1,e_2,e_3)$ to $(e_1,e_3,e_2)$, we obtain
    \[ M_{12} = (gMg^\top)_{12} = M_{13}; \quad M_{21} = (gMg^\top)_{21} = M_{31}; \quad M_{23} = (gMg^\top)_{23} = M_{32}. \]
    Likewise, $M_{33} = (gMg^\top)_{33} = M_{22}$.
    On the other hand, taking the element $g \in G_l$ that sends $(e_1,e_2,e_3)$ to $(e_1,e_3,-e_2)$, we obtain
    \[ M_{12} = (gMg^\top)_{12} = -M_{13}; \quad M_{21} = (gMg^\top)_{21} = -M_{31}; \quad M_{23} = (gMg^\top)_{23} = -M_{32}. \]
    These imply that all the off-diagonal entries of $M$ must vanish, and $M_{22} = M_{33}$. Hence 
    \[ M = \diag(a,b,b) = \frac{a-b}{4} l l^\top + b I \]
    for $a, b \in \R$ arbitrary.
    \item $l \in N_{diag}$: fixing the ordered basis $(l_1 e_1, l_2 e_2, l_3 e_3)$, again the invariant subspace decomposition is preserved. Under this basis, $l = (1,1,1)$.
    
    Suppose $M \in \R^{3\times3}$ such that $gMg^\top = M$ for all $g \in G_l$. Since $G_l$ consists of all permutations of this basis, any two diagonal entries can be mapped to each other, and likewise any two off-diagonal entries. That is,
    \[ M = \left[ \begin{matrix}
        a & b & b \\ b & a & b \\ b & b & a
    \end{matrix}\right] = b l l^\top + (a-b) I. \]
\end{itemize}
In both cases, we have shown that $\Fix(G_l) \subseteq \Span\{I, l l^\top\}$. The converse direction is straightforward: $I \in \Fix(G)$ and $g (l l^\top) g^\top = (g l) (g l)^\top = l l^\top$ for all $g \in G_l$. 
\end{proof}

\subsubsection{Invariant forms}\label{sect:invforms}

We say that a matrix functional $f$ on $\R^{3\times3}$ is $G$-invariant if
\[ f(g M g^\top) = f(M) \quad \text{for all} \quad g \in G, \ M \in \R^{3\times3}. \]
Such functionals will arise naturally in the perturbative analysis of the cost $F(M)$. In this section, we will characterize all linear and quadratic $G$-invariant functionals on matrices, and also evaluate the following particular forms, which will appear in later calculations:
\begin{equation}\label{eq:orbitavgs}
f_{l_0}(M) := \frac{1}{\abs{G \cdot l_0}} \sum_{l \in G \cdot l_0} \frac{1}{|l|^2} l^\top M l, \quad q_{l_0}(M) := \frac{1}{\abs{G \cdot l_0}} \sum_{l \in G \cdot l_0} \frac{1}{|l|^4} (l^\top M l)^2, \quad M \in \R^{3\times3}, 
\end{equation}
where the averages range over an orbit $G \cdot l_0 \subset N_0$. For $l_0 \in N_{ax}$ resp.~$N_{diag}$, we write $f_{ax}$ resp.~$f_{diag}$ in place of $f_{l_0}$, and likewise for $q$. 

By construction, $f_{l_0}$ resp.~$q_{l_0}$ are linear resp.~quadratic forms on matrices, and both forms are invariant under the conjugate action of $G$ on $\R^{3\times3}$:
\begin{multline*} f_{l_0}(g M g^\top) = \frac{1}{\abs{G \cdot l_0}} \sum_{l \in G \cdot l_0} \frac{1}{|l|^2} l^\top g M g^\top l \\
= \frac{1}{\abs{G \cdot l_0}} \sum_{l \in G \cdot l_0} \frac{1}{|g^{-1} l|^2} (g^{-1} l)^\top M (g^{-1} l) = \frac{1}{\abs{G \cdot l_0}} \sum_{l' \in G \cdot l_0} \frac{1}{|l'|^2}  l'^\top M l' = f_{l_0}(M), \end{multline*}
since the action of any $g^{-1} \in G \subset O(3)$ preserves orbits and lengths, and likewise for $q_{l_0}$.

\paragraph{Linear functionals.} 

\begin{lemma}[Characterization of invariant linear matrix functionals]\label{lem:invlinforms}
Any $G$-invariant linear functional $f$ on $\R^{3\times3}$ is given by $f(M) = c \Tr(M)$ for some $c \in \R$.

Given $l \in N_0$, any $G_l$-invariant linear functional $f$ on $\R^{3\times3}$ is given instead by
\[ f(M) = c \Tr(M) + d l^\top Ml \] 
for some $c, d \in \R$.
\end{lemma}
\begin{proof}
By the usual orthogonal structure on matrices, $f$ is given by $f(M) = \Tr(A^\top M)$ for some $A \in \R^{3\times3}$. The assumption of $G$-invariance implies that for arbitrary $g \in G$ and $M \in \R^{3\times3}$,
\[ \Tr(A^\top M) = f(M) = f(g M g^\top) = \Tr(A^\top g M g^\top) = \Tr(g^\top A^\top g M) = \Tr((g^\top A g)^\top M). \]
Therefore $g^\top A g = A$ for all $g \in G$; that is, $A$ is a fixed point of the action of $G$. Since $Fix(G) = \Span\{I\}$, it follows that $A = cI$ for some $c \in \R$.

Likewise, if $f$ is only $G_l$-invariant, we instead have $A \in Fix(G_l) = \Span\{I, l l^\top\}$, which yields the second statement.
\end{proof}

We obtain in particular that $f_{l_0}(M) = c \Tr(M)$ for some $c \in \R$. To evaluate $c$, we simply set $M = I$:
\[ 3 c = c \Tr(I) = f_{l_0}(I) = \frac{1}{\abs{G \cdot l_0}}  \sum_{l \in G \cdot l_0} \frac{1}{|l|^2}  l \cdot I l =  \frac{1}{\abs{G \cdot l_0}} \sum_{l \in G \cdot l_0} 1 = 1. \]
We therefore conclude that
\begin{equation}\label{eq:faxdiag} f_{ax}(M) = f_{diag}(M) = \frac13 \Tr(M). \end{equation}

\paragraph{Quadratic forms: invariant subspace decomposition.}

We now show that the only $G$-invariant quadratic forms on each $W_i$ are multiples of the squared Frobenius norm:
\begin{itemize}
\item $W_1$: trivial by equivalence of all quadratic forms on one-dimensional vector spaces.
\item $W_2$: representing each $M \in W_2$ as $M = \diag(a,b,-a-b)$, quadratic forms on $W_2$ will be given by
\[ q(M) = \alpha a^2 + \beta b^2 + 2 \gamma a b. \]
The action of $(12)$ permutes $a$ and $b$, while $(13)$ resp.~$(23)$ induce the action $(a,b) \mapsto (-a-b,b)$ resp.~$(a,b) \mapsto (a,-a-b)$. Setting $a = 0$, $b = 1$ gives $\alpha = \beta$, and setting $a = -b = 1$ gives $\alpha - 2 \gamma = 0$, so that
\[ q(M) = \alpha (a^2 + b^2 + ab ) = \frac\alpha2 ( a^2 + b^2 + (-a-b)^2) = \frac\alpha2 \norm{M}^2. \]
\item $W_3$: representing each $M \in W_3$ as 
\[ M = a (e_{12}+e_{21}) + b (e_{13}+e_{31}) + c(e_{23}+e_{32}) = \left[ \begin{matrix} 0 & a & b \\ a & 0 & c \\ b & c & 0 \end{matrix} \right], \] 
quadratic forms on $W_2$ will be given by
\[ q(M) = \alpha a^2 + \beta b^2 + \gamma c^2 + 2 \delta a b + 2 \varepsilon a c + 2 \zeta b c. \]
Considering the action of double negations (e.g.~$(1,1,0)$, which fixes $a$ and negates $b$ and $c$) shows that $\delta=\varepsilon=\zeta=0$.
Considering the action of 3-cycles then yields $\alpha=\beta=\gamma$. Therefore
\[ q(M) = \alpha (a^2+b^2+c^2) = \frac\alpha2 \norm{M}^2. \]
\item $W_4$: similar argument to $W_3$.
\end{itemize}
This implies also that each $W_i$ are irreducible subrepresentations of $G$.
To obtain the characterization on the whole space $\R^{3\times3}$, we must also check that there are no mixed terms between different $W_i$. Let $B$ be the symmetric bilinear form associated with $q$, so that
\[ q(M+N)=q(M)+2B(M,N)+q(N). \]
The $G$-invariance of $q$ implies, for every $g\in G$,
\[ B(gMg^\top,gNg^\top)=B(M,N). \]
In particular, whenever a symmetry fixes $M$ and sends $N$ to $-N$, we obtain $B(M,N)=-B(M,N)$, hence $B(M,N)=0$. We apply this observation as follows:
\begin{itemize}
\item \textit{Between $W_1$ and $W_2$:} exchanging coordinates $i$ and $j$ fixes $I$ and changes the sign of $e_{ii}-e_{jj}$. Since these differences span $W_2$, all such mixed terms vanish.
\item \textit{Between diagonal and off-diagonal matrices:} changing the sign of coordinate $i$ fixes every diagonal matrix and changes the sign of $e_{ij}+e_{ji}$ and $e_{ij}-e_{ji}$, for $j\ne i$. Thus there are no mixed terms between $W_1\oplus W_2$ and $W_3\oplus W_4$.
\item \textit{Between $W_3$ and $W_4$:} consider basis elements $S_{ij}=e_{ij}+e_{ji}$ and $A_{kl}=e_{kl}-e_{lk}$, with $i<j$ and $k<l$. If $\{i,j\}\ne\{k,l\}$, changing the sign of a coordinate in $\{k,l\}\setminus\{i,j\}$ fixes $S_{ij}$ and changes the sign of $A_{kl}$. If the two pairs coincide, exchanging coordinates $i$ and $j$ fixes $S_{ij}$ and changes the sign of $A_{ij}$. Hence these mixed terms also vanish.
\end{itemize}
Consequently, writing $M=\sum_{i=1}^4 M_i$ with $M_i\in W_i$, we have $q(M)=\sum_{i=1}^4q(M_i)$.

From this, we characterize all $G$-invariant forms on the entire space $\R^{3\times3}$:

\begin{lemma}[Characterization of invariant quadratic forms]\label{lem:invquadforms}
Any $G$-invariant quadratic form $q$ on $\R^{3\times3}$ is given by
\[ q(M) = \sum_{i=1}^4 \alpha_i \norm{\pi_{W_i} M}^2, \quad M \in \R^{3 \times 3}, \]
for some $\alpha_1, \alpha_2, \alpha_3, \alpha_4 \in \R$.
\end{lemma}

Observe also that since $W_1 = \Span\{I\}$,
\[ \norm{\pi_{W_1} M}^2 = \frac{\ip{I}{M}^2}{\norm{I}^2} = \frac{\Tr(I^\top M)^2}{\Tr(I^\top I)} = \frac13 \Tr(M)^2. \]
For ease of notation and computation, we will also work with the subspace $W_d := W_1 \oplus W_2$ of diagonal matrices, with norm
\[ \norm{M_d}^2 = \norm{M_1}^2 + \norm{M_2}^2 = \frac13 \Tr(M)^2 + \norm{M_2}^2; \qquad M = \underbrace{M_1}_{\in W_1} + \underbrace{M_2}_{\in W_2} \in W_d. \]

Now for the particular quadratic forms $q_{l_0}$, we likewise obtain
\[ q_{l_0}(M) = \alpha_1 \Tr(M)^2 + \alpha_2 \norm{M_2}^2 + \alpha_3 \norm{M_3}^2 + \alpha_4 \norm{M_4}^2, \qquad M_i = \pi_{W_i} M \in W_i. \]

First observe that
\[ 9 \alpha_1 = q_{l_0}(I) = \frac{1}{\abs{G \cdot l_0}} \sum_{l \in G \cdot l_0} \underbrace{\frac{1}{|l|^4} (l^\top I l)^2}_{1} = 1. \]
Therefore $\alpha_1 = \frac19$ for $q_{ax}$ as well as $q_{diag}$.

Observe also that $\alpha_4 = 0$, since $l^\top M l$ is identically zero for $M$ antisymmetric, and each summand vanishes.

For the axial orbit, since each $l l^\top$ for $l \in N_{ax}$ is a diagonal matrix and $l^\top M l = \ip{l l^\top}{M}$, we obtain that $\alpha_3 = 0$ as well. That is,
\[ q_{ax}(M) = \frac19 \Tr(M)^2 + \alpha_2 \norm{M_2}^2. \]
To evaluate $\alpha_2$, take $M = l l^\top$ for some $l \in N_{ax}$. 
Since $l = \pm 2 e_i$, $M$ is a diagonal matrix hence
\[ \norm{M_2}^2 = \norm{M}^2 - \norm{M_1}^2 = \Tr(l l^\top l l^\top) - \frac13 \Tr(l l^\top)^2 = \frac23 |l|^4 = \frac{32}{3}. \]
Meanwhile,
\[ q_{ax}(M) = \frac{1}{\abs{N_{ax}}} \sum_{l' \in N_{ax}} \frac{1}{|l'|^4} (\underbrace{l'^\top M l'}_{(l \cdot l')^2})^2 = \frac{1}{\abs{N_{ax}}} \sum_{l' \in N_{ax}} \frac{(l \cdot l')^4}{|l'|^4}. \]
Since all the elements of $N_{ax}$ apart from $\pm l$ are orthogonal to $l$, we obtain
\[ q_{ax}(M) = \frac26 {|l|^4} = \frac{16}{3} \overset{!}{=} \frac19 \underbrace{\Tr(M)^2}_{|l|^4 = 16} + \alpha_2 \underbrace{\norm{M_2}^2}_{\frac{32}{3}} = \frac{16}{9} \left( 1 + 6 \alpha_2 \right). \]
Therefore $\alpha_2 = \frac13$ and
\begin{equation}\label{eq:qax} q_{ax}(M) = \frac19 \Tr(M)^2 + \frac13 \norm{M_2}^2 = \frac13 \norm{M_d}^2. \end{equation}

Likewise, for the diagonal orbit, we have
\[ q_{diag}(M) = \frac19 \Tr(M)^2 + \alpha_2 \norm{M_2}^2 + \alpha_3 \norm{M_3}^2. \]
Now observe that $\alpha_2$ must vanish instead of $\alpha_3$: since every $l \in N_{diag}$ is of the form $(\pm1, \pm1, \pm1)$, the diagonals of the matrix $l l^\top$ are identically $1$ (indeed $l l^\top \in I + W_3$), hence $l l^\top$ is orthogonal to $W_2$. Then each $l^\top M l = \ip{l l^\top}{M}$ vanishes identically on $W_2$, so $q_{diag}$ vanishes as well.

Now to calculate $\alpha_3$, we test instead against $M = l l^\top$ for some $l \in N_{diag}$. Since every $l' \in N_{diag}$ is of the form $(\pm1,\pm1,\pm1)$, we have the following cases:
\begin{itemize}
    \item $l' = \pm l$, in which case $l' \cdot l = \pm |l|^2 = \pm 3$;
    \item $l'$ and $l$ have the same sign on two components and opposite signs on one, in which case $l' \cdot l = 1 + 1 - 1 = 1$;
    \item $l'$ and $l$ have the same sign on one component and opposite signs on two, in which case $l' \cdot l = 1 - 1 - 1 = -1$.
\end{itemize}
Therefore
\[ q_{diag}(M) = \frac{1}{\abs{N_{diag}}} \sum_{l' \in N_{diag}} \frac{(l \cdot l')^4}{|l'|^4} = \frac18 \left( 2 \cdot \frac{(\pm3)^4}{3^2} + 6 \cdot \frac{(\pm1)^4}{3^2} \right) =\frac18 \left( 18 + \frac23 \right) = \frac{7}{3}. \]
On the other hand, the zero-diagonal symmetric component $M_3$ will have six off-diagonal entries, each of which are $\pm 1$, so that $\norm{M_3}^2 = 6$. This gives us
\[ \frac{7}{3} = q_{diag}(M) \overset{!}{=} \frac19 \underbrace{\Tr(M)^2}_{|l|^4 = 9} + \alpha_3 \underbrace{\norm{M_3}^2}_{6} = 1 + 6 \alpha_3. \]
This yields $\alpha_3 = \frac29$ thus
\begin{equation}\label{eq:qdiag} q_{diag}(M) = \frac19 \Tr(M)^2 + \frac29 \norm{M_3}^2. \end{equation}

\section{Voronoi diagrams and their properties}\label{sect:moments}

\subsection{Voronoi diagrams, faces, pyramids}\label{sect:voronoi}

We first set the general formalism for Voronoi cells and associate to them various quantities which we will rely on in our analysis. Without loss of generality, we will consider Voronoi cells centered at the origin. We refer to Sections 1--3 of \cite{quantbook} for general background.

Let $S = \{x_i\}_{i \in I} \subset \R^3$ be a locally finite collection of points containing the origin: say, $0 \in I$ and $x_0 = 0$. 

In our case, the indexing set $I$ will be the BCC lattice $\sL$. 
 When $S$ is given by a linear deformation of $\sL$, i.e. $S = M \sL$ for a matrix $M \in \R^{3\times3}$, we will use $M \sL$ and $M$ interchangeably. 

We denote the \emph{Voronoi cell} of the origin with respect to $S$ by
\[ V(S) := \{ x \in \R^3 \mid \abs{x} = \min_{x_i \in S} \abs{x-x_i} \}. \]
Note that the minimum is always attained by the assumption of local finiteness. The Voronoi cell associated to any other $x_i \in S \setminus \{0\}$ is given by
\[ V_i(S) := x_i + V(S-x_i) = \{ x \in \R^3 \mid \abs{x-x_i} = \min_{x_j \in S} \abs{x-x_j} \}. \]
That is,
\[ V_i(S) - x_i = V(S-x_i) = \{ y \in \R^3 \mid \abs{y} = \min_{x_j \in S} \abs{(y+x_i)-x_j} \}. \]
We say that two distinct elements $x_i, x_j \in S$ are \emph{neighbors} if their Voronoi cells intersect, in which case we write $x_i \sim x_j$ or simply $i \sim j$. We denote the set of neighbors of the origin by
\[ N_0(S) := \{ i \in I \mid i \sim 0 \}. \]
The elements of $N_0(S)$ are exactly the indices for which the minimum in the definition of $V(S)$ is attained by $x_i$ for some $x \in V(S)$, i.e. the \emph{face}
\[ F_i(S) := \{ x \in \R^3 \mid \abs{x} = \abs{x-x_i} = \min_{x_j \in S} \abs{x-x_j} \} = V(S) \cap V_i(S) \]
is nonempty. Consequently, $V(S) = V(\{x_i\}_{i \in N_0(S) \cup \{0\}})$, and the boundary of $V(S)$ decomposes as
\[ \partial V(S) = \bigcup_{i \in N_0(S)} F_i(S). \]
We also define the pyramids 
\[ \Delta_i(S) := \pyr(F_i(S)), \quad \pyr(A) := \conv(\{0\} \cup A) = \{t a \mid t \in [0,1], \, a \in A\}, \]
so that
\[ V(S) = \pyr(\partial V(S)) = \bigcup_{i \in N_0(S)} \Delta_i(S). \]
Note also that generically, the above unions are disjoint up to measure zero.

We will also consider the \emph{interfaces}
\[ E_{ij}(S) := \Delta_i(S) \cap \Delta_j(S) = \pyr\bigl(F_i(S) \cap F_j(S))\bigr), \quad i, j \in N_0(S). \]
Note that $F_i(S) \cap F_j(S)$ is only nonempty for $i \sim j$, since it lies in the intersection of the Voronoi cells of $x_i$ and $x_j$. 

\begin{figure}[H]
\centering
\includegraphics[width=0.4\textwidth]{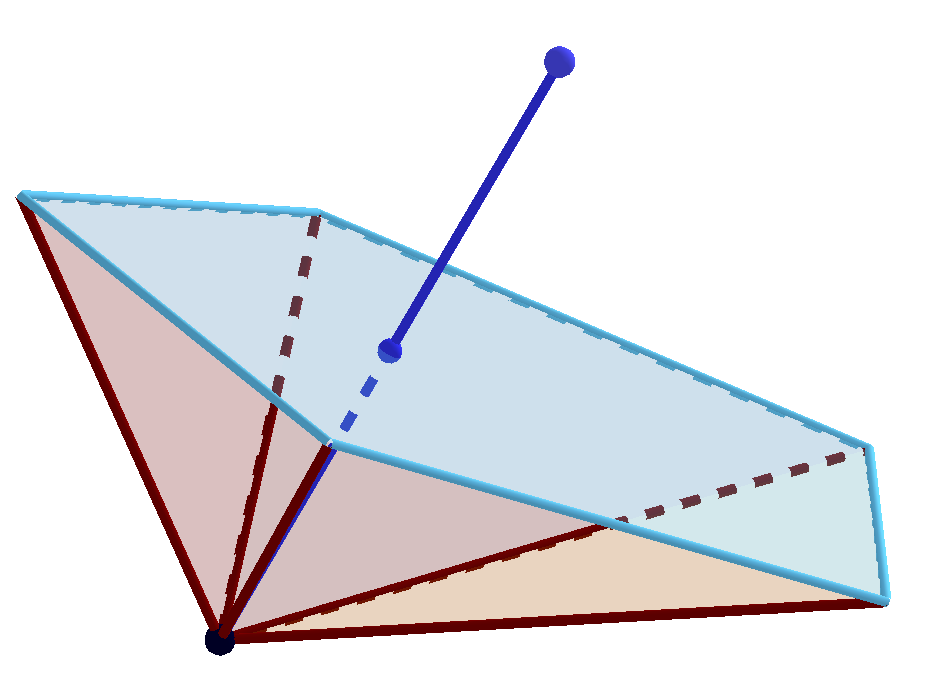}
\caption{Pyramid $\Delta_i(S)$ (brown) associated to neighbor $x_i$ (blue) of origin $O$ (black), with face $F_i(S)$ highlighted in blue and one interface $E_{ij}(S)$ in orange.}
\end{figure}

\subsubsection{Affine characterization by weights}

Observe that
\begin{align*}
\abs{x} \leq \abs{x-x_i} \quad \iff \quad 0 \mathrel{\geq} \abs{x}^2 - \abs{x-x_i}^2 = x_i \cdot (2x-x_i)
\quad \iff \quad 2 x_i \cdot x \leq \abs{x_i}^2.
\end{align*}
We can then define the \emph{weight function}
\[ w(l) := \left\{ \begin{matrix} \frac{2l}{\abs{l}^2}, & l \neq 0 \\ 0, & l = 0 \end{matrix} \right., \qquad l \in \R^3,  \]
so that the above sets can be characterized equivalently as follows:
\begin{align*}
    V(S) & = \bigl\{x  \mid \max_{i \in I} x \cdot w(x_i) \leq 1 \bigr\}; \\
    \partial V(S) & = \bigl\{x  \mid \max_{i \in I} x \cdot w(x_i) = 1 \bigr\}; \\
    F_i(S) & = \bigl\{x  \mid x \cdot w(x_i) = \max_{j \in I} x \cdot w(x_j) = 1 \bigr\}; \\
    \Delta_i(S) & = \bigl\{x  \mid x \cdot w(x_i) = \max_{j \in I} x \cdot w(x_j) \leq 1 \bigr\}; \\
    E_{ij}(S) & = \{ x \mid x \cdot w(x_i) = x \cdot w(x_j) = \max_{k \in I} x \cdot w(x_k) \leq 1 \}.
\end{align*}
This representation will be crucial to our analysis of Voronoi cells. 

\begin{remark}
The first equation shows that $V(S)$ is the polar body associated to the set $\{w(x_i)\}_{i \in I}$, whose only accumulation point is the origin by the local finiteness of $S$. 

In other words, $x \mapsto \max_i x \cdot w(x_i)$ is the asymmetric norm whose unit ball is $V(S)$. 
\end{remark}

From this representation, we can immediately deduce the following equivariance properties:

\begin{proposition}\label{prop:regioninv}
Let $T = \alpha R$ be a similarity transformation, i.e.~$\alpha > 0$, $R \in O(3)$. Then
\[ V(TS) = T V(S); \quad F_i(TS) = T F_i(S); \quad \Delta_i(TS) = T \Delta_i(S); \quad E_{ij}(TS) = T E_{ij}(S). \]
Here $TS = T(S) = \{T x_i\}_{i \in I}$.
\end{proposition}
\begin{proof}
    Follows directly from the above weight representation, with the observation that
    \[ w(Tl) = \frac{2}{\abs{\alpha R l}^2} \alpha R l = \alpha^{-1} R w(l) = T^{-\top} w(l) 
    \quad \implies \quad (Tx) \cdot w(Tl) = x^\top T^\top T^{-\top} w(l) = x \cdot w(l). \qedhere \]
\end{proof}

Note also the following identity:
\begin{equation}\label{eq:wdiff} 
\abs{w(l) - w(l')}^2 = \abs{w(l)}^2 - 2 w(l) \cdot w(l') + \abs{w(l')}^2 = \frac{4}{\abs{l}^2} - 2 \frac{4 l \cdot l'}{\abs{l}^2 \abs{l'}^2} + \frac{4}{\abs{l'}^2}= 4 \frac{\abs{l-l'}^2}{\abs{l}^2 \abs{l'}^2}.
\end{equation}

We also denote by
\[ h_i = \frac{x_i}{2}, \quad i \in N_0(S) \]
the height of the pyramid $\Delta_i(S)$. Note that $h_i$ is the midpoint between $0$ and $x_i$, and 
\[ h_i \cdot w(x_i) = \frac12 \frac{2}{\abs{x_i}^2}  x_i \cdot x_i = 1. \]
Consequently,
\[ F_i(S) - h_i \subset \{ x - h_i \mid x \cdot w(x_i) = 1 \} = \{ v \mid v \cdot w(x_i) = 0 \} = w(x_i)^\perp = x_i^\perp. \]
That is, the plane spanned by $F_i(S)$ is the perpendicular bisector between $0$ and $x_i$. 

Another consequence of this is that, with respect to the Voronoi cell $V_i(S) = x_i + V(S-x_i)$ of a neighbor $i \in N_0(S)$, the pyramid in the direction of $-x_i$, pointing to the neighbor $0$, is given by
\[ \conv(\{x_i\} \cup F_i(S)) = x_i + \conv(\{0\} \cup (F_i(S)-x_i))\]
and is exactly the reflection of $\Delta_i(S)$ around $h_i=x_i/2$. We formulate this in the following sense:

\begin{lemma}\label{lem:ijreflect}
For any $i \in N_0(S)$, considering the point set $S-x_i$ indexed by itself, the pyramid associated to the neighbor $-x_i \in S - x_i$ is
\[
\Delta_{-x_i}(S- x_i) = R_{x_i} \Delta_i(S),
\]
where $R_{x_i} \in O(3)$ is the reflection matrix along $x_i$:
\[ R_{x_i} = I - \frac{2}{\abs{x_i}^2} x_i x_i^\top = I - x_i w(x_i)^\top. \]
\end{lemma}
\begin{proof}
Denote $F_i(S)$ simply by $F_i$. Observe that 
\[ R_{x_i} F_i = R_{h_i} (F_i - h_i + h_i) = R_{h_i}(F_i - h_i) - h_i = F_i - h_i - h_i = F_i - x_i \]
where we apply $F_i - h_i \bot h_i$.

But $F_i - x_i$ is exactly the face associated to the neighbor $-x_i$ of $S - x_i$:
\begin{align*} 
F_i = V(S) \cap V_i(S) & = V(S) \cap (V(S-x_i) + x_i)
\\ \implies \quad F_i - x_i = (V(S) - x_i) \cap V(S-x_i) & = V(S-x_i) \cap \bigl( V((S-x_i)+x_i)-x_i \bigr).
\end{align*}
Therefore $\Delta_{-x_i}(S-x_i) = \conv(\{0\} \cup (F_i-x_i)) = R_{x_i} \Delta_i(S)$.
\end{proof}

If two distinct pyramids $\Delta_i(S)$ and $\Delta_j(S)$ share a face of positive area, then $x_i$ and $x_j$ are neighbors.

\begin{figure}[H]
\centering
\includegraphics[width=0.4\textwidth]{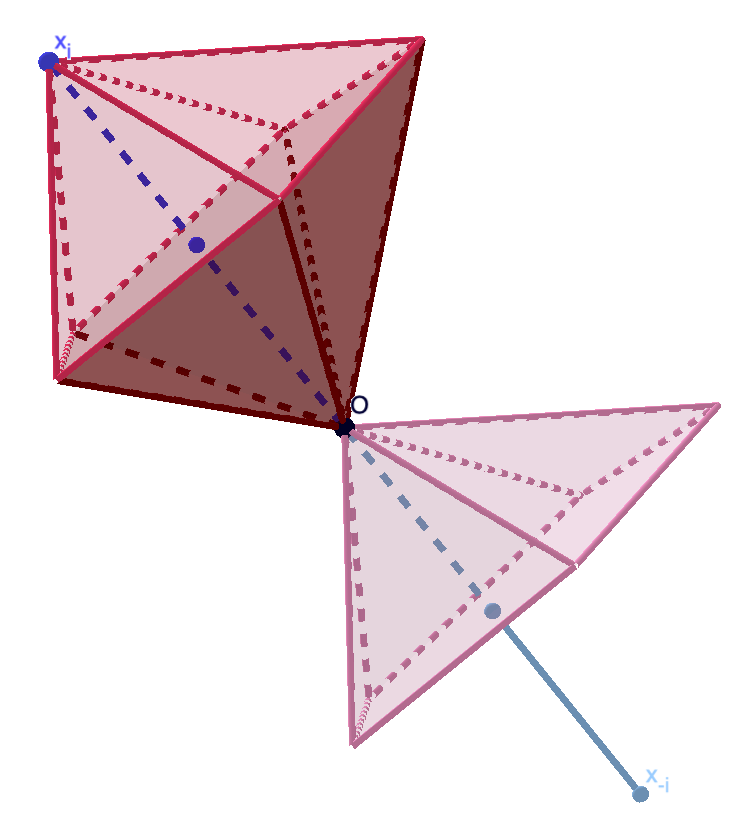}
\caption{Illustration of Lemma~\ref{lem:ijreflect}: $\Delta_i(S)$ colored brown, $x_i+\Delta_{-x_i}(S-x_i)$ colored red, $\Delta_{-x_i}(S-x_i) = R_{x_i} \Delta_i(S)$ colored pink. }
\end{figure}

Now suppose the point set is a lattice $\Lambda$. Then each Voronoi cell is simply a translate of the Voronoi cell of $0$:
\[ V_l(\Lambda) = l + V(\Lambda-l) = l + V(\Lambda), \quad l \in \Lambda. \]
By the central symmetry $-\Lambda = \Lambda$ of $\Lambda$, the Voronoi cell $V = V(\Lambda)$ is also centrally symmetric, e.g.~by the above weight representation. Lemma~\ref{lem:ijreflect} then implies the following:

\begin{lemma}\label{lem:latcell}
Let $\Lambda$ be a lattice, indexed by itself. Then for every neighbor $l \in N_0(\Lambda)$, 
\[ \Delta_{-l}(\Lambda) = R_l \Delta_l(\Lambda) = - \Delta_l(\Lambda). \]
\end{lemma}
\begin{proof}
The first equality follows by Lemma~\ref{lem:ijreflect} and the fact that $\Lambda$ is a lattice:
\[ \Delta_{-l}(\Lambda) = \Delta_{-l}(\Lambda - l) =  R_l \Delta_l(\Lambda). \]
By the central symmetry of $V = V(\Lambda)$, we also have
\[
F_l=V\cap(V+l) 
\quad \implies \quad
F_{-l} = V\cap(V-l) = (-V)\cap(-V-l) = -F_l,
\]
where $V\pm l = V_{\pm l}(\Lambda)$. This shows that $\Delta_{-l}(\Lambda) =  - \Delta_l(\Lambda)$.
\end{proof}

The second equality has the geometrical meaning that, for lattice configurations, the face $F_l$ is centrally symmetric around $h_l = l/2$.

\subsubsection{Moments, pyramid costs, interface costs}

We denote the \emph{moment} of the Voronoi cell $V(S)$ by
\[ F(S) := \int_{V(S)} \abs{x}^2 \,dx, \]
and define the \emph{pyramid costs}
\[ \lambda_i(S) := \int_{\Delta_i(S)} \abs{x}^2 \,dx, \quad i \in N_0(S), \]
so that $F(S) = \sum_{i \in N_0(S)} \lambda_i(S)$. 

We will also be interested in the vector quantities
\[ \mu_i(S) := \int_{\Delta_i(S)} \abs{x}^2 x \,dx, \quad \xi_{ij}(S) := \int_{E_{ij}(S)} \abs{x}^2 x \,dx; \qquad i,j \in N_0(S). \]

From Proposition~\ref{prop:regioninv}, we conclude also the following equivariance statement:

\begin{proposition}\label{prop:costinv}
Let $T = \alpha R$ be a similarity transformation, i.e.~$\alpha > 0$, $R \in O(3)$. Then
\begin{align*} 
F(TS) = \alpha^5 F(S), \quad \lambda_i(TS) = \alpha^5 \lambda_i(S); \quad \mu_i(TS) = \alpha^5 T \mu_i(S), \quad &\xi_{ij}(TS) = \alpha^{4} T\xi_{ij}(S).
\end{align*}
\end{proposition}
\begin{proof}
The first three identities follow from the change of variables $y = Tx$ in the volume integrals, using Proposition~\ref{prop:regioninv}, with $\abs{y}^2 \,dy = \alpha^5 \abs{x}^2 \,dx$ and $\abs{y}^2 y \,dy = \alpha^5 \abs{x}^2 Tx \,dx$. For $\xi_{ij}$, the surface measure scales by $\alpha^2$, giving $\xi_{ij}(TS)=\alpha^4 T\xi_{ij}(S)$.
\end{proof}

Lemma~\ref{lem:ijreflect} likewise yields
\begin{equation}\label{eq:costreflect}
\lambda_{-x_i}(S-x_i) = \lambda_i(S) \quad \text{and} \quad \mu_{-x_i}(S-x_i) = R_{x_i} \mu_i(S).
\end{equation}

We will also make use of the following relation between integrals on $\Delta_i$ and integrals on $F_i$:

\begin{lemma}\label{lem:faceint}
    Let $\phi$ be an $\alpha$-homogeneous function on $\R^3$. Then
    \[ \int_{\Delta_i} \phi(x) \,dx = \frac{\abs{h_i}}{\alpha+3} \int_{F_i} \phi(y) \,dy. \]
\end{lemma}
\begin{proof}
    Complete $\frac{h_i}{\abs{h_i}}$ into an orthonormal basis $\{e_1, e_2, \frac{h_i}{\abs{h_i}}\}$. Since $F_i - h_i \bot h_i$, we can parametrize $F_i$ isometrically by two coordinates $(x_1,x_2) \mapsto x_1 e_1 + x_2 e_2 + h_i$. Consider now the following parametrization of $\Delta_i$:
    \[ \Phi : [0,1] \times F_i \to \Delta_i; \quad (t, x_1, x_2) \mapsto t(x_1 e_1 + x_2 e_2 + \abs{h_i} \frac{h_i}{\abs{h_i}}) = t(x_1, x_2, \abs{h_i}). \]
    The Jacobian determinant of this parametrization is
    \[ \left| \frac{\partial \Phi}{\partial(t,x_1,x_2)} \right| = \left| \begin{matrix}
        x_1 & x_2 & \abs{h_i} \\ t & 0 & 0 \\ 0 & t & 0
    \end{matrix} \right| = t^2 \abs{h_i}. \]
    Thus
    \[ \int_{\Delta_i} \phi(x) \,dx = \int_0^1 \int_{F_i} \phi(t y) t^2 \abs{h_i} \,dt \,dy = \abs{h_i} \int_0^1 t^\alpha t^2 \,dt \int_{F_i} \phi(y) \,dy = \frac{\abs{h_i}}{\alpha+3} \int_{F_i} \phi(y) \,dy. \qedhere \]
\end{proof}

Thus in particular,
\begin{align*}
    \lambda_i(S) = \frac{\abs{h_i}}{5} \int_{F_i(S)} \abs{x}^2 \,dx, \qquad 
    \mu_i(S)  = \frac{\abs{h_i}}{6} \int_{F_i(S)} \abs{x}^2 x \,dx.
\end{align*}
From this, we obtain the following useful identity:
\begin{equation}\label{eq:muidotli}
    \mu_i(S) \cdot w(x_i) = \frac{\abs{h_i}}{6} \int_{F_i(S)} \abs{x}^2 \underbrace{x \cdot w(x_i)}_{\equiv 1} \,dx = \frac{\abs{h_i}}{6} \int_{F_i(S)} \abs{x}^2 \,dx = \frac56 \lambda_i(S).
\end{equation}
In view of Lemmas~\ref{lem:ijreflect} and~\ref{lem:latcell}, we deduce that
\begin{equation}\label{eq:muijreflect}
    \mu_i(S) - \mu_{-x_i}(S-x_i) = (I - R_{x_i}) \mu_i(S) = x_i w(x_i) \cdot \mu_i(S) = \frac56 \lambda_i(S) x_i,
\end{equation}
and for a lattice $\Lambda$,
\begin{equation}\label{eq:latcell}
    2 \mu_l(\Lambda) = \mu_l(\Lambda) - \mu_{-l}(\Lambda) = \frac56 \lambda_l(\Lambda) l \quad \implies \quad \mu_l(\Lambda) = \frac{5}{12} \lambda_l(\Lambda) l.
\end{equation}

Finally, we define the following quantity which will appear in the first variations of both discrete and continuous functionals:
\begin{equation}\label{eq:uldef}
    U_i(S) := \frac52 \lambda_i(S) w(x_i).
\end{equation}

\subsection{Symmetries of the quantization error}

Let $S = \{x_l\}_{l \in \sL}$ be a locally finite set of quantizers indexed by the BCC lattice $\sL$, with $x_0 = 0$. Then the induced quantization problem satisfies both continuous and discrete invariances owing to the symmetries of the BCC lattice.

\paragraph{Continuous symmetries.}

In light of Propositions~\ref{prop:regioninv} and~\ref{prop:costinv}, the following operations fix the quantization problem:
\begin{itemize}
    \item \textbf{Scaling:} $S \mapsto \alpha S$, $\alpha > 0$;
    \item \textbf{Rotation:} $S \mapsto R S$, $R \in O(3)$.
\end{itemize}

\paragraph{Discrete symmetries.}

The symmetry group $G$ acts on $S$ in two ways:
\begin{itemize}
    \item \textbf{Left action:} $S \mapsto gS = \{g x_l\}_{l \in \sL}$;
    \item \textbf{Right action:} $S \mapsto Sg = \{x_{gl}\}_{l \in \sL}$.
\end{itemize}
The first preserves the quantization error since each $g \in O(3)$, and the second by the fact that $G$ preserves $\sL$ and hence only permutes the indexing while keeping the underlying set invariant. This notation is justified by the fact that, for linear distortions of the form $S = M\sL = \{Ml\}_{l \in \sL}$, $gS = \{gMl\}_{l \in \sL}$ and $Sg = \{Mgl\}_{l \in \sL}$.

Putting the two actions together, we can define a \emph{conjugate action} 
\[ S \mapsto gSg^\top = \{g x_{g^\top l}\}_{l \in N_0}, \] 
and this fixes the unperturbed BCC lattice: 
\[ x_l = l \quad \implies \quad g x_{g^\top l} = g g^\top l = l. \]
For linear distortions, this likewise corresponds to the conjugate action of $G$ on matrices.

On the individual pyramids and interfaces and their respective costs, conjugation instead acts as follows:
\begin{align*} 
\Delta_l(gSg^\top) = g \Delta_{g^\top l}(S), & \quad E_{ll'}(gSg^\top) = g E_{g^\top l,g^\top l'}(S); \\
\lambda_l(gSg^\top ) = \lambda_{g^\top l}(S), \quad \mu_l(gSg^\top ) & = g \mu_{g^\top l}(S), \quad \xi_{ll'}(gSg^\top) = g\xi_{g^\top l,g^\top l'}(S).
\end{align*}
Note in particular that this conjugate action preserves the total error by permuting the pyramid costs $\lambda_l$, and that $\lambda_l(g S g^\top) = \lambda_l(S)$ for $g \in G_l$.

The vectors $U_l$ also transform the same way as $\mu_l$:
\begin{equation}\label{eq:ulsym}
U_l(gSg^\top) = \frac52 \lambda_l(gSg^\top) w(g x_{g^\top l}) = \frac52 \lambda_{g^\top l}(S) g w(x_{g^\top l}) = g U_{g^\top l}(S)
\end{equation}
since
\[ w(g x) = \frac{2}{\abs{gx}^2} g x = g \frac{2x}{\abs{x}^2} = g w(x), \quad x \in \R^3. \]

\paragraph{Unperturbed BCC lattice.}
Now consider the unperturbed lattice $\sL$, which is fixed by the above conjugate action.
In this case, we have
\[ \Delta_{gl}(\sL) = g \Delta_{g^\top g l}(\sL) = g \Delta_l(\sL), \]
hence
\[ \lambda_{gl}(\sL) = \lambda_l(\sL), \quad \mu_{gl}(\sL) = g \mu_l(\sL), \quad U_{gl}(\sL) = g U_l(\sL), \quad \ldots\]
Therefore the pyramid costs $\lambda_l(\sL)$ are constant along the $G$-orbits of $N_0$, and take only two values:
\begin{itemize}
\item $\lambda_{ax}$ on the 6 square pyramids ($l \in N_{ax}$);
\item $\lambda_{diag}$ on the 8 hexagonal pyramids ($l \in N_{diag}$).
\end{itemize}
Hence in total,
\[
F(\sL) = 6 \lambda_{ax} + 8 \lambda_{diag}.
\]
In Appendix~\ref{app:unperturb}, we compute explicitly
\[
\lambda_{ax} = \frac{13}{120}, \quad \lambda_{diag} = \frac{69}{320}
\quad \implies \quad
F(\sL) = 6\cdot \frac{13}{120} + 8\cdot \frac{69}{320} = \frac{19}{8}.
\]
The Voronoi cell $V(\sL)$ satisfies the additional property that all the interfaces $E_{ll'}$ are congruent to one another, being isosceles triangles with edge lengths $\frac{\sqrt5}{2},\,\frac{\sqrt5}{2},\,\frac{\sqrt2}{2}$. 
As a consequence, we show that for each $l \sim l'$,
\[ \xi_{ll'}(I) = \beta m_{ll'}, \quad \beta = \frac{7}{40} \]
for $m_{ll'}$ the midpoint of the edge $F_l \cap F_{l'}$. 
The fact that $m_{ll'} \in F_l \cap F_{l'}$ implies in particular that 
\begin{equation}\label{eq:muijcdotlij}
    w(l) \cdot \xi_{ll'}(I) = w(l') \cdot \xi_{ll'}(I) = \beta = \frac{7}{40}.
\end{equation}

\subsection{Voronoi diagrams on the torus}

Since we will be interested in quantizers on the torus, in order to apply the formalism of this section, we demonstrate that finite sets on the torus can be lifted to locally finite $\sL$-periodic sets on Euclidean space, with isometric Voronoi cells.

Let $S \subset \T$ be a finite subset of the flat torus $\T = \R^3/\sL$, imbued with the translation-invariant metric
\[ d_\T([x],[y]) = d_\T([x-y],[0]) = \min_{l \in \sL} \abs{x-y-l}, \quad [x],\,[y] \in \T. \]
That is, the toroidal distance between two congruence classes is the Euclidean distance between closest representatives. The minimum is attained by the local finiteness of $\sL$. This metric ensures that the canonical projection $\pi_\T \colon \R^3 \to \T$ is a Riemannian covering, i.e.~a covering map which is also a local isometry.

Taking the union of all conjugacy classes in $S$,
\[ \tilde S = \{ x \in \R^3 \mid [x] \in S\} \]
is evidently $\sL$-periodic, and locally finite: taking an arbitrary fundamental domain $\Pi$ for $\T$, we can write
\[ \tilde S = \bigsqcup_{l \in \sL} \tilde S \cap (\Pi+l), \]
and each set in the disjoint union maps bijectively to $S$ since each $[x] \in S$ has a unique representative in $\Pi + l$. 

Our goal is to present the Euclidean Voronoi diagram of $\tilde S$ as a lift of the toroidal Voronoi diagram of $S$: 

\begin{proposition}\label{prop:vorlift}
For any $[x_0] \in S$, the Voronoi cell 
\[ V([x_0]|S) = \{ [x] \in \T \mid d_\T([x],[x_0]) = d_\T([x],S) \} \]
of $[x_0]$ wrt $S$ and the Voronoi cell
\[ V(\tilde x_0|\tilde S) = \{ x \in \R^3 \mid \abs{x-\tilde x_0} = d(x,\tilde S) \} \]
of any representative $\tilde x_0 \in [x_0]$ wrt $\tilde S$ have the same moments:
\[ \int_{V([x_0]|S)} d_\T([x],[x_0])^2 \,d[x] = \int_{V(\tilde x_0|\tilde S)} \abs{x-\tilde x_0}^2 \,dx. \]
\end{proposition}
\begin{proof}
Ignoring differences of measure zero, we can work equivalently with the open Voronoi cells
\begin{align*}
    V^o([x_0]|S) &:= \bigl\{ [x] \in \T \mid d_\T([x],[x_0]) < d_\T([x], S \setminus \{[x_0]\}) \bigr\}; \\
    V^o(\tilde x_0|\tilde S) &:= \bigl\{ x \in \R^3 \mid \abs{x-\tilde x_0} < d(x,\tilde S \setminus \{\tilde x_0\})  \bigr\}.
\end{align*}
By the $\sL$-periodicity of $\tilde S$, for any other representative $\tilde x_0 + l$ of $[x_0]$, the cells
\[ V^o(\tilde x_0+ l|\tilde S) = \tilde V^o(\tilde x_0 |\tilde S)  +l \]
are disjoint. 
\begin{itemize}
    \item $\pi_\T$ restricted to $\tilde V^o := V^o(\tilde x_0|\tilde S)$ is injective: Suppose to the contrary that there exist $x, y \in \tilde V^o$ such that $y = x + l$. But then
    \[ y \in \tilde V^o +l = V^o(\tilde x_0+ l|\tilde S) \]
    which is disjoint with $\tilde V^o$, contradiction.
    \item $\pi_\T(\tilde V^o)$ is a full-measure subset of $V^o([x_0]|S)$: $[x] \in \pi_\T(\tilde V^o)$ iff there exists $l \in \sL$ such that $x+l \in \tilde V^o$, equivalently $x \in V^o(\tilde x_0-l|\tilde S)$. That is, the minimum $\min_{\tilde y \in \tilde S} \abs{x-\tilde y}$ is only attained by a single element of $\tilde S$, which also lies in $[x_0] = \tilde x_0 + \sL$.
    
    In contrast, $[x] \in V^o([x_0]|S)$ iff 
    \[ d_\T([x],[x_0]) = \min_{l \in \sL} \abs{x-\tilde x_0 - l} <  \min_{[y] \in S \setminus \{[x_0]\}} \min_{l \in \sL}  \abs{x-y-l} = \min_{\tilde y \in \tilde S \setminus [x_0]} \abs{x-\tilde y}. \]
    That is, the minimum $\min_{\tilde y \in \tilde S} \abs{x-\tilde y}$ is only attained by an element of $[x_0] = \tilde x_0 + \sL$, but might also be attained by multiple such elements. The latter is only possible on a zero-measure subset of $\T$, namely the projection of the following countable union of hyperplanes
    \[ \bigcup_{\substack{l,l' \in \sL\\ l\ne l'}} \{ x \in \R^3 \mid \abs{x-\tilde x_0-l} = \abs{x-\tilde x_0-l'} \} =\bigcup_{\substack{l,l' \in \sL\\ l\ne l'}} \left( \tilde x_0+\frac{l+l'}{2} + (l'-l)^\bot \right). \]
\end{itemize}
Moreover, $\pi_\T(\tilde V^o)$ is open in $\T$, since $\tilde V^o$ is by construction an open subset of $\R^3$ and the canonical projection $\pi_\T$ is an open map. 
Since the map $\pi_\T|_{\tilde V^o} \colon \tilde V^o \to \pi_\T(\tilde V^o)$ is bijective by the first item, and a local diffeomorphism and indeed a local isometry, it is also a global isometry.
We then conclude that 
\[ \int_{V([x_0]|S)} d_\T([x],[x_0])^2 \,d[x] = \int_{\pi_\T(\tilde V^o)} d_\T([x],[x_0])^2 \,d[x] = \int_{\tilde V^o} \abs{x-\tilde x_0}^2 \,dx = \int_{V(\tilde x_0|\tilde S)} \abs{x-\tilde x_0}^2 \,dx. \qedhere \]
\end{proof}

The purpose of choosing open Voronoi cells in the above argument is to ensure the injectivity of the canonical projection: otherwise, for the unscaled BCC lattice $S = \sL \subset \R^3$ for which $[S] = \{[0]\} \subset \T$, the Voronoi cell of $[0]$ in $\T$ is $\T$ itself, which is not homeomorphic to the Voronoi cell of $0$ wrt $\sL$ in $\R^3$.

\subsection{Local stability of neighbors for the BCC lattice}\label{sect:nbrstab}

We also show that for uniformly small perturbations of the BCC lattice, the incidence relations between Voronoi cells remain unchanged. 
This will be necessary for the proofs of Theorems~\ref{thm1} and~\ref{thm2}. The precise stability statement is as follows:

\begin{lemma}\label{lem:nbrstab}
There exists $\rho_0 > 0$ such that 
\[ N_0(S) = N_0(\sL) \]
for all perturbations $S = \{x_l\}_{l \in \sL}$ such that 
\[ \abs{x_l - l} \leq \rho_0 \abs{l} \quad \text{for all} \quad l \in \sL. \] 
Moreover, the Voronoi cell of $0$ is contained in a ball of radius $R > 0$ independent of $S$.
\end{lemma}

In particular, for linear perturbations, $M \sL$ has the same incidence relations as $\sL$ as long as $\norm{M-I}_{op} \leq \rho_0$. We will later deduce this also for perturbations induced by nonlinear maps.

\begin{proof}[Proof of Lemma~\ref{lem:nbrstab}]
Let $S = \{x_l\}_{l \in \sL}$ be a small perturbation of the BCC lattice $\sL$, with $x_l = l + \Delta l$, $\abs{\Delta l} \leq \rho_0 \abs{l}$ for some $\rho_0 \ll 1$ to be determined. Write $V = V(\sL)$ and $V' = V(S)$, and denote likewise by $F_l$ resp.~$F_l'$ the faces associated to each $l \in \sL \setminus \{0\}$ (n.b.~$x_0 = 0$ by construction).

To prove Lemma~\ref{lem:nbrstab}, we need to enforce the following:
\begin{enumerate}
    \item $V' \subseteq \bar B_R(0)$;
    \item $F_l'$ is empty for every $l \not\in N_0$;
    \item $F_l'$ is nonempty for every $l \in N_0$.
\end{enumerate}
For any $x \in \R^3$ and $\emptyset \neq I \subseteq \sL \setminus \{0\}$, observe that
\begin{equation}\label{eq:nbrlip} \abs{\max_{l \in I} x \cdot w(x_l) - \max_{l \in I} x \cdot w(l)} \leq \max_{l \in I} \abs{x \cdot w(x_l) - x \cdot w(l)} \leq \abs{x} \max_{l \in I} \abs{w(x_l) - w(l)} \leq L(\rho_0) \abs{x}. \end{equation}
where we apply the identity \eqref{eq:wdiff} in order to bound
\[ \max_{l \in I} \abs{w(x_l) - w(l)} = \max_{l \in I} \frac{2 \abs{\Delta l}}{\abs{l} \abs{x_l}} \leq \max_{l \in I} \frac{2 \rho_0 \abs{l}}{\abs{l} (1-\rho_0) \abs{l}} \leq \frac2{\sqrt{3}} \frac{\rho_0}{1-\rho_0}  =: L(\rho_0), \]
noting that $\min_{l \in \sL \setminus \{0\}} \abs{l} = \sqrt3 > 0$ for the BCC lattice, attained by any diagonal neighbor of $0$. 

Note likewise that
\[ \abs{ \max_{l \in I} x \cdot w(x_l) - \max_{l \in I} y \cdot w(x_l)} \leq \max_{l \in \sL \setminus \{0\}} \abs{x-y} \abs{w(x_l)} = \abs{x-y} \max_{l \in \sL \setminus \{0\}} \frac{2}{\abs{x_l}}, \]
with
\[ \min_{l \in \sL \setminus \{0\}} \abs{x_l} \geq (1-\rho_0) \min_{l \in \sL \setminus \{0\}} \abs{l} > 0 \]
as long as $\rho_0 < 1$. In particular, for the unperturbed lattice, $x \mapsto \max_{l \in I} x \cdot w(l)$ is $\frac{2}{\sqrt3}$-Lipschitz.

We then argue each point separately:
\begin{enumerate}
    \item For the unperturbed cell, we have $V \subseteq \bar B_{R_0}(0)$ for some $R_0 > 0$ (in fact, $R_0 = \frac{\sqrt5}{2}$). We will show that $V' \subseteq B_{2R_0}(0)$ for $\rho_0$ sufficiently small.

    By the convexity of Voronoi cells, it suffices to show that $V'$ never intersects the sphere of radius $2R_0$. That is,
    \[ \abs{x} = 2 R_0 \quad \implies \quad \max_{l \in \sL \setminus \{0\}} x \cdot w(x_l) > 1. \]
    Now since the unperturbed cell is contained in $B_{R_0}(0)$, we have
    \[ \abs{x} = R_0 \quad \implies \quad \max_{l \in \sL \setminus \{0\}} x \cdot w(l) \geq 1, \]
    and rescaling $x$ gives that $\max_{l \in \sL \setminus \{0\}} x \cdot w(l) \geq 2$ whenever $\abs{x} = 2 R_0$. Applying \eqref{eq:nbrlip}, we then obtain
    \[ \max_{l \in \sL \setminus \{0\}} x \cdot w(x_l) \geq \max_{l \in \sL \setminus \{0\}} x \cdot w(l) - L(\rho_0) \abs{x} \geq 2 - L(\rho_0) \cdot 2 R_0. \]
    Choosing $L(\rho_0) < \frac1{2R_0}$ then ensures that $\max_{l \in \sL \setminus \{0\}} x \cdot w(x_l) > 1$ on the sphere of radius $2R_0$, thus $V' \subseteq B_{2R_0}(0)$.
    
    \item We prove that the maximum $\max_{l \in \sL \setminus \{0\}} x \cdot w(x_l)$ is never attained by $l \not\in N_0$ for any $x \neq 0$. Normalizing by $\abs{x}$, it suffices to prove this for $x \in \bS^2$.

    Since the unperturbed cell is convex and contains the origin in its interior, its boundary projects down to the entire unit sphere. Hence the desired statement indeed holds for the unperturbed BCC lattice:
\[ \max_{l \in \sL \setminus \{0\}} x \cdot w(l) > \max_{l \not\in N_0} x \cdot w(l) \quad \text{for all }x \in \bS^2. \]
    Since $\bS^2$ is compact and both sides are Lipschitz, there exists a positive gap $\alpha_0 > 0$ such that 
\[ \max_{l \in \sL \setminus \{0\}} x \cdot w(l) \geq \max_{l \not\in N_0} x \cdot w(l) + \alpha_0 \quad \text{for all }x \in \bS^2. \]
    Then applying \eqref{eq:nbrlip} twice,
\[ \max_{l \in \sL \setminus \{0\}} x \cdot w(x_l) \geq \max_{l \not\in N_0} x \cdot w(x_l) + \alpha_0 - 2 L(\rho_0) \quad \text{for all }x \in \bS^2. \]
    Taking $L(\rho_0) < \frac{\alpha_0}{2}$ then ensures that the positive gap persists for the perturbed cell.
    
    \item This is the only step that is more specific to the BCC lattice: we note that for each $l \in N_0$, the face $F_l$ has nonempty interior, hence is not contained in any other face $F_{l'}$. That is, there exists $y_l \in F_l$ (e.g.~$y_l = l/2$) such that
    \[ 1 = y_l \cdot w(l) > \max_{l' \neq l} y_l \cdot w(l') > 0. \]
    Then likewise for the normalization $u_l = y_l/\abs{y_l}$,
    \[ u_l \cdot w(l) > \max_{l' \neq l} u_l \cdot w(l') > 0. \]
    Denoting by $\beta_l > 0$ the difference of the two sides, applying \eqref{eq:nbrlip} twice yields
    \begin{align*} u_l \cdot w(x_l) \geq u_l \cdot w(l) - L(\rho_0) = \max_{l' \neq l} u_l \cdot w(l') + \beta_l - L(\rho_0) \geq \max_{l' \neq l} u_l \cdot w(x_{l'}) + \beta_l - 2 L(\rho_0).
    \end{align*}
    Taking $L(\rho_0) < \beta_l/2$ then ensures that 
    \[ u_l \cdot w(x_l) > \max_{l' \neq l} u_l \cdot w(x_{l'}), \]
    and $y_l' = \frac{u_l}{u_l \cdot w(x_l)}$ (noting also that $u_l \cdot w(x_l) \geq \beta_l - L(\rho_0) > 0$) remains in the interior of the face $F_l' \subset V'$.
\end{enumerate}
In conclusion, taking $\rho_0 > 0$ sufficiently small, so that
\[ \frac{2}{\sqrt3} \frac{\rho_0}{1-\rho_0} = L(\rho_0) < \frac12 \min\left\{ \frac1{R_0}, \alpha_0, \min_{l \in N_0} \beta_l \right\} \]
ensures that every neighbor $l \in N_0$ of the unperturbed cell $V$ remains a neighbor of the perturbed cell $V'$, no $l \not\in N_0$ becomes a neighbor of the perturbed cell, and the perturbed cell remains uniformly bounded. 
\end{proof}

\begin{remark}
In light of steps 2 and 3, the obstacle to the local stability of neighbors is not the addition of new neighbors but the elimination of existing ones: if the unperturbed cell had any degenerate faces with empty interior, then there would exist arbitrarily small perturbations for which the corresponding neighbors are not preserved. 

\begin{figure}[H]
\centering
\includegraphics[width=.3\textwidth]{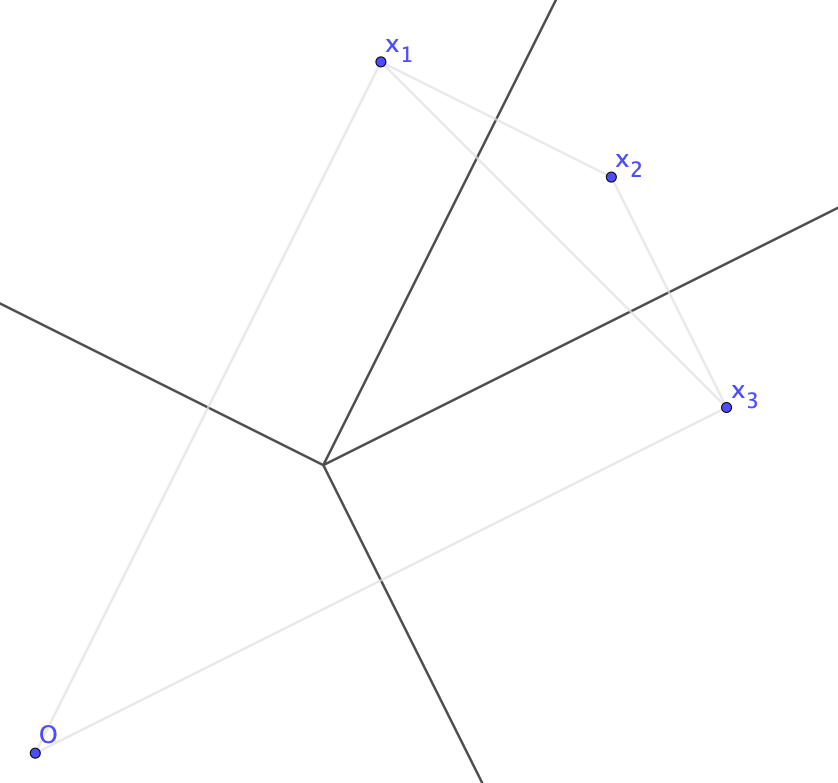}\hfill
\includegraphics[width=.3\textwidth]{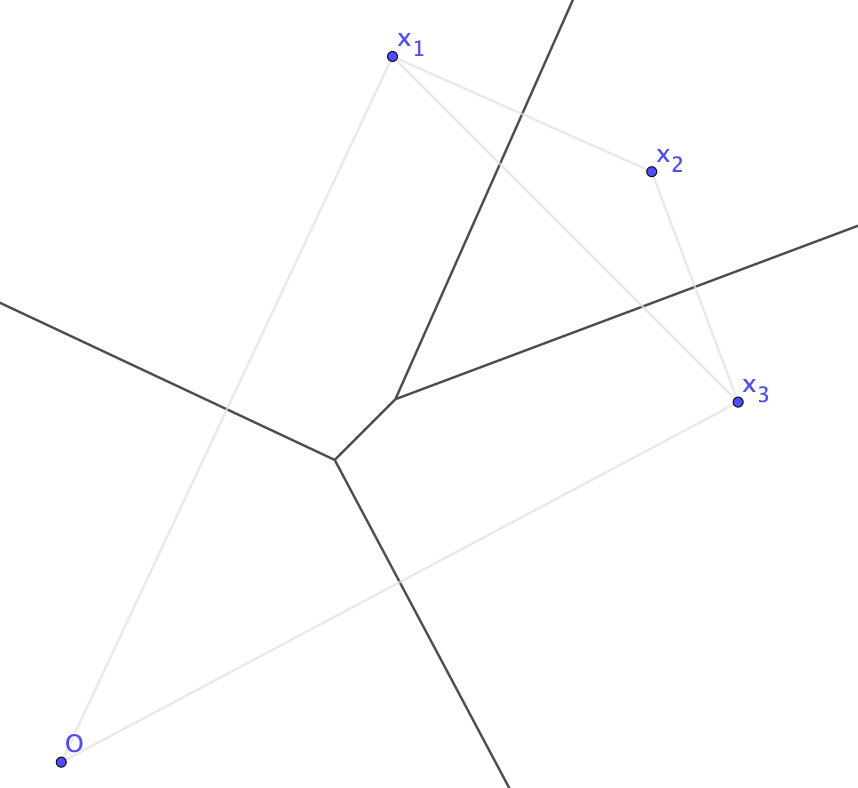}\hfill
\includegraphics[width=.3\textwidth]{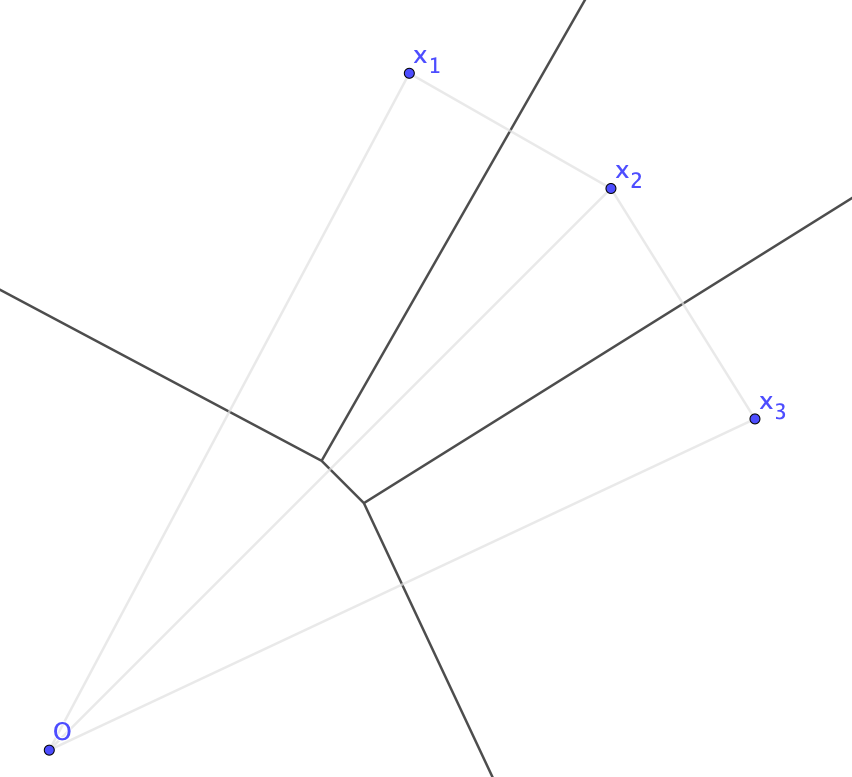}\hfill

\caption{Counterexample to stability: a Voronoi diagram with degenerate faces (in this case, a vertex with multiplicity $>d+1$, or an indeterminate Delaunay triangulation where adjacent simplices are circumscribed by the same sphere) can be perturbed to eliminate neighbor relations. }
\end{figure}
\end{remark}

\section{Smooth dependence of moments on neighbors}\label{sect:cellgrad}

We now seek to quantify how the moment of the Voronoi cell $V(S)$ is affected by small perturbations of its neighbors. Throughout this section, we assume that the set of neighbors $N_0(S)$ remains finite and unchanged under small perturbations, as guaranteed by Lemma~\ref{lem:nbrstab}.

Under this assumption, each vertex of the perturbed Voronoi cell depends smoothly on the same three neighbors as in the unperturbed cell, and the moments of each pyramid as well as the entire cell are smooth functions of the vertices hence the neighbors. In this section, we derive explicit formulas for the derivatives of costs up to order two as direct functions of the neighbors, avoiding the intermediate dependence on vertices.

\subsection{Differentiating costs in time}

Around an initial stable configuration $S_0 = \{x_i(0)\}_{i \in I}$, we consider linear perturbations $S(t) = \{x_i(t)\}_{i \in I}$ given by
\[ x_i(t) := x_i(0) + t \Delta x_i, \quad i \in I, \]
where we stipulate that $\Delta x_i = 0$ for the index $i \in I$ for which $x_i(0) = 0$, and that for $t$ sufficiently small, $N_0(S(t)) = N_0(S_0)$ is finite and $V(S(t))$ is bounded. We can then treat the quantization problem as dependent only on the finite set of neighbors $\{x_i(t)\}_{i \in N_0(S_0)}$.

We denote the Voronoi cell $V(S(t))$ simply by $V(t)$, and likewise for $F(t)$, $\lambda_i(t)$ and so on. We also denote $w(x_i(t))$ simply by $w_i(t)$.

Our main observation will be to express the time derivatives of cell and pyramid costs in terms of surface integrals, and reformulate these quantities as pyramid integrals using Lemma~\ref{lem:faceint}.

Since only the boundary of the Voronoi cell $V(t)$ changes in time, and not the integrand, the time derivative of the cost is given by
\begin{equation} \frac{d}{dt} F(t) = \int_{\partial V(t)} \abs{x}^2 (v \cdot \hat n) \,dx = \sum_{i=1}^k \int_{F_i} \abs{x}^2 (v \cdot \hat n_i) \,dx, \end{equation}
where $\hat n_i = \frac{x_i}{\abs{x_i}}$ is the outward-pointing normal vector to the face $F_i$, and $v \cdot \hat n_i$ is the normal velocity on the surface. 

Likewise for individual pyramid costs $\lambda_i(t)$, we have
\begin{equation} \frac{d}{dt} \lambda_i(t) = \int_{\partial \Delta_i(t)} \abs{x}^2 (v \cdot \hat n) \,dx = \int_{F_i(t)} \abs{x}^2 (v \cdot \hat n_i) \,dx + \sum_{j \sim i} \int_{E_{ij}(t)} \abs{x}^2 (v \cdot \hat n_{ij}) \,dx, \end{equation}
where $\hat n_{ij} = \frac{w_j-w_i}{\abs{w_j-w_i}}$ is the outward-pointing normal vector to the face $E_{ij}$. Indeed, observe that
\[ E_{ij} = \{ x \mid x \cdot w_i = x \cdot w_j = \max_{l = 1, \ldots, k} x \cdot w_l \leq 1 \} \subset \{ x \mid x \cdot (w_j-w_i) = 0 \}, \]
and the outward direction can be determined by noting that $x \cdot w_i > x \cdot w_j$ in the interior of $\Delta_i$ while $x \cdot w_i < x \cdot w_j$ in the interior of $\Delta_j$.

\begin{figure}[H]
\centering
\includegraphics[width=.4\textwidth]{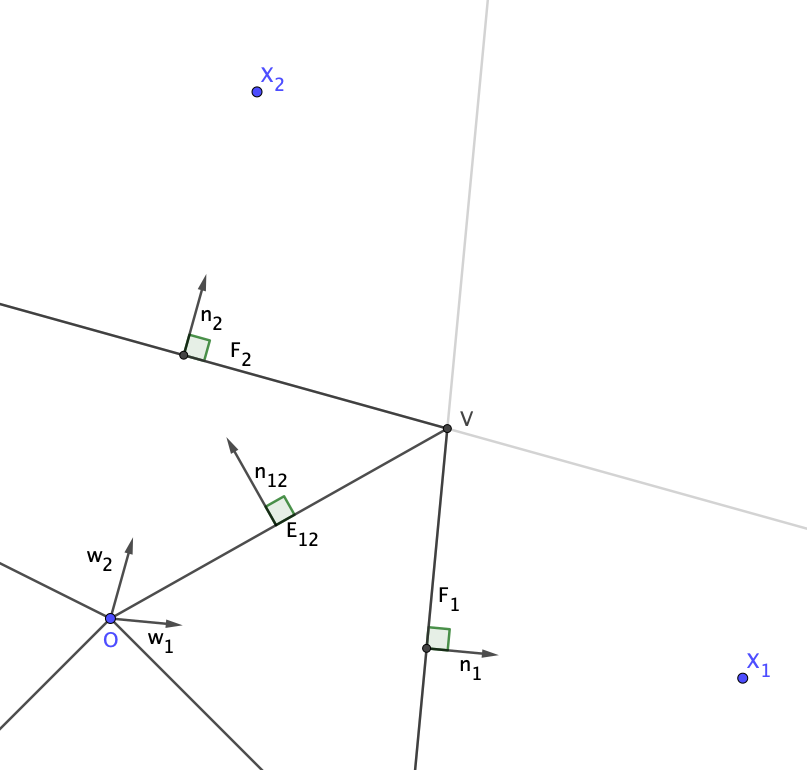}
\caption{Normals to faces and interfaces of a Voronoi cell.}
\end{figure}

Note that $\hat n_{ji}=-\hat n_{ij}$, so that the interaction terms coming from $x_j \sim x_i$ cancel out for the total cost 
\[ \dot F(t) = \sum_{i=1}^k \dot \lambda_i(t). \]

\paragraph{Computation of normal velocities.}
Let $t'\mapsto y(t')$ be a trajectory such that $y(t')\in F_i(t')$ for $t'$ near $t$. Then by construction, the trajectory satisfies the constraint
\[ y(t') \cdot w_i(t') \equiv 1. \]
Differentiating at time $t$ gives
\[ \dot y(t) \cdot w_i(t) + y(t) \cdot \dot w_i(t) = 0, \]
so that the normal velocity at $y(t)$ is given by
\begin{equation}\label{eq:normalvel_i} v \cdot \hat n_i = \dot y \cdot \frac{w_i}{\abs{w_i}} = - y \cdot \frac{\dot w_i}{\abs{w_i}}, \end{equation}
where we also compute
\begin{equation}\label{eq:dw} \dot w_i = \frac{d}{dt} w(x_i) = \frac{d}{dt} \frac{2 x_i}{\abs{x_i}^2} = \frac{2}{\abs{x_i}^2} \Delta x_i - \frac{2 x_i}{\abs{x_i}^4} 2 x_i \cdot \Delta x_i = \frac{2}{\abs{x_i}^2} R_{x_i} \Delta x_i,  \end{equation}
where again
\[ R_u = I - \frac{2}{\abs{u}^2} u u^\top \]
is the reflection matrix in the direction $u \neq 0$.

Likewise, let $t'\mapsto y(t')$ be a trajectory such that $y(t')\in E_{ij}(t')$ for $t'$ near $t$.
Since $y(t')$ stays on the moving interface, it satisfies the constraint
\[ y(t') \cdot (w_j(t') - w_i(t')) \equiv 0. \]
Differentiating at time $t$ gives
\[ \dot y(t) \cdot (w_j(t) - w_i(t)) + y(t) \cdot (\dot w_j(t) - \dot w_i(t)) = 0. \]
Dividing by $\abs{w_j(t) - w_i(t)}$ yields the normal velocity of the moving interface:
\begin{equation}\label{eq:normalvel_ij}
v \cdot \hat n_{ij}
=
\dot y\cdot \frac{w_j-w_i}{\abs{w_j-w_i}}
=
y \cdot \frac{\dot w_i-\dot w_j}{|w_i-w_j|}.
\end{equation}

\paragraph{Insertion into boundary integrals.}
Using \eqref{eq:normalvel_i} and \eqref{eq:normalvel_ij}, we then obtain that
\begin{align*}
    \int_{F_i} \abs{x}^2 (v \cdot \hat n_i) \,dx & = -\frac{\dot w_i}{\underbrace{\abs{w_i}}_{2/\abs{x_i}}} \cdot \int_{F_i} \abs{x}^2 x \,dx = -\frac{1}{\abs{x_i}} R_{x_i} \Delta x_i \cdot \int_{F_i} \abs{x}^2 x \,dx, \\
    \int_{E_{ij}} \abs{x}^2 (v \cdot \hat n_{ij}) \,dx & = \frac{\dot w_i-\dot w_j}{\abs{w_i - w_j}} \cdot \int_{E_{ij}} \abs{x}^2 x \,dx.
\end{align*}
We lastly apply Lemma~\ref{lem:faceint} to obtain
\[ \int_{F_i} \abs{x}^2 x \,dx = \frac{6}{\abs{h_i}} \int_{\Delta_i} \abs{x}^2 x \,dx = \frac{12}{\abs{x_i}} \mu_i. \]
Therefore, the time derivatives of pyramid costs and the total cell cost are given by 
\begin{align}
\label{eq:dm}    \dot F(t) & = -\sum_{i=1}^k \frac{12}{\abs{x_i}^2} \Delta x_i \cdot R_{x_i} \mu_i;
\\ \label{eq:dlambdai}
    \dot \lambda_i & = -\frac{12}{\abs{x_i}^2} \Delta x_i \cdot R_{x_i} \mu_i + \sum_{x_j \sim x_i} \frac{\dot w_i-\dot w_j}{\abs{w_i - w_j}} \cdot \xi_{ij}.
\end{align}
Note that since $\xi_{ji} = \xi_{ij}$, the interface terms coming from $x_j \sim x_i$ cancel out when summing over all $i$. 

In other words, \eqref{eq:dm} yields the following:

\begin{lemma}\label{lem:momentc1}
Given $k$ vectors $x_1,\ldots,x_k\in\R^3\setminus\{0\}$ whose induced Voronoi cell $V(x_1,\ldots,x_k)$ is bounded and each face $F_i$ is nonempty,
the moment $F(x_1,\ldots,x_k)$ is locally $C^1$ with
\begin{equation}\label{eq:gradcost}
\nabla_{x_i}F(x_1,\ldots,x_k)=-\frac{12}{\abs{x_i}^2}R_{x_i}\mu_i.
\end{equation}
\end{lemma}

Observe that we can continue the argument and differentiate $\mu_i$ in time analogously to $\lambda_i$, in order to obtain that the moment function is at least $C^2$ provided the original assumption of Lemma~\ref{lem:momentc1} holds. This will not actually be necessary for our purposes: we will see in the next section that the derivative of the total cost can be expressed in terms of $U_i$. This is evident for lattice perturbations:

\begin{remark}[Specialization to lattices] \label{rem:latcell}
If $S = \Lambda$ is a lattice, \eqref{eq:latcell} implies that
\[ \mu_i(S) = \mu_i(\Lambda) = \frac{5}{12} \lambda_i x_i \quad \implies \quad -\frac{12}{\abs{x_i}^2} \Delta x_i \cdot R_{x_i} \mu_i = \frac{5}{\abs{x_i}^2} \lambda_i x_i \cdot \Delta x_i = \underbrace{\frac52 \lambda_i w(x_i)}_{U_i(\Lambda)} \cdot \Delta x_i. \]
\end{remark}

\begin{remark}[Dilation invariance] \label{rem:dilsym}
Taking $\Delta x_i := x_i$ induces the family of dilations $S(t) = (1+t)S$, for which we know by Proposition~\ref{prop:costinv} that
\[ F((1+t)S) = (1+t)^5 F(S), \quad \lambda_i((1+t)S) = (1+t)^5 \lambda_i(S). \]
Differentiating both sides at $t = 0$, the first invariance is consistent with \eqref{eq:dm}:
\[ 5 F(S) = \dot F(0) = -\sum_{i=1}^k \frac{12}{\abs{x_i}^2} \mu_i \cdot R_{x_i} \Delta x_i = \sum_{i=1}^k \frac{12}{\abs{x_i}^2} \mu_i \cdot x_i = \sum_{i=1}^k 6 \mu_i \cdot w(x_i) = \sum_{i=1}^k 5 \lambda_i = 5 F(S), \]
where we apply \eqref{eq:muidotli} in the second-to-last equality. Differentiating $\lambda_i$ instead gives 
\[ 5 \lambda_i(S) = - \frac{12}{\abs{x_i}^2} \mu_i \cdot R_{x_i} \Delta x_i + \sum_{j \sim i} \frac{\dot w_i - \dot w_j}{\abs{w_i - w_j}} \cdot \xi_{ij}, \]
where the first term is equal to $5 \lambda_i(S)$ as above, and
\[ \dot w_j = \frac{2}{\abs{x_j}^2} R_{x_j} \Delta x_j = - \frac{2}{\abs{x_j}^2} x_j = - w(x_j), \]
and it follows analogously to \eqref{eq:muidotli} that $\xi_{ij} \cdot w(x_i) = \xi_{ij} \cdot w(x_j)$, so that the sum over $j \sim i$ vanishes.
\end{remark}

\paragraph{Derivative of $U_i$.}
In order to compute the Hessian of both costs around the unperturbed BCC lattice, as well as obtaining convergence and continuity statements for both costs, we also differentiate $U_i$ in time:

\begin{lemma}\label{lem:dui}
Again considering time-dependent perturbations $S(t) = \{x_i(t)\}_{i \in I}$,
\[ \frac{d}{dt} U_i(S) = A_i(S) \Delta x_i  + \sum_{j \sim i} B_{ij}(S) \Delta x_j, \]
where we define 
\begin{align*}
    A_i(S) & := \frac{5}{\abs{x_i}^2} \left[ \lambda_i(S) I - 6 w(x_i) \mu_i(S)^\top + w(x_i) \sum_{j \sim i} \frac{\xi_{ij}(S)^\top}{\abs{w(x_i)-w(x_j)}}  \right] R_{x_i}; \\
    B_{ij}(S) & := - \frac{10 x_i}{\abs{x_i}^2 \abs{x_j}^2} \frac{\xi_{ij}(S)^\top}{\abs{w(x_i)-w(x_j)}} R_{x_j}.
\end{align*}
\end{lemma}
\begin{proof}
Setting $w_i := w(x_i)$, \eqref{eq:dlambdai} gives
\[ \dot \lambda_i(t) = - \frac{12}{\abs{x_i}^2} \mu_i \cdot R_{x_i} \Delta x_i + \sum_{j \sim i} \xi_{ij} \cdot \frac{\dot w_i - \dot w_j}{\abs{w_i - w_j}}. \]
Observing that the first term can also be rewritten as $- 6 \mu_i \cdot \dot w_i$, we have
\begin{align*}
    \frac25 \dot U_i & = \lambda_i \dot w_i + \dot \lambda_i w_i
     = \left[ \lambda_i I - 6 w_i \mu_i^\top + w_i \sum_{j \sim i} \frac{\xi_{ij}^\top}{\abs{w_i - w_j}} \right] \dot w_i - w_i \sum_{j \sim i} \frac{\xi_{ij}^\top}{\abs{w_i - w_j}} \dot w_j.
\end{align*}
The desired statement follows by plugging in the values of $w_i$ and
\[ \dot w_i(t) = \frac{2}{\abs{x_i}^2} R_{x_i} \Delta x_i. \qedhere \]
\end{proof}

\begin{remark}[Discrete symmetries of $A_l$ and $B_{ll'}$]\label{rem:absym}
Around the unperturbed BCC lattice $S = \sL$, consider $x_l(t) = l + t \Delta x_l$ and their conjugates ${}^g x_l(t) = l + t g \Delta x_{g^\top l}$.

We can derive the law for $A_l$ and $B_{ll'}$ by differentiating \eqref{eq:ulsym} using Lemma~\ref{lem:dui}:
\begin{align*}
    \frac{d}{dt} U_l(gSg^\top) & = \frac{d}{dt} g U_{g^\top l}(S) \\
    A_l(gSg^\top) g \Delta x_{g^\top l}  + \sum_{l' \sim l} B_{ll'}(gSg^\top) g \Delta x_{g^\top l'} & = g A_{g^\top l}(S) \Delta x_{g^\top l}  + \sum_{l' \sim l} g B_{g^\top l,g^\top l'}(S) \Delta x_{g^\top l'}.
\end{align*}
where we apply the fact that $x_{l'} \sim x_l$ iff $l' \sim l$ iff $g^\top l' \sim g^\top l$, with $l, l' \in N_0$ (noting that conjugation preserves the neighbors of the origin).

Equating each $\Delta x_{g^\top l}$ and $\Delta x_{g^\top l'}$ term then gives
\begin{equation}\label{eq:absym} A_l(gSg^\top) = g A_{g^\top l}(S) g^\top; \quad B_{ll'}(gSg^\top) = g B_{g^\top l,g^\top l'}(S) g^\top. \end{equation}
In particular, since the unperturbed lattice is fixed by conjugation, we obtain $A_{gl}(\sL) = g A_l(\sL) g^\top$ and $B_{gl,gl'}(\sL) = g B_{ll'}(\sL) g^\top$. We will rely on this symmetry in Section~\ref{sect:thm2}.
\end{remark}

\subsection{Qualitative estimates for small perturbations}

Using the above explicit formulas for time derivatives, we now deduce local Lipschitz estimates for small perturbations of the BCC lattice in the sense of Lemma~\ref{lem:nbrstab}.

\begin{lemma}\label{lem:costsbdd}
For any perturbation $S = \{x_l\}_{l \in \sL}$ of the BCC lattice satisfying the assumptions of Lemma~\ref{lem:nbrstab}, the quantities
\[ F(S), \, \lambda_l(S), \, \abs{\mu_l(S)}, \, \abs{\xi_{ll'}(S)}, \, \norm{A_l(S)}, \, \norm{B_{ll'}(S)} \]
are uniformly bounded above by constants.
\end{lemma}
\begin{proof}
By Lemma~\ref{lem:nbrstab}, the Voronoi cell $V(S)$ is always contained in a ball centered at the origin with constant radius $R_0 > 0$. Then we can bound the scalar quantities by
\[ \lambda_l(S) \leq F(S) \leq \int_{R_0 \D^3} \abs{x}^2 \,dx = C_1 R_0^5 \]
for $\D^3$ the unit disk,
the vectors $\mu_l(S)$ by
\[ \abs{\mu_l(S)} \leq \int_{\Delta_l(S)} \abs{x}^3 \,dx \leq \int_{R_0 \D^3} \abs{x}^3 \,dx = C_2 R_0^6, \]
and the vectors $\xi_{ll'}(S)$ by 
\[ \abs{\xi_{ll'}(S)} \leq \int_{E_{ll'}(S)} \abs{x}^3 \,dx \leq \int_{R_0 \D^2} \abs{x}^3 \,dx = C_3 R_0^5 \]
where $\D^2$ denotes the unit two-dimensional disk on the hyperplane spanned by $E_{ll'}(S)$. Here $C_1, C_2, C_3 > 0$ are dimensional constants for the unit disk of dimension $3$ or $2$. 

To conclude the bounds on the matrices $A_l(S)$ and $B_{ll'}(S)$, it remains only to observe that by \eqref{eq:wdiff}, the vectors $w(x_l)$ are uniformly bounded and the differences $\abs{w(x_l)-w(x_{l'})}$ bounded away from zero. 
\end{proof}

\begin{lemma}\label{lem:costslip}
Let $S_1 = \{x_l\}_{l \in \sL}$, $S_2 = \{y_l\}_{l \in \sL}$ be two perturbations of the BCC lattice satisfying the assumptions of Lemma~\ref{lem:nbrstab}. 

Then
\[ \abs{F(S_1) - F(S_2)}, \, \abs{\lambda_l(S_1) - \lambda_l(S_2)}, \, \abs{U_l(S_1) - U_l(S_2)} \leq C \max_{l \in N_0} \abs{x_l - y_l} \]
for some constant $C > 0$ independent of $S_1$, $S_2$ or $l$.
\end{lemma}
\begin{proof}
Defining $S(t) = \{x_l(t)\}_{l \in \sL}$ by $x_l(t) = (1-t) x_l + t y_l$, the sets $S(t)$ also satisfy the assumptions of Lemma~\ref{lem:nbrstab} for all $t \in [0,1]$:
\[ \abs{x_l(t) - l} = \abs{(1-t) (x_l - l) + t (y_l - l)} \leq (1-t) \abs{x_l - l} + t \abs{y_l - l} \leq \rho_0 \abs{l}. \]
Thus for all $t \in [0,1]$, $F(S(t))$ is a function of the neighbors $x_l(t)$ for $l \in N_0$, and we can write
\[ \abs{F(S_2) - F(S_1)} = \abs{F(S(1)) - F(S(0))} \leq \int_0^1 \abs{\frac{d}{dt} F(S(t))} \,dt \]
where by Lemma~\ref{lem:momentc1},
\[ \abs{\frac{d}{dt} F(S(t))} \leq \sum_{l \in N_0} \frac{12}{\abs{x_l(t)}^2} \abs{ R_{x_l(t)}\mu_l(S(t)) \cdot \dot x_l(t)} \leq \sum_{l \in N_0} \frac{12}{\abs{x_l(t)}^2} \abs{\mu_l(S(t))} \abs{\dot x_l(t)} \]
with $\dot x_l(t) = y_l - x_l$ constant in time. Since the coefficients $\frac{12 \abs{\mu_l}}{\abs{x_l(t)}^2}$ are uniformly bounded by Lemma~\ref{lem:costsbdd} and the fact that $\abs{x_l} \geq (1-\rho_0) \abs{l}$ by assumption, the bound for $F$ follows.

The bounds for $\lambda_l$ and $U_l$ follow analogously, with the observation that
\[ \abs{\frac{d}{dt} w(x_l(t))} = \frac{2}{\abs{x_l(t)}^2} \abs{R_{x_l(t)} \dot x_l(t)} \leq C \abs{\dot x_l(t)} \leq C \max_{l \in N_0} \abs{x_l - y_l}. \qedhere \]
\end{proof}

\subsection{Gradient and Hessian of linear perturbations}\label{sect:latgrad}

Now consider the cost of linear perturbations $S = M \sL$ of the BCC lattice $\sL$, for $\norm{M-I} \leq \rho_0$, so that $N_0(M) = N_0(I)$.

In light of Remark~\ref{rem:latcell}, we can express Lemma~\ref{lem:momentc1} as follows:
\begin{equation}\label{eq:gradf} \nabla F(M) = \sum_{l \in N_0} U_l(M) l^\top. \end{equation}
Indeed for $M(t) = M_0 + t N$, $S(t) = \{M_0 l + t N l\}_{l \in \sL}$, hence
\[ \ip{\nabla F(M)}{N} = \frac{d}{dt} F(M(t)) = \sum_{l \in N_0} U_l(M) \cdot N l = \sum_{l \in N_0} \Tr( (U_l(M) l^\top)^\top N). \]
Differentiating again and applying Lemma~\ref{lem:dui}, we also obtain
\begin{equation}\label{eq:hessf} \nabla^2 F(M)[N, N] = \frac{d^2}{dt^2} F(M(t)) = \sum_{l \in N_0} N l \cdot \left( A_l(M) N l + \sum_{l' \sim l} B_{ll'}(M) N l' \right), \end{equation}
with
\begin{align*}
    A_l(M) & = \frac{5}{\abs{Ml}^2} \left[ \lambda_l(M) I - 6 w(Ml) \mu_l(M)^\top + w(Ml) \sum_{l' \sim l} \frac{\xi_{ll'}(M)^\top}{\abs{w(Ml)-w(Ml')}}  \right] R_{Ml}
    \\ & = \frac{5}{\abs{Ml}^2} \left[ \lambda_l(M) \left( I - \frac{5}{\abs{Ml}^2} (Ml) (Ml)^\top \right) + w(Ml) \sum_{l' \sim l} \frac{\xi_{ll'}(M)^\top}{\abs{w(Ml)-w(Ml')}}  \right] R_{Ml}; \\
    B_{ll'}(M) & = - \frac{10 Ml}{\abs{Ml}^2 \abs{Ml'}^2} \frac{\xi_{ll'}(M)^\top}{\abs{w(Ml)-w(Ml')}} R_{Ml'}.
\end{align*}
In the second line, we apply $\mu_l(M) = \frac{5}{12} \lambda_l(M) M l$ again by Remark~\ref{rem:latcell}. 

We remark again that the matrices $A_l(M)$ and $B_{ll'}(M)$ can be differentiated further, in order to conclude the smoothness of $F(M)$ in the $\rho_0$-neighborhood of the identity; we omit such calculations because we will only need the specific formulas for derivatives up to second order.

\subsubsection{Discrete and continuous symmetries}

Differentiating the discrete invariance $F(gMg^\top) = F(M)$ along perturbations $M(t) = M_0 + t N$, we obtain
\[ \ip{\nabla F(g M g^\top)}{g N g^\top} = \ip{\nabla F(M)}{N}, \quad \nabla^2 F(gMg^\top)[gN g^\top, g N g^\top] = \nabla^2 F(M)[N, N]. \]
In particular, at $M = I$, the gradient and Hessian or $F$ are also $G$-invariant functionals. By Lemmas~\ref{lem:invlinforms} and~\ref{lem:invquadforms}, we then conclude that
\[ \nabla F(I) = c I, \quad \nabla^2 F(I)[N,N] = \sum_{i=1}^4 c_i \norm{\pi_i N}^2. \]
Most of these coefficients are also fixed by the continuous symmetries of the quantization error. Taking $M_0 = N = I$ hence $M(t) = (1+t) I$, differentiating the dilation symmetry
\[ F(M(t)) = F((1+t)I) = (1+t)^5 F(I) \]
both ways at $t = 0$, we obtain
\[ \ip{\nabla F(I)}{I} = \dot F(M(0)) = 5 F(I) \quad \implies \quad c = \frac53 F(I) \]
where the factor of $3$ comes from the fact that $\ip{I}{I} = \Tr(I) = 3$. We then conclude that
\[ \nabla F(I) = \frac53 F(I) I = \frac{95}{24} I, \]
which is consistent with \eqref{eq:gradf}:
\begin{align*}
\nabla F(I) & = \sum_{l \in N_0} U_l(I) l^\top = \sum_{l \in N_0} 5 \lambda_l(I) \frac{l l^\top}{\abs{l}^2}
\\ & = 5 \lambda_{ax} \sum_{l \in N_{ax}} \frac{l l^\top}{\abs{l}^2} + 5 \lambda_{diag} \sum_{l \in N_{diag}} \frac{l l^\top}{\abs{l}^2} 
\\ & = \frac53 (\abs{N_{ax}} \lambda_{ax} + \abs{N_{diag}} \lambda_{diag}) I = \frac53 F(I) I,
\end{align*}
where we plug in \eqref{eq:faxdiag}.

Differentiating twice, we instead obtain
\[ \nabla^2 F(I)[I,I] = \ddot F(M(0)) = 20 F(I) \quad \implies \quad c_1 = \frac{20}{3} F(I) \]
where again $\norm{\pi_1 I}^2 = \norm{I}^2 = 3$. 

We can also fix the antisymmetric factor $c_4$ using the rotation symmetry of $F$. 
Since any antisymmetric matrix $A \in \R^{3\times3}$ induces a one-parameter family of rotations $e^{tA}$, all the $t$-derivatives of $F(e^{tA}) = F(I)$ must vanish identically. By the chain rule, we obtain
\[ \frac{d}{dt} F(e^{tA}) = \ip{\nabla F(e^{tA})}{A e^{tA}} \equiv 0 \]
hence
\[ \frac{d^2}{dt^2} F(e^{tA}) = \nabla^2 F(e^{tA})[A e^{tA}, A e^{tA}] + \ip{\nabla F(e^{tA})}{A^2 e^{tA}} \equiv 0. \]
Now setting $t = 0$ yields
\[ 0 = \nabla^2 F(I)(A, A) + \ip{\nabla F(I)}{A^2} = \nabla^2 F(I)(A,A) + \frac53 F(I) \Tr(A^2) \]
thus
\[ \nabla^2 F(I)(A,A) = - \frac53 F(I) \Tr(A^2) = \frac53 F(I) \Tr(A^\top A) = \frac53 F(I) \norm{A}^2 \]
for any antisymmetric matrix $A$. Therefore $c_4 = \frac53 F(I)$.

Therefore, inserting $\norm{\pi_1 N}^2 = \frac13 \Tr(N)^2$, the continuous symmetries of the quantization error force the Hessian at identity to be of the form
\begin{equation}\label{eq:hesssym} \nabla^2 F(I)[N,N] = \frac{20}{9} F(I) \Tr(N)^2 + c_2 \norm{\pi_2 N}^2 + c_3 \norm{\pi_3 N}^2 + \frac53 F(I) \norm{\pi_4 N}^2. \end{equation}
The factors $c_2$ and $c_3$ are instead specific to the shape of the lattice, and will be computed in Section~\ref{sect:thm2}.

\subsubsection{Normalized cell cost}

To render dilations completely invariant, we can also introduce the normalization
\[ \hat F(M) := (\det M)^{-5/3} F(M). \]
The geometric motivation for this definition is that the Voronoi cell of the lattice $M\sL$, being a fundamental domain for the lattice, has volume $\det (M L) = \det M \cdot \det L$ where $L$ is any generating matrix for the BCC lattice.

Setting again $M(t) = M_0 + t N$, observe the identity
\begin{align*} 
\frac{d}{dt} \left[ (\det M)^{-5/3} g(M) \right] & =(\det M)^{-5/3} \left[ \frac{d}{dt} g(M) - \frac53 g(M) \frac{d}{dt} \log \det M \right] 
\\ & =  (\det M)^{-5/3} \left[ \frac{d}{dt} - \frac53 \Tr(M^{-1} N) \right] g(M), 
\end{align*}
where we use that $\frac{d}{dt} \det(M) = \det(M) \Tr(M^{-1} \frac{dM}{dt})$.

Thus
\[ \ip{\nabla \hat F(M)}{N} = \frac{d}{dt} \hat F(M(t)) = (\det M)^{-5/3} \left[ \frac{d}{dt} F(M(t)) - \frac53 \Tr(M^{-1} N) F(M(t)) \right], \]
and the gradient of $\hat F(M)$ is given by
\[ \nabla \hat F(M) = (\det M)^{-5/3} \left[ \nabla F(M)  - \frac53 F(M) M^{-\top} \right]. \]
At $M = I$, we verify that
\[ \nabla \hat F(I) = \nabla F(I)  - \frac53 F(I) I = 0 \]
by the above symmetry calculations.

Differentiating again yields
\begin{align*}
 \nabla^2 \hat F(M)[N, N] & = \frac{d^2}{dt^2} \hat F(M(t)) 
 = \frac{d}{dt} (\det M)^{-5/3} \left[ \frac{d}{dt} F(M(t)) - \frac53 \Tr(M^{-1} N) F(M(t)) \right]
 \\ & = (\det M)^{-5/3} \left[ \frac{d}{dt} - \frac53 \Tr(M^{-1} N) \right]^2 F(M(t)),
\end{align*}
where for $g(t) = \frac53 \Tr(M(t)^{-1} N)$, we have
\begin{align*}
\left[ \frac{d}{dt} - g(t) \right]^2 f(t) 
& = \frac{d^2}{dt^2} f(t) - g(t) \frac{d}{dt} f(t) - \frac{d}{dt} (g(t) f(t)) + g(t)^2 f(t) 
\\ & = \ddot f(t) - 2 g(t) \dot f(t) + (-\dot g(t)+g(t)^2) f(t)
\end{align*} 
and
\[ \dot g(t) =  \frac53 \Tr\left(\frac{d}{dt} (M(t)^{-1}) N\right) = - \frac53 \Tr(M^{-1} N M^{-1} N). \]
In particular, for $M = I$, we obtain
\begin{align*}
\nabla^2 \hat F(I)[N, N] 
& = \nabla^2 F(I)[N, N] - \frac{10}3 \Tr(N) \ip{\nabla F(I)}{N} + \frac53 \left(  \Tr(N^2) + \frac53 \Tr(N)^2 \right) F(I)
\\ & = \nabla^2 F(I)[N, N] + \frac53 F(I) \left( \Tr(N^2) - 2 \Tr(N) \cdot \frac53 \Tr(N) + \frac53 \Tr(N)^2 \right)
\\ & = \nabla^2 F(I)[N, N] + \frac53 F(I) \left( \Tr(N^2) - \frac53 \Tr(N)^2 \right)
\end{align*}
where we plug in $\nabla F(I) = \frac53 F(I) I$ (that is, $\dot f(0) = g(0) f(0)$). 

Now decomposing $N = N_s + N_a$ into symmetric and antisymmetric components, observe that
\[ \Tr(N^2) = \Tr((N_s+N_a) N) = \Tr((N_s^\top - N_a^\top) N) = \ip{N_s}{N} - \ip{N_a}{N} = \norm{N_s}^2 - \norm{N_a}^2 \]
by orthogonality. Decomposing the symmetric component further into invariant subspaces, we arrive at
\begin{equation}\label{eq:trdecomp} \Tr(N^2) = \underbrace{\norm{\pi_1 N}^2}_{\frac13 \Tr(N)^2} + \norm{\pi_2 N}^2 + \norm{\pi_3 N}^2 - \norm{\pi_4 N}^2.  \end{equation}
In light of \eqref{eq:hesssym}, we conclude that the Hessian of the normalized cost decomposes as
\begin{equation}\label{eq:normhesssym} \nabla^2 \hat F(I)[N, N] = \left(c_2 + \frac53 F(I)\right) \norm{\pi_2 N}^2 + \left(c_3 + \frac53 F(I)\right) \norm{\pi_3 N}^2. \end{equation}
Note that the identity and antisymmetric components vanish exactly, consistent with the invariance of the normalized cost. In Section~\ref{sect:thm2}, we will verify that both coefficients of \eqref{eq:normhesssym} are positive, corroborating the Barnes-Sloane result \cite{barslo} that the BCC lattice is a local minimizer of the normalized moment among lattice quantizers.

\section{Variational analysis of discrete and continuous functionals}\label{sect:thm1}

Having established the differentiability of moments and pyramid costs, we now move on to the analysis of the total quantization error of small perturbations of the BCC lattice. 

We first demonstrate in Section~\ref{sect:findiff} that the quantization error induced by diffeomorphisms admits the following finite-difference representation: 
\[ \cF_N(X) = \frac1N \sum_{x_0 \in \sL_\eta} F(\cD_\eta X(x_0)), \]
where $F(\cD_\eta X(x_0))$ is defined as in Section~\ref{sect:voronoi}, with respect to the index point set
\[ \cD_\eta X(x_0) := \{D^{(l)}_\eta X(x_0)\}_{l \in \sL}; \quad D^{(l)}_\eta X(x_0) := l + D^{(l)}_\eta u(x_0) = l + \frac{u(x_0+\eta l)-u(x_0)}{\eta}.  \]

The formal limit of the functionals $\cF_N(X)$ as $N \to \infty$, which we will make rigorous in this section, is the integral functional
\[ \cF(X) = \frac{1}{\abs{\T}} \int_\T F(\nabla X(x)) \,dx \]
where for $M \in \R^{3\times3}$, we denote $F(M \sL)$ simply by $F(M)$. Evidently, both functionals coincide at identity:
\[ \cF_N(\id) = \cF(\id) = F(I) = \int_{V(0|\sL)} \abs{x}^2 \,dx = \frac{19}{8}. \]

We prove Theorem~\ref{thmdiff} in Section~\ref{sect:thmdiff}, by considering time-dependent families $(X_t)_t$ of diffeomorphisms satisfying Hypothesis~\ref{hyp1} at all times and applying the results of the previous section to obtain
\begin{align} \label{eq:dfntul}
\frac{d}{dt} \cF_N(X_t) &= \frac1N \sum_{x_0 \in \sL_\eta} \sum_{l \in N_0} U_l(\cD_\eta X_t(x_0)) \cdot D_\eta^{(l)}\dot X_t(x_0); \\
\label{eq:dftul}
\frac{d}{dt} \cF(X_t) &= \frac1{\abs{\T}} \int_\T \sum_{l \in N_0} U_l(\nabla X_t(x)) \cdot (\nabla \dot X_t(x) l) \,dx.
\end{align}

\subsection{Discrete calculus on the BCC lattice}\label{sect:graphcalc}

For each $\eta = n^{-1}$, the set $\sL_\eta = \eta\sL/\sL$ is an abelian group with $N = n^3$ elements. Equating $\sL_\eta$ with its Cayley graph induced by the generators $\eta N_0$, we can define a discrete calculus on $\sL_\eta$. 

The discrete equivalent of vector fields on graphs are functions on directed edges $a \to b$, representing the amount of flow from vertex $a$ to vertex $b$. In our setting, instead of picking a direction for each edge at random, we simply define a vector field to be a function $V$ on $N_0 \times \sL_\eta$, denoted $V_l(x_0)$, such that it is consistent along each edge $x_0 \to x_0+\eta l$:
\begin{equation}\label{eq:discvector} V_{-l}(x_0+\eta l) = - V_l(x_0). \end{equation}
It is evident that, for any scalar function $u$ on $\sL_\eta$, its directional derivatives
\[ D^{(l)}_\eta u(x_0) := \frac{u(x_0+\eta l) - u(x_0)}{\eta} \]
satisfy the above consistency condition. The discrete vector field $D^{(l)}_\eta u(x_0)$ is then a discrete analogue of the gradient of $u$. Likewise, the functions $U_l(\cD_\eta X(x_0))$ define a discrete vector field by \eqref{eq:ulvector}. 

We also define the \emph{discrete divergence} of a vector field $V$ as the average flux exiting $x_0$:
\[ \div_\eta V(x_0) := \frac2\eta \sum_{l \in N_0} \gamma_l V_l(x_0) = \sum_{l \in N_0} \gamma_l \frac{V_l(x_0) + V_{-l}(x_0)}{\eta} = \sum_{l \in N_0} \gamma_l D^{(l)}_\eta V_l(x_0-\eta l), \]
where $\gamma_l>0$ is a fixed choice of weights on each $l \in N_0$ such that $\gamma_{-l} = \gamma_l$.

Another motivation for this choice of divergence is that, applied to directional derivatives, one recovers a discrete Laplacian:
\[ \div_\eta D^{(l)}_\eta u(x_0) = \sum_{l \in N_0} \gamma_l \frac{D^{(l)}_\eta u(x_0) + D^{(-l)}_\eta u(x_0)}{\eta} = \sum_{l \in N_0} \gamma_l \frac{u(x_0+\eta l) - 2 u(x_0) + u(x_0-\eta l)}{\eta^2}. \]

Here $\gamma_l$ acts as a discrete metric tensor, defining an inner product on vector fields:
\[ \ip{U}{V} := \frac1N \sum_{x_0 \in \sL_\eta} \sum_{l \in N_0} \gamma_l U_l(x_0) \cdot V_l(x_0). \]
With respect to this inner product, $\div_\eta$ is exactly the negative adjoint to the discrete gradient: for any vector field $V$ and scalar function $u$, we have
\begin{align*}
    \frac1N \sum_{x_0 \in \sL_\eta} \sum_{l \in N_0} \gamma_l V_l(x_0) \cdot D^{(l)}_\eta u(x_0) 
    & = \frac1N \sum_{x_0 \in \sL_\eta} \frac1\eta \sum_{l \in N_0} \gamma_l V_l(x_0) \cdot \left(u(x_0+\eta l) - u(x_0) \right) 
    \\ & = \frac1N \sum_{x_0 \in \sL_\eta} \frac1\eta \sum_{l \in N_0} \gamma_l \left( - V_{-l}(x_0+\eta l) \cdot u(x_0+\eta l) - V_l(x_0) \cdot u(x_0) \right) 
    \\ & = - \frac2\eta \frac1N \sum_{x_0 \in \sL_\eta} \sum_{l \in N_0} \gamma_l u(x_0) \cdot V_l(x_0)
    \\ & = - \frac1N \sum_{x_0 \in \sL_\eta} u(x_0) \cdot \div_\eta V(x_0),
\end{align*}
where we use the edge consistency condition for $V$ and $\gamma_{-l} = \gamma_l$. This is the discrete analogue of the integration by parts identity $\int v \cdot \nabla u = - \int u \div v$.

\subsubsection{Discrete-to-continuous limit}

Fix a fundamental domain $\Pi \ni 0$ for the lattice $\sL$, for example the Voronoi cell $V(\sL)$. Then the canonical projection restricts to a bijection $\Pi \to \T$, up to sets of measure zero, and for any $\eta = 1/n$, we can cover $\T$ with rescaled translates of $\Pi$:
\[ \T = \bigsqcup_{x_0 \in \sL_\eta} (x_0 + \eta \Pi)\quad\text{up to sets of measure zero}. \]

We will prove convergence of functionals by comparing values at $x_0$ with averages over the corresponding cell $x_0 + \eta \Pi$.

Note that since $\T$ is compact, every continuous (scalar- or vector-valued) function $f$ on $\T$ is uniformly continuous, hence admits a modulus of continuity $\omega$:
\[ \norm{f(x) - f(y)} \leq \omega(\abs{x-y}), \quad \lim_{t\to0^+} \omega(t) = 0. \]

\begin{lemma}[Consistency of discrete derivatives]\label{lem:discderiv}
Let $v \in C^1(\T; \R^d)$ and let $\omega$ be a modulus of continuity for $\nabla v$.
Let $\Pi$ be a fundamental domain for $\T$ and let $x, x_0 \in \T$ such that $x \in x_0 + \eta \Pi$. Then for any $l \in N_0$,
\[ \abs{D^{(l)}_\eta v(x_0) - \nabla v(x) l} \leq C \omega(C \eta). \]
In particular, if $v$ is $C^{1,\alpha}$, the right-hand side is of the form $C [\nabla v]_{C^{0,\alpha}} \eta^\alpha$.
\end{lemma}
\begin{proof}
By the fundamental theorem of calculus,
\[ D^{(l)}_\eta v(x_0) = \frac1\eta [v(x_0 + \eta l) - v(x_0)] = \int_0^1 \nabla v(x_0+s\eta l)l\,ds. \]
Therefore
\[ \abs{D^{(l)}_\eta v(x_0) - \nabla v(x) l} \leq \abs{l}\int_0^1 \norm{\nabla v(x_0+s\eta l)-\nabla v(x)}\,ds \leq \abs{l}\int_0^1\omega(\abs{x_0+s\eta l-x})\,ds. \]
Now since $x \in x_0 + \eta \Pi$, $\abs{x-x_0} \leq \eta \diam(\Pi)$, hence $\abs{x_0+s\eta l - x} \leq \abs{x_0-x} + s \eta \abs{l} \leq C \eta$ for every $s\in[0,1]$. This gives the desired inequality.
\end{proof}

\paragraph{Discrete and continuous functionals.} 
We now consider general functionals of the form
\[ \cW_N(v) := \frac1N \sum_{x_0 \in \sL_\eta} W(\cD_\eta v(x_0)), \quad \cW(v) := \frac1{\abs\T} \int_\T W(\nabla v(x) \sL) \,dx, \quad v \in C^1(\T; \R^d), \]
where $W$ is a real-valued function on indexed sets $S = \{v_l\}_{l \in \sL} \in (\R^d)^\sL$, which depends only on the indices $l \in N_0$ and satisfies the Lipschitz estimate
\[ \abs{W(S_1) - W(S_2)} \leq \Lip(W) \max_{l \in N_0} \abs{v_l - w_l}, \quad S_1 = \{v_l\}_{l \in \sL}, \, S_2 = \{w_l\}_{l \in \sL} \subset \R^d. \]

\begin{remark}\label{rem:lipext}
Without loss of generality, we can restrict to Lipschitz closed balls
\[ D = \{ v \in C^1(\T; \R^d) \mid \norm{\nabla v}_\infty \leq \rho_0 \}, \]
under the assumption that $W$ instead admits a Lipschitz estimate
\[ \abs{W(S_1) - W(S_2)} \leq \Lip(W) \max_{l \in N_0} \abs{v_l - w_l} \]
only for $S_1 = \{v_l\}_{l \in \sL}, \, S_2 = \{w_l\}_{l \in \sL} \in (\R^d)^\sL$ such that $\displaystyle \sup_{l \in \sL\setminus\{0\}} \frac{\abs{v_l}}{\abs{l}}, \, \sup_{l \in \sL\setminus\{0\}} \frac{\abs{w_l}}{\abs{l}} \leq \rho_0$.

This follows by taking a globally Lipschitz extension $\tilde W$ of $W$ on $(\R^d)^{N_0}$ and noting that $\tilde \cW_N$ and $\tilde \cW$ coincide respectively with $\cW_N$ and $\cW$ on $D$. 

In our case, $D$ will correspond to Hypothesis~\ref{hyp1} on $u$, and the Lipschitz estimate will follow from Lemma~\ref{lem:costslip}.
\end{remark}

\begin{lemma}\label{lem:genfunclip}
For any $v, w \in \Lip(\T; \R^d)$, 
\[ \abs{\cW_N(w) - \cW_N(v)},\, \abs{\cW(w) - \cW(v)} \leq C \Lip(W) \norm{\nabla w - \nabla v}_\infty, \]
for a fixed constant $C$ independent of $N$.
\end{lemma}
\begin{proof}
In the discrete case, we can bound
\begin{align*} 
\abs{\cW_N(w) - \cW_N(v)} & \leq \frac1N \sum_{x_0 \in \sL_\eta} \abs{W(\cD_\eta w(x_0)) - W(\cD_\eta v(x_0))}
\\ & \leq \frac1N \sum_{x_0 \in \sL_\eta} C \Lip(W) \max_{l \in N_0} \abs{D^{(l)}_\eta w(x_0) - D^{(l)}_\eta v(x_0)},
\end{align*}
where
\[ \abs{D^{(l)}_\eta w(x_0) - D^{(l)}_\eta v(x_0)} = \frac1\eta \abs{(w-v)(x_0+\eta l) - (w-v)(x_0)} \leq \frac1\eta \Lip(w-v) (\eta \abs{l}) = \norm{\nabla w - \nabla v}_\infty \abs{l}. \]
The continuous case follows likewise:
\[ \abs{\cW(w) - \cW(v)} \leq \frac1{\abs{\T}} \int_\T C \Lip(W) \max_{l \in N_0} \abs{\nabla w(x) l - \nabla v(x) l} \,dx \leq C \Lip(W) \norm{\nabla w - \nabla v}_\infty. \qedhere \]
\end{proof}

\begin{lemma}\label{lem:gendisccont}
Let $v \in C^1(\T; \R^d)$ and let $\omega$ be a \emph{modulus of continuity} for $\nabla v$. 

Then for any $N = n^3$, $\eta = 1/n$, $n \in \NN$,
\[ \abs{\cW_N(v) - \cW(v)} \leq C \Lip(W) \omega(C \eta), \]
again for $C$ independent of $N$.
\end{lemma}
\begin{proof}
Recall that $\eta=1/n$ and $N=n^3$, hence $\eta^3=1/N$. Partitioning $\T$ into the cells $x_0+\eta\Pi$, up to sets of measure zero, each cell has volume $\abs{\T}/N$. Thus we can write
\begin{align*}
\cW_N(v) - \cW(v) & = \sum_{x_0 \in \sL_\eta} \left[ \frac1N W(\cD_\eta v(x_0)) - \frac1{\abs{\T}} \int_{x_0+\eta\Pi} W(\nabla v(x)\sL) \,dx \right]
\\ & = \frac1{\abs{\T}} \sum_{x_0 \in \sL_\eta} \int_{x_0+\eta\Pi} \left[ W(\cD_\eta v(x_0)) - W(\nabla v(x) \sL) \right] \,dx.
\end{align*}
Taking absolute values and applying Lemma~\ref{lem:discderiv} on each cell $x_0+\eta\Pi$, we conclude
\begin{align*}
\abs{\cW_N(v) - \cW(v)} & \leq \frac1{\abs{\T}} \sum_{x_0 \in \sL_\eta} \int_{x_0+\eta\Pi} C \Lip(W) \max_{l \in N_0} \abs{D^{(l)}_\eta v(x_0) - \nabla v(x) l} \,dx
\\ & \leq \frac1{\abs{\T}} \sum_{x_0 \in \sL_\eta} \int_{x_0+\eta\Pi} C \Lip(W) \cdot C \omega(C \eta) \,dx = C \Lip(W) \omega(C \eta). \qedhere
\end{align*}
\end{proof}

From these two estimates, we directly conclude

\begin{theorem}\label{thm:gencontlimit}
For any $v \in C^1(\T; \R^d)$,
\[ \lim_{N \to \infty} \cW_N(v) = \cW(v). \]
More generally, given any $v_N \to v$ converging in the Lipschitz topology,
\[ \lim_{N \to \infty} \cW_N(v_N) = \cW(v). \]
\end{theorem}
\begin{proof}
The first convergence statement follows directly from Lemma~\ref{lem:gendisccont}, since any $\nabla v \in C^0(\T; \R^{d\times3})$ admits a modulus of continuity $\omega$, for which $\omega(C \eta) \to 0$ as $\eta \to 0$. 

The second statement then follows by bounding $\abs{\cW_N(v_N) - \cW_N(v)} \leq C \Lip(W) \norm{\nabla v_N - \nabla v}_\infty$.
\end{proof}

\subsection{Finite-difference representation of discrete cost}\label{sect:findiff}

Now suppose $X$ is a diffeomorphism on $\T$ satisfying Hypothesis~\ref{hyp1}, inducing the quantizers $\cX_\eta := X(\sL_\eta)$ for $\eta = 1/n > 0$.

We define a diffeomorphism $\tilde X$ on $\R^3$ by taking an $\sL$-periodic lift $\tilde u \colon \R^3 \to \R^3$ of $X - \id_\T \in C^1(\T;\T)$ and setting $\tilde X = \id_{\R^3} + \tilde u$. Then for any $x \in \R^3$ and $l \in \sL$,
\[ \tilde X(x+l) = (x+l) + \tilde u(x+l) = x + l + \tilde u(x) = \tilde X(x) + l \]
and 
\[ X([x]) = [\tilde X(x)] = \{ \tilde X(x+l) \mid l \in \sL \} = \{ \tilde X(x) +l \mid l \in \sL \}. \]

\begin{lemma}
The locally finite, $\sL$-periodic set $\tilde \cX_\eta = \tilde X(\eta \sL) \subset \R^3$ is the lift of the finite set $\cX_\eta \subset \T$:
\[ \tilde \cX_\eta = \{ x \in \R^3 \mid [x] \in \cX_\eta\} = \bigcup \cX_\eta. \]
\end{lemma}
\begin{proof}
Observe directly that
\[ \bigcup \cX_\eta = \bigcup_{[\eta l_0] \in \sL_\eta} X([\eta l_0]) = \bigcup_{\eta l_0 \in \eta \sL} \{ \tilde X(\eta l_0+l) \mid l \in \sL\} = \tilde X(\eta \sL) \]
since $\eta \sL$ is closed under translation by any single $l \in \sL$.
\end{proof}

Therefore by Proposition~\ref{prop:vorlift}, the toroidal Voronoi cells of $\cX_\eta$ and the Euclidean Voronoi cells of $\tilde \cX_\eta$ have the same moments. We can then represent the Voronoi cell of any $\tilde x_0 = \tilde X(\eta l_0) \in \tilde \cX_\eta$ as
\[ \tilde x_0 + V(\tilde \cX_\eta - \tilde x_0) = \tilde x_0 + V(\eta \cD_\eta \tilde X(\eta l_0)), \]
where we observe that
\[ \frac{\tilde \cX_\eta - \tilde x_0}{\eta} = \Bigl\{ \frac{\tilde X(\eta l) - \tilde X(\eta l_0)}{\eta} \Bigr\}_{l \in \sL} = \Bigl\{ \frac{\tilde X(\eta (l_0+l)) - \tilde X(\eta l_0)}{\eta} \Bigr\}_{l \in \sL} = \cD_\eta \tilde X(\eta l_0). \]
Then by Proposition~\ref{prop:costinv}, the cost of the Voronoi cell of $X(\eta l_0) \in \cX_\eta$ is
\[ F(\eta \cD_\eta \tilde X(\eta l_0)) = \eta^5 F(\cD_\eta \tilde X(\eta l_0)), \]
and summing over $\cX_\eta$, we obtain
\[ \cF_N(X) = N^{2/3} \sum_{[\eta l_0] \in \sL_\eta} \underbrace{\eta^5}_{N^{-5/3}} F(\cD_\eta \tilde X(\eta l_0)) = \frac1N \sum_{[\eta l_0] \in \sL_\eta} F(\cD_\eta \tilde X(\eta l_0)). \]
Finally, in order to express the summand only in terms of the congruence class $[\eta l_0]$, we invoke the convention $X = \id_\T + u$ with $u \colon \T \to \R^3$, where evidently $\tilde u$ is an $\sL$-periodic lift of $u$ in the sense that $\tilde u(x) = u([x])$. The same applies to the directional derivatives
\[  D^{(l)}_\eta \tilde u(\eta l_0) = \frac{\tilde u(\eta (l_0+l)) - \tilde u(\eta l_0)}{\eta} = \frac{u([\eta (l_0+l)]) - u([\eta l_0])}{\eta} = D^{(l)}_\eta u([\eta l_0]), \]
so that
\[ D^{(l)}_\eta \tilde X(\eta l_0) = l + D^{(l)}_\eta \tilde u(\eta l_0) = l + D^{(l)}_\eta u([\eta l_0]) =: D^{(l)}_\eta X([\eta l_0]). \]
We have thus shown that
\begin{equation}
    \cF_N(X) = \frac1N \sum_{x_0 \in \sL_\eta} F(\cD_\eta X(x_0)).
\end{equation}
We observe in particular that for the unperturbed BCC lattice, the summand is identically equal to $F(I)$, and
\begin{equation}
    \cF_N(\id) = \cF(\id) = F(I) = \frac{19}{8}.
\end{equation}

Theorem~\ref{thm1} then follows directly from Theorem~\ref{thm:gencontlimit} in light of Remark~\ref{rem:lipext}: 

\begin{proof}[Proof of Theorem~\ref{thm1}]
With the choice of domain 
\[ D := \{ u \in C^1(\T; \R^3) \mid \norm{\nabla u}_\infty \leq \rho_0 \}, \]
the cost
\[ W(\{u_l\}_{l \in \sL}) := F(\{l + u_l\}_{l \in \sL}), \quad \{u_l\}_{l \in \sL} \in (\R^3)^\sL \]
satisfies the desired Lipschitz estimate by Lemma~\ref{lem:nbrstab} provided $\abs{u_l} \leq \rho_0 \abs{l}$. Taking a Lipschitz extension $\tilde W$ of $W$, we then conclude that for every $u \in D$,
\[ \cF_N(X) = \cW_N(u) = \tilde \cW_N(u) \to \tilde \cW(u) = \cW(u) = \cF(X). \qedhere \]
\end{proof}

This also provides the estimate 
\[ \abs{\cF_N(X) - \cF(X)} \leq C \omega(C \eta) \]
where $\omega$ is a modulus of continuity for $\nabla u$. 

\subsection{First variation of both costs}\label{sect:thmdiff}

As remarked above, under Hypothesis~\ref{hyp1}, both cell costs depend only on the perturbed neighbors associated to each $l \in N_0$.
The results of Section~\ref{sect:cellgrad} are then applicable in both settings, under the finite indexing set $N_0$. 

Suppose $\norm{\nabla u}_\infty < \rho_0$ so that $X$ lies in the interior of the admissible set of diffeomorphisms. Let $v \in C^1(\T; \R^3)$ be arbitrary and set $X_t = X + tv = \id + (u + tv)$. 

For each $x_0 \in \sL_\eta$ and $l \in N_0$, set 
\[ x_{x_0,l}(t) := D^{(l)}_\eta X_t(x_0) = \underbrace{D^{(l)}_\eta X(x_0)}_{x_{x_0,l}(0)} + t \underbrace{D^{(l)}_\eta v(x_0)}_{\Delta x_{x_0,l}}, 
\] 
and denote also $\mu_{x_0,l}(t) := \mu_l(\cD_\eta X_t(x_0))$ and likewise for $\lambda_{x_0,l}$.

In light of the above remark, we can apply Lemma~\ref{lem:momentc1} to obtain
\[ \frac{d}{dt} \cF_N(X_t) = \frac1N \sum_{x_0 \in \sL_\eta} \frac{d}{dt} F(\cD_\eta X_t(x_0)) = \frac1N \sum_{x_0 \in \sL_\eta} - \sum_{l \in N_0} \frac{12}{\abs{x_{x_0,l}}^2} R_{x_{x_0,l}} \mu_{x_0,l} \cdot \Delta x_{x_0,l}, \]
Now grouping the $(x_0, l)$ and $(x_0+\eta l, -l)$ terms together, and noting that 
\[ x_{x_0+\eta l,-l}(t) = D^{(-l)}_\eta X_t(x_0+\eta l) = - D^{(l)}_\eta X_t(x_0) = - x_{x_0,l}(t)
\] 
and likewise for $\dot X_t$, we obtain
\begin{equation}\label{eq:dfnlong} \frac{d}{dt} \cF_N(X_t) = - \frac1N \sum_{x_0 \in \sL_\eta} \sum_{l \in N_0} \frac{6}{\abs{x_{x_0,l}}^2} R_{x_{x_0,l}} \Bigl(\mu_{x_0,l} - \mu_{x_0+\eta l,-l} \Bigr) \cdot \Delta x_{x_0,l}. \end{equation}
Now observe that $\cD_\eta X(x_0+\eta l) = \cD_\eta X(x_0) - D^{(l)}_\eta X(x_0)$:
\[ D^{(l'-l)}_\eta X(x_0+\eta l) = \frac{X(x_0+\eta l')-X(x_0+\eta l)}{\eta} = D^{(l')}_\eta X(x_0) - D^{(l)}_\eta X(x_0), \quad l' \in \sL. \]
Therefore, Lemma~\ref{lem:ijreflect} and \eqref{eq:muijreflect} imply
\[ \mu_{x_0,l} - \mu_{x_0+\eta l,-l} = \mu_l(\cD_\eta X(x_0)) - \mu_{-l}\bigl(\cD_\eta X(x_0) - D^{(l)}_\eta X(x_0)\bigr) = \frac56 \lambda_l(\cD_\eta X(x_0)) D^{(l)}_\eta X(x_0).  \]
Plugging in \eqref{eq:dfnlong} then yields
\begin{align*}
    \frac{d}{dt} \cF_N(X_t) & = - \frac1N \sum_{x_0 \in \sL_\eta} \sum_{l \in N_0} \frac{5}{\abs{x_{x_0,l}}^2} R_{x_{x_0,l}} \left( \lambda_{x_0,l} x_{x_0,l} \right) \cdot \Delta x_{x_0,l}
    \\ & = \frac1N \sum_{x_0 \in \sL_\eta} \sum_{l \in N_0} \frac{5}{\abs{x_{x_0,l}}^2} \lambda_{x_0,l} x_{x_0,l}  \cdot \Delta x_{x_0,l}
    \\ & = \frac1N \sum_{x_0 \in \sL_\eta} \sum_{l \in N_0} \underbrace{\frac52 \lambda_{x_0,l} w(x_{x_0,l})}_{U_l(\cD_\eta X_t(x_0))} \cdot D^{(l)}_\eta v(x_0).
\end{align*}
In comparison, for the time derivative of the continuous cost, we apply Remark~\ref{rem:latcell} to obtain
\begin{align*} 
\frac{d}{dt} \cF(X_t) & = \frac1{\abs{\T}} \int_\T \frac{\partial}{\partial t} F(\underbrace{\nabla X_t(x)}_{\nabla X + t \nabla v}) \,dx
\\ &= \frac1{\abs{\T}} \int_\T \ip{\nabla F(\nabla X_t(x))}{\nabla v(x)} \,dx
\\ &= \frac1{\abs{\T}} \int_\T \sum_{l \in N_0} U_l(\nabla X_t(x)) \cdot (\nabla v(x) l) \,dx,
\end{align*}
where we apply \eqref{eq:gradf} for the gradient of $F$.

We have therefore shown that the first variations of the functionals $\cF_N$ and $\cF$ are
\begin{align*}
    \delta \cF_N(X)[v] & = \frac1N \sum_{x_0 \in \sL_\eta} \sum_{l \in N_0} U_l(\cD_\eta X(x_0)) \cdot D_\eta^{(l)} v(x_0); \\
    \delta \cF(X)[v] & = \frac1{\abs\T} \int_\T \sum_{l \in N_0} U_l(\nabla X(x)) \cdot (\nabla v(x) l) \,dx.
\end{align*}
This proves Theorem~\ref{thmdiff}.

We also deduce from Theorem~\ref{thm:gencontlimit} the convergence of the first variations as $N \to \infty$:

\begin{theorem}\label{thm1dt}
For any $X$ satisfying Hypothesis~\ref{hyp1} and $v \in C^1(\T; \R^3)$,
\[ \lim_{N \to \infty} \delta \cF_N(X)[v] = \delta \cF(X)[v]. \]
The same holds if $X$ and $v$ in the left-hand side are replaced with $X_N \to X$ and $v_N \to v$ converging in the Lipschitz topology.
\end{theorem}
\begin{proof}
Normalizing, we can assume wlog that $\norm{\nabla v}_\infty \leq \rho_0$. We consider functions $(u,v) \in C^1(\T; \R^6)$ taking values in
\[ D := \bigl\{ (u,v) \in C^1(\T; \R^6) \mid \norm{\nabla u}_\infty, \, \norm{\nabla v}_\infty \leq \rho_0 \bigr\}, \]
and define
\[ G(\{(u_l, v_l)\}_{l \in \sL}) := \sum_{l \in N_0} U_l(\{l+u_l\}_{l \in \sL}) \cdot v_l. \]
Lemmas~\ref{lem:costsbdd} and~\ref{lem:costslip} imply that $G$ satisfies the desired Lipschitz estimate provided $\abs{u_l}, \abs{v_l} \leq \rho_0 \abs{l}$:
\begin{align*}
 &  \abs{ G(\{(u_l, v_l)\}_{l \in \sL})  - G(\{(u'_l, v'_l)\}_{l \in \sL}) }
 \\ & \leq \sum_{l \in N_0} \left[ \abs{U_l(\{l'+u_{l'}\}_{l' \in \sL}) - U_l(l'+u'_{l'}\}_{l' \in \sL})} \abs{v_l} + \abs{U_l(\{l'+u'_{l'}\}_{l' \in \sL})} \abs{v_l - v'_l} \right]
 \\ & \leq C \rho_0 \sum_{l \in N_0} \abs{u_l - u'_l} + C \sum_{l \in N_0} \abs{v_l - v'_l} \leq C \sum_{l \in N_0} \abs{(u_l,v_l) - (u'_l,v'_l)}. 
\end{align*}
Taking a globally Lipschitz extension $W$ of $G$ and defining $\cW_N$ and $\cW$ accordingly, we conclude by Theorem~\ref{thm:gencontlimit} that
\[ \delta \cF_N(X)[v] = \cW_N((u,v)) \to \cW((u,v)) = \delta \cF(X)[v], \]
and the same carries over to convergent subsequences.
\end{proof}

Because the first variations have the structure of an inner product with $v$, one could improve the estimate Lemma~\ref{lem:gendisccont} and thus also arrive at a better estimate for the convergence of the functionals themselves, but we omit this discussion as it will not be necessary for our purposes.

\subsubsection{Functional gradients}

We now express the first variations of both functionals in the form of a gradient.
On the spaces of functions on the discrete graph $\sL_\eta$ resp.~the torus $\T$, we define the following normalized $L^2$ inner products
\[ \ip{u}{v}_{\ell^2(\sL_\eta)} := \frac1N \sum_{x_0 \in \sL_\eta} u(x_0) \cdot v(x_0) \quad \text{resp.} \quad \ip{u}{v}_{L^2(\T)} := \frac1{\abs\T} \int_\T u(x) \cdot v(x) \,dx. \]
For the discrete case, we first note that by Lemma~\ref{lem:ijreflect}, $U_l$ defines a discrete vector field on $\sL_\eta$:
\begin{equation}\label{eq:ulvector} U_{-l}(\cD_\eta X(x_0+\eta l))
 = \frac52 \lambda_l(\cD_\eta X(x_0)) w(-D^{(l)}_\eta X(x_0)) = - U_l(\cD_\eta X(x_0)). 
\end{equation}
In the quantization formulas below we take $\gamma_l=1$ for every $l\in N_0$. Therefore the discrete integration by parts identity gives:
\begin{equation}\label{eq:discintpart}
   \delta \cF_N(X)[v]  = \frac1N \sum_{x_0 \in \sL_\eta} \sum_{l \in N_0} U_l(\cD_\eta X(x_0)) \cdot D_\eta^{(l)} v(x_0) = - \frac1N \sum_{x_0 \in \sL_\eta} v(x_0) \cdot \div_\eta U(\cD_\eta X(x_0)).
\end{equation}
In the continuous case, assuming $u \in C^2(\T; \R^3)$, we can instead apply integration by parts along each directional derivative $l \cdot \nabla$ , noting the $l$-periodicity of $U_l(\nabla X(x)) \mathbin{\cdot} v(x)$:
\begin{align*}
   \delta \cF(X)[v] & = \frac1{\abs\T} \sum_{l \in N_0} \int_\T U_l(\nabla X(x)) \cdot (\nabla v(x) l) \,dx = - \frac1{\abs\T} \sum_{l \in N_0} \int_\T (l \cdot \nabla) U_l(\nabla X(x)) \cdot v(x) \,dx.
\end{align*}
Equivalently, 
\begin{align*}
   \delta \cF(X)[v] & = \frac1{\abs\T} \int_\T \ip{\nabla F(\nabla X(x))}{\nabla v(x)} \,dx
   \\ & = \frac1{\abs\T} \int_\T \sum_{i,j=1}^3 \underbrace{(\nabla F(\nabla X(x)))_{ij}}_{(\partial_{M_{ij}} F)(\nabla X(x))} \underbrace{(\nabla v(x))_{ij}}_{\partial_j v_i(x)}  \,dx
   \\ & = - \frac1{\abs\T} \int_\T \sum_{i,j=1}^3 \partial_j (\nabla F(\nabla X(x)))_{ij} v_i(x)  \,dx
   \\ & = - \frac1{\abs\T} \int_\T \div_x \nabla F(\nabla X(x)) \cdot v(x)  \,dx
\end{align*}
where we define the divergence of matrix-valued functions rowwise:
\[ (\div_x A)_i := \sum_{j=1}^3 \partial_j A_{ij}. \]
Note that this is consistent with the formula \eqref{eq:gradf} for $\nabla F$:
\[ (\div_x \nabla F(\nabla X(x)))_i = \div_x\bigg(\sum_{l \in N_0} U_l(\nabla X(x)) l^\top\bigg)_i = \sum_{j=1}^3 \partial_j \sum_{l \in N_0} U_l(\nabla X(x))_i l_j = \sum_{l \in N_0} (l \cdot \nabla) U_l(\nabla X(x))_i. \]
Therefore, the (formal) functional gradients of $\cF_N$ resp.~$\cF$ are given by the discrete resp.~continuous divergences
\begin{align}\label{eq:fnfuncgrad}
    (\nabla_{\ell^2(\sL_\eta)} \cF_N(X))(x_0) & = - \div_\eta U(\cD_\eta X(x_0)) = - \sum_{l \in N_0} D^{(l)}_\eta U_l(\cD_\eta X(x_0-\eta l));
     \\ \label{eq:ffuncgrad}
    (\nabla_{L^2(\T)} \cF(X))(x) & = - \div_x \nabla F(\nabla X(x)) = - \sum_{l \in N_0} (l \cdot \nabla) U_l(\nabla X(x)). 
\end{align}

This equivalent formulation also makes it apparent that both first variations vanish at $X = \id$:

\begin{lemma}\label{lem:firstvarzero}
The unperturbed BCC lattice is a stationary point of both functionals:
\[ \delta \cF_N(\id) = \delta \cF(\id) = 0. \]
\end{lemma}
\begin{proof}
In both settings, $U_l(\cD_\eta X(x_0)) \equiv U_l(I) \equiv U_l(\nabla X(x))$ for $X = \id$.

The discrete case follows from the fact that
\[ \div_\eta U(\cD_\eta \id(x_0)) = \sum_{l \in N_0} D^{(l)}_\eta U_l(\cD_\eta \id(x_0-\eta l)) = \sum_{l \in N_0}  \frac{U_l(I) - U_l(I)}{\eta} = 0. \]
For the continuous case, note that since $U_l(\nabla \id(x)) = U_l(I)$ is constant in $x$, its derivatives vanish identically. Thus $\delta \cF(\id)$ vanishes on all $v \in C^2$, thus on all $v \in C^1$ by density.
\end{proof}

This is exactly what is expected by the stationarity of lattice quantizers, and this formulation makes it clear that the stationarity is a consequence of the periodicity of the tessellation.

\begin{remark}\label{rem:divclosedform}
Lemma~\ref{lem:dui} provides a closed-form formula for the continuous gradient \eqref{eq:ffuncgrad}.

Fix $l \in N_0$, and consider the translations $X_t(x) := X(x+tl)$, so that $\left. \dot X_t(x) \right|_{t=0} = (l \cdot \nabla) X(x)$, and
\begin{align*} 
(l \cdot \nabla) U_l(\nabla X(x)) & = \left. \frac{d}{dt} U_l(\nabla X_t(x)) \right|_{t=0} \\ & = A_l(\nabla X(x)) (\nabla \dot X(x)l) + \sum_{l' \sim l}  B_{ll'}(\nabla X(x)) (\nabla \dot X(x)l')
\\ & = \Bigl[ A_l(\nabla X(x)) (l \cdot \nabla)^2 + \sum_{l' \sim l}  B_{ll'}(\nabla X(x)) (l'\cdot \nabla) (l \cdot \nabla) \Bigr] X(x).
\end{align*}
Summing over all $l \in N_0$ gives the nonlinear operator
\[ \sum_{l \in N_0} (l \cdot \nabla) U_l(\nabla X) = \left[ \sum_{l \in N_0} A_l(\nabla X) (l \cdot \nabla)^2 + \sum_{l \sim l' \in N_0} \frac12(B_{ll'}(\nabla X) + B_{l'l}(\nabla X)) (l \cdot \nabla) (l' \cdot \nabla) \right] X.   \]
However, this expression will not be necessary for the perturbative analysis we will pursue in the next section.
\end{remark}

\subsection{Second variation of both costs}\label{sect:uldiff}

Using Lemma~\ref{lem:dui}, we can further differentiate $U_l$ to determine the second variations of both functionals:

\begin{lemma}\label{lem:2ndvars}
Suppose $\norm{\nabla u}_\infty < \rho_0$ so that $X$ lies in the interior of the admissible set of diffeomorphisms.

Then the second variations of $\cF_N$ and $\cF$ at $X$ in the direction $v \in C^1(\T; \R^3)$ are given by
\begin{align*}
\delta^2 \cF_N(X)[v,v] & = \frac1N \sum_{x_0 \in \sL_\eta} \sum_{l \in N_0} D^{(l)}_\eta v(x_0) \cdot \left[ A_l(\cD_\eta X(x_0)) D^{(l)}_\eta v(x_0) + \sum_{l' \sim l} B_{ll'}(\cD_\eta X(x_0)) D^{(l')}_\eta v(x_0) \right] ; \\
\\ & = -\frac{2}{N\eta} \sum_{x_0 \in \sL_\eta} v(x_0) \cdot \sum_{l \in N_0} \left[ A_l(\cD_\eta X(x_0)) D^{(l)}_\eta v(x_0) + \sum_{l' \sim l} B_{ll'}(\cD_\eta X(x_0)) D^{(l')}_\eta v(x_0) \right]; \\
\delta^2 \cF(X)[v,v] & = \frac1{\abs\T} \int_\T \sum_{l \in N_0} \nabla v(x) l  \cdot \left[ A_l(\nabla X(x)) \nabla v(x) l + \sum_{l' \sim l} B_{ll'}(\nabla X(x)) \nabla v(x) l' \right] \,dx,
\end{align*}
where $A_l(S)$ and $B_{ll'}(S)$ are defined as in Lemma~\ref{lem:dui}.
\end{lemma}
\begin{proof}
Again let $v \in C^1(\T; \R^3)$ be arbitrary and set $X_t = X + tv = \id + (u + tv)$ for $t$ sufficiently small as to preserve Hypothesis~\ref{hyp1}. In light of the first variation calculations, we have
\begin{align*} 
\delta^2 \cF_N(X_t)[v,v] = \frac{d^2}{dt^2} \cF_N(X_t) & = \frac1N \sum_{x_0 \in \sL_\eta} \sum_{l \in N_0} D^{(l)}_\eta v(x_0) \cdot \frac{d}{dt} U_l(\cD_\eta X_t(x_0)); \\
\delta^2 \cF(X_t)[v,v] = \frac{d^2}{dt^2} \cF(X_t) & = \frac1{\abs\T} \int_\T \sum_{l \in N_0} (\nabla \dot X_t(x) l) \cdot \frac{d}{dt} U_l(\nabla X_t(x))  \,dx.
\end{align*}
The first expressions then follow from Lemma~\ref{lem:dui}, with the observation that both point sets $\cD_\eta X_t(x_0)$ and $\nabla X_t(x) \sL$ maintain the same neighborhood relations of the BCC lattice for $t$ small:
\begin{align*} 
\frac{d}{dt} U_l(\cD_\eta X_t(x_0)) & = A_l(\cD_\eta X_t(x_0)) D^{(l)}_\eta v(x_0) + \sum_{l' \sim l} B_{ll'}(\cD_\eta X_t(x_0)) D^{(l')}_\eta v(x_0); \\
\frac{d}{dt} U_l(\nabla X_t(x)) & = A_l(\nabla X_t(x)) \nabla v(x) l + \sum_{l' \sim l} B_{ll'}(\nabla X_t(x)) \nabla v(x) l'. 
\end{align*}
The alternative representation for the discrete second variation follows instead by applying the discrete integration by parts identity
\[ \frac{d}{dt} \cF_N(X_t) = \frac1N \sum_{x_0 \in \sL_\eta} \sum_{l \in N_0} D^{(l)}_\eta v(x_0) \cdot U_l(\cD_\eta X_t(x_0)) = \frac1N \sum_{x_0 \in \sL_\eta} -\frac2\eta v(x_0) \cdot \sum_{l \in N_0} U_l(\cD_\eta X_t(x_0)) \]
and differentiating each $U_l(\cD_\eta X_t(x_0))$ likewise.
\end{proof}

Arguing analogously to the first variation, one could show the differentiability of $A_l(S)$ and $B_{ll'}(S)$ and thus conclude also the convergence of second variations as $N \to \infty$.
This will not be necessary for our purposes in the next section: we will instead restrict to $X = \id$, in which case the matrices $A_l(\sL)$ and $B_{ll'}(\sL)$ will be the same in both settings.

\section{Perturbative expansion around the BCC lattice}\label{sect:thm2}

Having obtained the above formulas for the first and second variations of the functionals $\cF_N$ and $\cF$, in this section we compute the explicit perturbative expansion of both functionals around $X = \id$.

\begin{theorem}
    Around $\varepsilon=0$, the functionals $\cF_N$ and $\cF$ admit the perturbative expansions
    \begin{align*}
        \cF_N(\id+\varepsilon u) & = \underbrace{\cF_N(\id)}_{\abs{}} + \frac12 \varepsilon^2 G_N(u) + O(\varepsilon^3), \\
        \cF(\id+\varepsilon u) & = \ \overbrace{\cF(\id)} + \frac12 \varepsilon^2 G(u) + O(\varepsilon^3);
    \end{align*}
    where
    \begin{align}
        G_N(u) & := \frac1N \sum_{x_0 \in \sL_\eta} \sum_{l \in N_0} D^{(l)}_\eta u(x_0) \cdot \Bigl( A_l D^{(l)}_\eta u(x_0) + \sum_{l' \sim l} B_{ll'} D^{(l')}_\eta u(x_0)\Bigr); \\ 
        G(u) & := \frac1{\abs\T} \int_\T \sum_{l \in N_0} (\nabla u(x) l) \cdot \Bigl( A_l \nabla u(x) l + \sum_{l' \sim l} B_{ll'} \nabla u(x) l'\Bigr) \,dx.
    \end{align}
    Moreover, $G_N(u) \to G(u)$ as $N \to \infty$, and both admit the equivalent representations
    \begin{align}
        G_N(u) & = \frac1N \sum_{x_0 \in \sL_\eta} \sum_{l \in N_0} D^{(l)}_\eta u(x_0)^\top C_l D^{(l)}_\eta u(x_0); \\ 
        G(u) & = \frac1{\abs\T} \int_\T \sum_{l \in N_0} (\nabla u(x) l)^\top C_l (\nabla u(x) l) \,dx
    \end{align}
    for $C_l := A_l + \sum_{l' \sim l} B_{l'l}$.
\end{theorem}

This equivalence between two forms of $G(u)$ implies, not that the integrands themselves are equal, but that they differ by a null-Lagrangian term whose integral is identically zero:

\begin{lemma}\label{lem:nulllag}
The quadratic form
\[ D(M) := \underbrace{\sum_{l \in N_0} (Ml) \cdot \bigl(A_l Ml + \sum_{l' \sim l} B_{ll'} M l' \bigr)}_{\nabla^2F(I)[M,M]} - \sum_{l \in N_0} (Ml) \cdot (C_l Ml), \quad M \in \R^{3\times3}\]
is \emph{null-Lagrangian}: for any $u \in C^2(\T; \R^3)$,
\[ \int_\T D(\nabla u(x)) \,dx = 0. \]
\end{lemma}

In Section~\ref{sect:clcalc}, we exploit the symmetries of the BCC lattice to reduce each matrix $C_l$ to the form
\[ C_l = \frac{\alpha_l}{\abs{l}^2} I + \frac{\beta_l}{\abs{l}^4} l l^\top, \]
with $\alpha_l, \, \beta_l \in \R$ constant along $G$-orbits. With minimal numerical calculations, we then obtain
\[ \alpha_l = \left\{ \begin{matrix} -\frac{1}{24}, & l \in N_{ax}; \\ -\frac{15}{64}, & l \in N_{diag};\end{matrix} \right. \qquad \beta_l = \left\{ \begin{matrix} \frac{25}{24}, & l \in N_{ax}; \\ \frac{123}{64}, & l \in N_{diag};\end{matrix} \right. \]
In Section~\ref{sect:hesscalc}, we then evaluate the sum
\[ H(M) := \sum_{l \in N_0} (M l)^\top C_l (M l), \]
showing that it decomposes along $G$-invariant subspaces as follows:
\[ H(M) = \frac{13}6 \Tr(M)^2 + \frac{11}8 \norm{M_2}^2 + \frac{65}{24} \norm{M_3}^2 - \frac{17}{24} \norm{M_4}^2. \]
While this quadratic form is not positive definite, the addition of a null-Lagrangian term $\alpha(\Tr(M)^2 - \Tr(M^2))$ does yield a positive definite form:
\[ H_\alpha(M) = \frac{103}{36} \Tr(M)^2 + \frac13 \norm{M_2}^2 + \frac53 \norm{M_3}^2 + \frac13 \norm{M_4}^2. \]
This representation allows us to conclude the positivity of the second variation at identity:

\begin{theorem}\label{thm2}
    The quadratic form $G(u) = \delta^2 \cF(\id)[u,u]$ satisfies
    \[ \frac13 \frac1{\abs\T} \int_\T \norm{\nabla u}^2 \,dx \leq G(u) \leq 9 \frac1{\abs\T} \int_\T \norm{\nabla u}^2 \,dx. \]
\end{theorem}

We also provide in Section~\ref{sect:altcalc} an alternative explicit calculation of $\nabla^2 F(I)$, providing an independent verification of the above results.
We obtain specifically that
\begin{equation}\label{eq:dfinal} D(M) := \nabla^2 F(I)[M,M] - H(M) = \frac{14}{3} \left( \Tr(M)^2 - \Tr(M^2) \right), \end{equation}
which is a multiple of the same null-Lagrangian term added to $H$, confirming Lemma~\ref{lem:nulllag}.

\subsection{Second variations at identity}

We consider small perturbations of the form $X_\varepsilon = \id + \varepsilon u$ for $\norm{\nabla u} \leq \rho_0$. Then around $\varepsilon = 0$, we can expand both costs as
\begin{align*}
    \cF_N(X_\varepsilon) & = \cF_N(\id) + \varepsilon \underbrace{\left. \frac{d}{d\varepsilon} \cF_N(X_\varepsilon)\right|_{\varepsilon = 0}}_{\delta \cF_N(\id)[u]} + \frac12 \varepsilon^2 \underbrace{\left. \frac{d^2}{d\varepsilon^2} \cF_N(X_\varepsilon)\right|_{\varepsilon = 0}}_{\delta^2 \cF_N(\id)[u,u]} + O(\varepsilon^3); \\ 
    \cF(X_\varepsilon) & = \cF(\id) + \varepsilon \underbrace{\left. \frac{d}{d\varepsilon} \cF(X_\varepsilon)\right|_{\varepsilon = 0}}_{\delta \cF(\id)[u]}  + \frac12 \varepsilon^2 \underbrace{\left. \frac{d^2}{d\varepsilon^2} \cF(X_\varepsilon)\right|_{\varepsilon = 0}}_{\delta^2 \cF(\id)[u,u]} + O(\varepsilon^3).
\end{align*}
Note that both $\cF_N(\id) = \cF(\id) = F(I) = \frac{19}{8}$, and that both first-order terms vanish by Lemma~\ref{lem:firstvarzero}, as expected by the stationarity of lattice quantizers. 

Meanwhile, by Lemma~\ref{lem:2ndvars}, the second variations of $\cF_N$ resp. $\cF$ at $X = \id$ are given by
\begin{align*}
    G_N(u) := \delta^2 \cF_N(\id)[u,u]
     & = \frac1N \sum_{x_0 \in \sL_\eta} \sum_{l \in N_0} D^{(l)}_\eta u(x_0) \cdot \bigl(A_l D^{(l)}_\eta u(x_0) + \sum_{l' \sim l} B_{ll'} D^{(l')}_\eta u(x_0) \bigr); \\
    G(u) := \delta^2 \cF(\id)[u,u]
    & = \frac1{\abs{\T}} \int_\T \sum_{l \in N_0} (\nabla u(x)l) \cdot \bigl(A_l \nabla u(x)l + \sum_{l' \sim l} B_{ll'} \nabla u(x) l' \bigr) \,dx.
\end{align*}
Here note that the matrices $A_l = A_l(\sL)$ and $B_{ll'} = B_{ll'}(\sL)$ are identical in both settings.
From these expressions, the convergence $G_N(u) \to G(u)$ follows for $u \in C^1$, arguing analogously to Theorem~\ref{thm1dt} based on Theorem~\ref{thm:gencontlimit}.

\paragraph{Alternate formulation.} 
Now in the discrete case, we also have the alternate representation
\begin{align*}
    G_N(u) 
    & = -\frac{2}{N\eta} \sum_{x_0 \in \sL_\eta}  u(x_0) \cdot \sum_{l \in N_0}  \bigl(A_l D^{(l)}_\eta u(x_0) + \sum_{l' \sim l} B_{ll'} D^{(l')}_\eta u(x_0) \bigr)
    \\ & = -\frac{2}{N\eta} \sum_{x_0 \in \sL_\eta}  u(x_0) \cdot \sum_{l \in N_0}  \bigl(A_l + \sum_{l' \sim l} B_{l'l} \bigr) D^{(l)}_\eta u(x_0),
\end{align*}
where in the last line we collect all the $D^{(l)}_\eta u(x_0)$ terms, noting that the relation $l' \sim l$ is symmetric. Setting then
\begin{equation}\label{eq:cldef}
    C_l := A_l + \sum_{l' \sim l} B_{l'l},
\end{equation}
and noting that $C_{-l} = C_l$ (which we will obtain explicitly in Section~\ref{sect:clcalc}), we can apply discrete integration by parts again to conclude
\[ G_N(u) = \frac1N \sum_{x_0 \in \sL_\eta} \sum_{l \in N_0} D^{(l)}_\eta u(x_0) \cdot C_l D^{(l)}_\eta u(x_0). \]
But since $G_N(u) \to G(u)$, $G(u)$ itself must be of the form
\[ G(u) = \frac1{\abs\T} \int_\T \sum_{l \in N_0} (\nabla u(x) l)^\top  C_l (\nabla u(x) l) \,dx. \]

We have thus shown Theorem~\ref{thmpert}. It remains to evaluate the matrices $C_l$, exploiting the $G$-invariance of the BCC lattice, and conclude that the continuous Hessian satisfies $G(u)\geq0$, with equality if and only if $u$ is constant. In particular, $G$ is positive definite on the space of displacements with zero mean.

\begin{remark}[Dynamical perspective]
Taking the functional gradients of $\frac12G_N(u)$ and $\frac12G(u)$, the quadratic forms induce the following linearized gradient flows:
\begin{align*} 
\dot u^{(N)}_t(x_0) = \frac2\eta \sum_{l \in N_0} C_l D^{(l)}_\eta u^{(N)}_t(x_0) & = \sum_{l \in N_0} C_l \frac{u^{(N)}_t(x_0+\eta l) - 2 u^{(N)}_t(x_0) + u^{(N)}_t(x_0-\eta l)}{\eta^2}; \\
\dot u_t(x) = \div_x \sum_{l \in N_0} C_l (\nabla u_t(x) l) l^\top & = \sum_{l \in N_0} C_l (l \cdot \nabla)^2 u_t(x).
\end{align*}
The numerical calculations in this section will imply the ellipticity of the continuous differential operator $L u =  \sum_{l \in N_0} C_l (l \cdot \nabla)^2 u$.
\end{remark}

\subsection{Calculation of the matrices $C_l$}\label{sect:clcalc}

Recalling the formulas for $A_l$ and $B_{ll'}$ given in Lemma~\ref{lem:dui}, as well as \eqref{eq:latcell}, we can write
\[ C_l = \frac{5}{\abs{l}^2} \left[ \lambda_l I - \frac{5 \lambda_l}{\abs{l}^2} l l^\top + \sum_{l' \sim l} \frac{w(l) - w(l')}{\abs{w(l)-w(l')}} \xi_{ll'}^\top \right] R_l, \]
noting that
\[ B_{l'l} = - \frac{5}{\abs{l}^2} \frac{2 l'}{\abs{l'}^2} \frac{\xi_{ll'}^\top}{\abs{w(l)-w(l')}} R_l. \]
Now observe that since $A_l$ and $B_{ll'}$ are $G$-equivariant by Remark~\ref{rem:absym}, so is $C_l$:
\[ C_{gl} = A_{gl} + \sum_{gl' \sim gl} B_{gl',gl} = g A_l g^\top + \sum_{l' \sim l} g B_{l'l} g^\top = g C_l g^\top. \] 
Therefore, by Lemma~\ref{lem:stabfix}, we can represent $C_l$ as
\[ C_l = \frac{\alpha_l}{\abs{l}^2} I + \frac{\beta_l}{\abs{l}^4} l l^\top, \]
where the coefficients $\alpha_l, \beta_l$ are constant along $G$-orbits.

To compute the coefficients $\alpha_l$ and $\beta_l$, it then suffices to test $C_l$ against $I$ and $\frac{ll^\top}{\abs{l}^2}$. We compute
\begin{align*}
    \frac{3\alpha_l+\beta_l}{\abs{l}^2} = \ip{C_l}{I} & = \frac{5}{\abs{l}^2} \left[ \lambda_l \Tr(R_l) - \frac{5 \lambda_l}{\abs{l}^2} l \cdot (R_l l) + \sum_{l' \sim l} \frac{w(l) - w(l')}{\abs{w(l)-w(l')}} \cdot R_l \xi_{ll'} \right]; \\
    \frac{\alpha_l+\beta_l}{\abs{l}^2} = \ip{C_l}{\frac{ll^\top}{\abs{l}^2}} & = \frac{5}{\abs{l}^2} \left[ \lambda_l \frac{l \cdot R_l l}{\abs{l}^2} - \frac{5 \lambda_l}{\abs{l}^2} (l \cdot l) \frac{l \cdot R_l l}{\abs{l}^2} + \sum_{l' \sim l} l \cdot \frac{w(l) - w(l')}{\abs{w(l)-w(l')}} \frac{l}{\abs{l}^2} \cdot R_l \xi_{ll'} \right].
\end{align*}
Since $R_l l = -l$ and $\Tr(R_l) = \Tr(I) - \frac{2}{\abs{l}^2} \Tr(ll^\top) = 1$, the $\lambda_l$ terms simplify to $6\lambda_l$ and $4 \lambda_l$ respectively. 

To evaluate the sums over $l' \sim l$, we apply \eqref{eq:muijcdotlij}:
\begin{align*} 
\qquad R_l \xi_{ll'} = \xi_{ll'} - l (w(l) \cdot \xi_{ll'}) = \xi_{ll'} - \beta l
\implies & \quad (w(l)-w(l')) \cdot R_l \xi_{ll'} = - \beta l \cdot (w(l)-w(l')); 
\\ \frac{l}{\abs{l}^2} \cdot R_l \xi_{ll'} = R_l \frac{l}{\abs{l}^2} \cdot \xi_{ll'} & = - \frac12 w(l) \cdot \xi_{ll'} = -\frac\beta2.
\end{align*}
We then obtain
\begin{align*}
    3\alpha_l+\beta_l  & = 5 \left[ 6 \lambda_l - \beta l \cdot\sum_{l' \sim l} \frac{w(l) - w(l')}{\abs{w(l)-w(l')}} \right]; \\
    \alpha_l+\beta_l  & = 5 \left[ 4\lambda_l - \frac\beta2 l \cdot  \sum_{l' \sim l} \frac{w(l) - w(l')}{\abs{w(l)-w(l')}} \right].
\end{align*}
Thus for both terms, it suffices to evaluate the coefficients 
\[ \Gamma_l = l \cdot  \sum_{l' \sim l} \frac{w(l) - w(l')}{\abs{w(l)-w(l')}}, \]
so that
\begin{align*}
    \alpha_l & = 5 \left(  \lambda_l - \frac\beta4 \Gamma_l \right); \quad 
    \beta_l = 5 \left( 3\lambda_l - \frac\beta4 \Gamma_l \right) = \alpha_l + 10 \lambda_l.
\end{align*}
The only coefficients left to evaluate for the different elements of $N_0$ are $\Gamma_l$, which we break up into axial and diagonal cases:

\begin{itemize}
\item \textbf{$l \in N_{ax}$:} Each $l \in N_{ax}$ is adjacent to four $l' \in N_{diag}$, say $l = (2,0,0)$ for which $l' = (1,\pm1,\pm1)$. We then have $w(l) = (1,0,0)$ and $w(l') = (\frac23, \pm\frac23, \pm\frac23)$ thus
\[ l \cdot \frac{w(l)-w(l')}{\abs{w(l)-w(l')}} = (2,0,0) \cdot (\frac13, \frac23, \frac23) = \frac23. \]
Summing over $l' \sim l$ yields $\Gamma_l = \frac83$.

\item \textbf{$l \in N_{diag}$:} Each $l \in N_{diag}$ is adjacent to three $l'_a \in N_{ax}$ and three $l'_d \in N_{diag}$, say $l = (1,1,1)$ for which $l'_a = 2e_i$ and $l'_d = l-2e_i$.

For e.g. $l'_a = (2,0,0)$, we again have $w(l) = (\frac23, \frac23, \frac23)$ and $w(l'_a) = (1,0,0)$, thus
\[ l \cdot \frac{w(l)-w(l')}{\abs{w(l)-w(l')}} = (1,1,1) \cdot (-\frac13, \frac23, \frac23) = 1. \]
For e.g. $l'_d = (1,1,-1)$, we instead have $w(l'_d) = (\frac23,\frac23, -\frac23)$ hence again
\[ l \cdot \frac{w(l)-w(l')}{\abs{w(l)-w(l')}} = (1,1,1) \cdot (0,0,1) = 1. \]
Summing over $l' \sim l$ (noting that the inner products are preserved along the action of $G_l$) yields $\Gamma_l = 6$.
\end{itemize}

\paragraph{Summary: explicit values of $\alpha_l$ and $\beta_l$.} Recalling now that $\lambda_{ax} = \frac{13}{120}$, $\lambda_{diag} = \frac{69}{320}$ and $\beta = \frac{7}{40}$, we finally obtain
\begin{align*}
    \alpha_{ax} = 5 \left( \frac{13}{120} - \frac14 \frac{7}{40} \frac83  \right) = \frac{13-14}{24} = - \frac{1}{24}, \qquad & \beta_{ax} = \alpha_{ax} + 10 \lambda_{ax} = \frac{25}{24}; \\
    \alpha_{diag} = 5 \left( \frac{69}{320} - \frac14 \frac{7}{40} 6 \right) = \frac{69-84}{64} = - \frac{15}{64}, \qquad &\beta_{diag} = \alpha_{diag} + 10 \lambda_{diag} = \frac{123}{64}.
\end{align*}

\subsection{Calculation of quadratic forms and null-Lagrangian terms}\label{sect:hesscalc}

Since the coefficients $\beta_l$ are positive but $\alpha_l$ are negative, it does not follow directly that the quadratic form
\[ H(M) = \sum_{l \in N_0} (M l)^\top  C_l (M l) = \sum_{l \in N_0} \alpha_l \frac{\abs{M l}^2}{\abs{l}^2} + \sum_{l \in N_0} \beta_l \frac{(l^\top Ml)^2}{\abs{l}^4} \]
is positive definite. Recall however the following identities obtained in Section~\ref{sect:invforms}:
\begin{align*}
    \frac{1}{\abs{N_{ax}}} \sum_{l \in N_{ax}} \frac{\abs{M l}^2}{\abs{l}^2} = \frac{1}{\abs{N_{diag}}} \sum_{l \in N_{diag}} \frac{\abs{M l}^2}{\abs{l}^2} & = \frac13 \norm{M}^2; \\ 
    \frac{1}{\abs{N_{ax}}} \sum_{l \in N_{ax}} \frac{(l^\top M l)^2}{\abs{l}^4} & = \frac19 \Tr(M)^2 + \frac13 \norm{M_2}^2;
    \\
     \frac{1}{\abs{N_{diag}}} \sum_{l \in N_{diag}} \frac{(l^\top M l)^2}{\abs{l}^4} &= \frac19 \Tr(M)^2 + \frac29 \norm{M_3}^2.
\end{align*}
We therefore obtain
\begin{align}\nonumber
    H(M) & = \frac{6 \alpha_{ax} + 8 \alpha_{diag}}{3} \norm{M}^2 + \frac{6\beta_{ax}+8\beta_{diag}}{9} \Tr(M)^2 + \frac{6\beta_{ax}}3 \norm{M_2}^2 + \frac{16\beta_{diag}}{9} \norm{M_3}^2
    \\ \nonumber & = -\frac{17}{24} \norm{M}^2 + \frac{173}{72} \Tr(M)^2 + \frac{50}{24} \norm{M_2}^2 + \underbrace{\frac{123}{36}}_{\frac{82}{24}} \norm{M_3}^2
    \\ \label{eq:hfinal} & = \underbrace{\frac{156}{72}}_{\frac{13}{6}} \Tr(M)^2 + \underbrace{\frac{33}{24}}_{\frac{11}{8}} \norm{M_2}^2 + \frac{65}{24} \norm{M_3}^2 - \frac{17}{24} \norm{M_4}^2,
\end{align}
where in the last line we plug in $\norm{M}^2 = \frac13 \Tr(M)^2 + \norm{M_2}^2 + \norm{M_3}^2 + \norm{M_4}^2$. Therefore $H$ is positive definite on symmetric matrices but negative definite on antisymmetric matrices. 

But the latter is only a local defect which can be corrected by the addition of a null-Lagrangian term, since antisymmetric matrices are infinitesimal generators of rotations, which preserve the quantization error. Precisely, we will add the following term:

\begin{lemma}\label{lem:q0nulllag}
The quadratic form
\[ Q_0(M) := \Tr(M)^2-\Tr(M^2) \] 
is null-Lagrangian: for any $u \in H^1(\T; \R^3)$,
\[ \int_\T Q_0(\nabla u(x)) \,dx = 0. \]
\end{lemma}
\begin{proof}
Now suppose $u \in C^2(\T; \R^3)$.
Integrating by parts, we can write
\[ \int_\T Q_0(\nabla u) \,dx = -\int_\T \ip{\nabla u}{\nabla u^\top - \Tr(\nabla u) I} \,dx = +\int_\T u \cdot \div_x (\nabla u^\top - \Tr(\nabla u) I) \,dx,  \]
where divergence is taken rowwise:
\begin{align*}
    \div( \nabla u^\top - \Tr(\nabla u)I)_i 
    & = \partial_j ( \nabla u^\top - \Tr(\nabla u)I)_{ij} 
    \\ & = \partial_j (\partial_i u_j) - \partial_j (\partial_k u_k) \delta_{ij}
    \\ & = \partial_i (\partial_j u_j) - \partial_i (\partial_k u_k) = 0. 
\end{align*} 
In other words, $\div(\nabla u^\top) = \nabla (\div u) = \nabla \Tr(\nabla u) = \div (\Tr(\nabla u) I)$. 

We have thus shown that $\int_\T Q_0(\nabla u) \,dx = 0$ for $u \in C^2$. The extension to $H^1$ follows by the density of $C^2$ and the fact that the quadratic form $Q_0$ is uniformly controlled by multiples of $\norm{\cdot}^2$.
\end{proof}

\begin{lemma}\label{lem:halpha}
For the explicit choice $\alpha = \frac{25}{24}$, the quadratic form $H_\alpha(M) := H(M) + \alpha Q_0(M)$ is given by
\begin{equation}\label{eq:halpha} H_\alpha(M) = \frac{103}{36} \Tr(M)^2 + \frac13 \norm{M_2}^2 + \frac53 \norm{M_3}^2 + \frac13 \norm{M_4}^2 \end{equation}
where we denote by $M_i := \pi_i M$ the projection of $M$ to the invariant subspace $W_i$, and is strictly positive definite:
\[ \frac13 \norm{M}^2 \leq H_\alpha(M) \leq 9 \norm{M}^2. \]
\end{lemma}
\begin{proof}
First decomposing $M$ into symmetric and antisymmetric components, we can write
\[ \Tr(M^2) = \Tr((M_s+M_a)M) = \Tr((M_s^\top - M_a^\top) M) = \ip{M_s}{M} - \ip{M_a}{M} = \norm{M_s}^2 - \norm{M_a}^2. \]
Decomposing the symmetric component further along invariant subspaces, we obtain
\[ Q_0(M) = \underbrace{\Tr(M)^2}_{3 \norm{M_1}^2} + \norm{M_a}^2 - \norm{M_s}^2 = \frac23 \Tr(M)^2 - \norm{M_2}^2 - \norm{M_3}^2 + \norm{M_4}^2. \]
Thus for $\alpha \in \R$ arbitrary, we have
\begin{align*}
    H_\alpha(M) & = \left( \frac{13}{6} + \frac23 \alpha \right) \Tr(M)^2 + \left( \frac{11}{8} - \alpha \right) \norm{M_2}^2 + \left( \frac{65}{24} - \alpha \right) \norm{M_3}^2 + \left(\alpha - \frac{17}{24} \right) \norm{M_a}^2.
\end{align*}
The optimal choice is $\alpha = \frac{25}{24}$ and yields
\begin{align*}
    H_\alpha(M) & = \frac{103}{36} \Tr(M)^2 + \frac{8}{24} \norm{M_2}^2 + \frac{40}{24} \norm{M_3}^2 + \frac{8}{24} \norm{M_a}^2
    \\ & = \frac{103}{12} \norm{M_1}^2 + \frac13 \norm{M_2}^2 + \frac53 \norm{M_3}^2 + \frac13 \norm{M_4}^2.
\end{align*}
Using the decomposition $\norm{M}^2 = \norm{M_1}^2 + \norm{M_2}^2 + \norm{M_3}^2 + \norm{M_4}^2$, we obtain
\[ \frac13 \norm{M}^2 \leq H_\alpha(M) \leq \frac{103}{12} \norm{M}^2 \leq 9 \norm{M}^2. \]
\end{proof}

We can thus conclude that the continuous Hessian satisfies $G(u)\geq0$, with equality if and only if $u$ is constant. In particular, $G$ is positive definite on the space of displacements with zero mean.

\begin{theorem}\label{thm2refined}
For all $u \in H^1(\T; \R^3)$,
\[ \frac13 \frac1{\abs\T} \int_\T \norm{\nabla u}^2 \,dx \leq G(u) \leq 9 \frac1{\abs\T} \int_\T \norm{\nabla u}^2 \,dx.\]
\end{theorem}
\begin{proof}
By Lemma~\ref{lem:q0nulllag}, for any choice of $\alpha$, $H_\alpha(M) = H(M) + \alpha Q_0(M)$ satisfies
\[ G(u) = \frac1{\abs\T} \int_\T H(\nabla u(x)) \,dx = \frac1{\abs\T} \int_\T H_\alpha(\nabla u(x)) \,dx. \]
Lemma~\ref{lem:halpha} then implies the desired bounds for the optimal $\alpha = \frac{25}{24}$.
\end{proof}

\subsection{Alternate explicit calculation of Hessian}\label{sect:altcalc}

Recall that by \eqref{eq:hesssym}, the symmetries of the quantization error force the Hessian of $F$ at $I$ to be given by
\[ \nabla^2 F(I)[N,N] = \frac{20}{9} F(I) \Tr(N)^2 + c_2 \norm{N_2}^2 + c_3 \norm{N_3}^2 + \frac53 F(I) \norm{N_4}^2. \]

Along the flow $M(t) = I + t N$, by \eqref{eq:gradf} we have
\[ \frac{d}{dt} F(M(t)) = \sum_{l \in N_0} U_l(M(t)) \cdot N l = \sum_{l \in N_0} \frac52 \lambda_l(M(t)) w(M(t)l) \cdot N l. \]
Differentiating again at $t = 0$, we obtain
\begin{align*} 
\nabla^2 F(I)[N,N] 
& = \sum_{l \in N_0} \frac52 \left. \frac{d}{dt} \lambda_l(M(t)) \right|_{t=0} w(l) \cdot N l + \sum_{l \in N_0} \frac52 \lambda_l(I) \left. \frac{d}{dt} w(M(t)l) \cdot N l \right|_{t=0}
\\ & = \sum_{l \in N_0} \frac{5}{\abs{l}^2} \ip{\nabla \lambda_l(I)}{N} l^\top N l + \sum_{l \in N_0} \frac{5}{\abs{l}^2} \lambda_l(I) (R_l N l) \cdot (N l)
\end{align*}
where we plug in $\dot w(x(t)) = \frac{2}{\abs{x(t)}^2} R_{x(t)} \dot x(t)$. 

\subsubsection{Symmetries of pyramid costs}

Recalling also that $G$ acts on pyramid costs as $\lambda_l(gMg^\top) = \lambda_{g^\top l}(M)$, the conjugate action of $G_l$ in particular leaves $\lambda_l$ invariant. Differentiating the resulting identity
\[ \lambda_l(I + t g N g^\top) = \lambda_l(I + t N), \quad g \in G_l \]
at $t = 0$, we obtain
\[ \ip{\nabla \lambda_l(I)}{g N g^\top} = \ip{\nabla \lambda_l(I)}{N}, \quad g \in G_l \]
i.e. $\nabla \lambda_l(I) \in \Fix(G_l)$. Lemma~\ref{lem:invlinforms} then implies that
\[ \nabla \lambda_l(I) \in \Span\{I, l l^\top\}. \]
We choose to decompose $\nabla \lambda_l(I)$ as
\[ \nabla \lambda_l(I) = \frac{a_l}{\abs{l}^2} l l^\top + b_l \left( I - \frac{3}{\abs{l}^2} l l^\top \right) \]
so that the second term is traceless hence orthogonal to the identity matrix. By the general action of $G$ on $\lambda_l$, we can also conclude that the coefficients $a_l$ and $b_l$ are constant along $G$-orbits.

The coefficient $a_l$ is also fixed by the dilation symmetry of $\lambda_l$: differentiating the identity
\[ \lambda_l(I + t I) = \lambda_l((1+t) I) = (1+t)^5 \lambda_l(I) \]
both ways at $t = 0$, we obtain
\[ a_l = \frac{a_l}{\abs{l}^2} \ip{ll^\top}{I} = \ip{\nabla \lambda_l(I)}{I} = \left. \frac{d}{dt} \lambda_l(I+tI) \right|_{t=0} = 5 \lambda_l(I). \]
Thus for general $N$,
\[ \ip{\nabla \lambda_l(I)}{N} = b_l \Tr(N) + (5 \lambda_l - 3 b_l) \frac{l^\top N l}{\abs{l}^2}. \]
Inserting this decomposition back in the formula obtained for the Hessian, 
\begin{align*} 
\nabla^2 F(I)[N,N] 
 & = \sum_{l \in N_0} \frac{5}{\abs{l}^2} \lambda_l(I) \left[ \frac{5}{\abs{l}^2} (l^\top N l)^2 + (R_l Nl) \cdot (N l) \right] 
 + \sum_{l \in N_0} \frac{5}{\abs{l}^2} b_l \left[ \Tr(N) l^\top N l - \frac{3}{\abs{l}^2} (l^\top N l)^2 \right]
 \\ & = \sum_{l \in N_0} \frac{5}{\abs{l}^2} \lambda_l(I) \abs{Nl}^2 + \sum_{l \in N_0} \frac{15}{\abs{l}^4} (\lambda_l(I) - b_l) (l^\top N l)^2 + \Tr(N) \sum_{l \in N_0} \frac{5}{\abs{l}^2} b_l l^\top N l,
\end{align*}
where we plug in
\begin{align*}
 (R_l Nl) \cdot (N l) & = \left( Nl - \frac{2}{\abs{l}^2} (l^\top Nl) l \right) \cdot N l = \abs{Nl}^2 - \frac{2}{\abs{l}^2} (l^\top N l)^2. 
\end{align*}
Noting the $G$-invariance of $\lambda_l(I)$ and $b_l$, we can split each sum along the two $G$-orbits and plug in the linear and quadratic forms treated in Section~\ref{sect:invforms}:
\begin{align*}
\nabla^2 F(I)[N,N] & = 5 \underbrace{(6 \lambda_{ax} + 8 \lambda_{diag})}_{F(I)} \cdot \frac13 \norm{N}^2 + 5 \Tr(N) (6 b_{ax} + 8 b_{diag}) \cdot \frac13 \Tr(N) 
\\ & + 15  (\lambda_{ax} - b_{ax}) \cdot 6 \left( \frac19 \Tr(N)^2 + \frac13 \norm{N_2}^2 \right) + 15 (\lambda_{diag} - b_{diag}) \cdot 8 \left( \frac19 \Tr(N)^2 + \frac29 \norm{N_3}^2 \right)
\\ & = \frac53 F(I) \norm{N}^2 + \frac53 F(I) \Tr(N)^2 + 30  (\lambda_{ax} - b_{ax}) \norm{N_2}^2 + \frac{80}{3} (\lambda_{diag} - b_{diag}) \norm{N_3}^2,
\end{align*}
where the $(6 b_{ax} + 8 b_{diag}) \Tr(N)^2$ term cancels out exactly. 

Finally plugging in $\norm{N}^2 = \frac13 \Tr(N)^2 + \norm{N_2}^2 + \norm{N_3}^2 + \norm{N_4}^2$, we obtain
\begin{equation}\label{eq:hessf2}
\nabla^2 F(I)[N,N] = \frac53 \cdot \frac43 F(I) \Tr(N)^2 + c_2 \norm{N_2}^2 + c_3 \norm{N_3}^2 + \frac53 F(I) \norm{N_4}^2, 
\end{equation}
exactly in the form necessitated by the symmetries of the quantization error, with 
\begin{equation}
c_2 = \frac53 F(I) + 30  (\lambda_{ax} - b_{ax}); \qquad c_3 = \frac53 F(I) + \frac{80}{3} (\lambda_{diag} - b_{diag}).
\end{equation}

\subsubsection{Calculation of the coefficients $b_l$}

In order to fix the values of $b_l$, we will rely on the additional property specific to the BCC lattice, on top of its $G$-symmetry, that all the interfaces $E_{ll'}$ are congruent, hence \eqref{eq:muijcdotlij} applies.

We now compute the coefficients $b_l$ by testing $\nabla \lambda_l(I)$ against the orthogonal projection
\[ N = I - \frac{1}{\abs{l}^2} l l^\top. \] 
Observe that $\ip{ll^\top}{N} = 0$ and 
\begin{align*} 
\ip{I - \frac{3}{\abs{l}^2} l l^\top}{N} 
& = \Tr\left( \Bigl(I - \frac{3}{\abs{l}^2} l l^\top \Bigr) \Bigl(I - \frac{1}{\abs{l}^2} l l^\top \Bigr) \right)
\\ & = \Tr\left( I - \frac{3}{\abs{l}^2} l l^\top - \frac{1}{\abs{l}^2} l l^\top + \frac{3}{\abs{l}^4} l l^\top l l^\top \right) 
\\ & = 3 - \frac{3}{\abs{l}^2} \abs{l}^2 - \frac{1}{\abs{l}^2} \abs{l}^2 + \frac{3}{\abs{l}^4} \abs{l}^4 = 2,
\end{align*}
so that
\[ \ip{\nabla \lambda_l(I)}{N} = a_l \cdot 0 + b_l \cdot 2 = 2 b_l. \]
We now evaluate the left-hand side by considering the flow $x_{l'}(t) = (I + tN) l'$. 
Along this flow, we have $\dot x_l = N l = 0$ by construction, hence also
\[ \dot w_l := \frac{d}{dt} w(x_l(t)) = \frac{2}{\abs{l}^2} R_l \dot x_l = 0, \]
and for $l' \sim l$,
\begin{align*} 
\dot w_{l'} := \frac{d}{dt} w(x_{l'}(t))  = \frac{2}{\abs{l'}^2} R_{l'} \dot x_{l'} & = R_{l'} N w(l') 
\\ & = R_{l'} \Bigl(w(l') - \frac{w(l') \cdot w(l)}{\abs{w(l)}^2} w(l) \Bigr)
\\ & = -w(l') - \frac{w(l') \cdot w(l)}{\abs{w(l)}^2} \Bigl(w(l) - \frac{2w(l') \cdot w(l)}{\abs{w(l')}^2} w(l') \Bigr)
\\ & = \left( 2 \frac{(w(l) \cdot w(l'))^2}{\abs{w(l)}^2 \abs{w(l')}^2} - 1 \right) w(l') -  \frac{w(l) \cdot w(l')}{\abs{w(l)}^2} w(l).
\end{align*}
Then applying \eqref{eq:dlambdai},
\begin{align*}
\ip{\nabla \lambda_l(I)}{N} & = \left. \frac{d}{dt} \lambda_l(I+tN) \right|_{t=0}
= \frac{5}{\abs{l}^2} \mu_l(I) \cdot N l + \sum_{l' \sim l} \xi_{ll'}(I) \cdot \frac{\dot w_l - \dot w_{l'}}{\abs{w(l) - w(l')}}
 = - \sum_{l' \sim l} \frac{\xi_{ll'}(I) \cdot \dot w_{l'}}{\abs{w(l) - w(l')}}
    \\ & = - \sum_{l' \sim l} \frac{1}{\abs{w(l) - w(l')}} \left[  \left( 2 \frac{(w(l) \cdot w(l'))^2}{\abs{w(l)}^2 \abs{w(l')}^2} - 1 \right) \xi_{ll'}(I) \cdot w(l') - \frac{w(l) \cdot w(l')}{\abs{w(l)}^2} \xi_{ll'}(I) \cdot w(l) \right]
    \\ & = \beta \sum_{l' \sim l} \frac{1}{\abs{ w(l) - w(l') } } \left[ 1 + \frac{w(l) \cdot w(l')}{\abs{w(l)}^2}  - 2 \frac{(w(l) \cdot w(l'))^2}{\abs{w(l)}^2 \abs{w(l')}^2} \right]
\end{align*}
where the last line follows from \eqref{eq:muijcdotlij}.

We have thus obtained
\[ b_l = \frac{\beta}{2} \sum_{l' \sim l} \frac{1}{\abs{ w(l) - w(l') } } \left( 1 + \frac{w(l) \cdot w(l')}{\abs{w(l)}^2}  - 2 \frac{(w(l) \cdot w(l'))^2}{\abs{w(l)}^2 \abs{w(l')}^2} \right). \]
Noting that the summand as well as the incidence relations are conserved by the action of $G$, it remains to compute this quantity once for $l \in N_{ax}$ and once for $l \in N_{diag}$. Indeed, the summand is conserved within directed edge types, of which there are three: $l$ axial and $l'$ diagonal, $l$ diagonal and $l'$ axial, $l$ and $l'$ both diagonal. 

\paragraph{Axial-diagonal edges:} Up to change of basis, take $l = (2,0,0)$ and $l' = (1,1,1)$. Then $w(l) = (1,0,0)$ and $w(l') = (\frac23, \frac23, \frac23)$, so that
\[ \abs{w(l)}^2 = 1, \quad \abs{w(l')}^2 = \frac43, \quad w(l) \cdot w(l') = \frac23, \quad \abs{w(l) - w(l')}^2 = 1 - 2 \cdot \frac23 + \frac43 = 1. \]
Therefore
\[ \frac{1}{\abs{w(l) - w(l')}} \left( 1 + \frac{w(l) \cdot w(l')}{\abs{w(l)}^2}  - 2 \frac{(w(l) \cdot w(l'))^2}{\abs{w(l)}^2 \abs{w(l')}^2} \right) = 1 + \frac23 - 2 \frac{\frac49}{\frac43} = 1 \]
and for the other direction,
\[ \frac{1}{\abs{w(l) - w(l')}} \left( 1 + \frac{w(l) \cdot w(l')}{\abs{w(l')}^2}  - 2 \frac{(w(l) \cdot w(l'))^2}{\abs{w(l)}^2 \abs{w(l')}^2} \right) = 1 + \frac{\frac23}{\frac43} - 2 \frac{\frac49}{\frac43} = 1 + \frac12 - \frac23 = \frac56. \]

\paragraph{Diagonal-diagonal edges:} Up to change of basis, take $l = (1,1,1)$ and $l' = (-1,1,1)$. Then $w(l) = (\frac23, \frac23, \frac23)$ and $w(l') = (-\frac23, \frac23, \frac23)$, so that
\[ \abs{w(l)}^2 = \abs{w(l')}^2 = \frac43, \quad w(l) \cdot w(l') = \frac49, \quad \abs{w(l) - w(l')}^2 = \frac43 - 2 \cdot \frac49 + \frac43 = \frac{16}{9}. \]
Therefore
\[ \frac{1}{\abs{w(l) - w(l')}} \left( 1 + \frac{w(l) \cdot w(l')}{\abs{w(l)}^2}  - 2 \frac{(w(l) \cdot w(l'))^2}{\abs{w(l)}^2 \abs{w(l')}^2} \right) 
= \frac34 \left( 1 + \frac13 - 2 \frac19 \right) = \frac34 \frac{10}{9} = \frac56. \]

\paragraph{Final values of $b_l$.} Each axial neighbor $l \in N_{ax}$ is adjacent to four diagonal neighbors, each contributing a factor of $1$, while each diagonal neighbor $l \in N_{diag}$ is adjacent to three axial and three diagonal neighbors, each contributing a factor of $\frac56$. We then conclude that
\[ b_{ax} = 2 \beta; \quad b_{diag} = \frac{5\beta}2. \]

\subsubsection{Conclusion for the Hessian}

To conclude the calculation of the Hessian, we plug in 
\[ F(I) = \frac{19}{8}, \quad \lambda_{ax} = \frac{13}{120}, \quad \lambda_{diag} = \frac{69}{320}, \quad b_{ax} = 2 \beta = \frac{7}{20}, \quad b_{diag} = \frac52 \beta = \frac{7}{16} \]
into \eqref{eq:hessf2}, obtaining that
\[ 
c_2 = \frac53 \cdot \frac{19}{8} + 30 \frac{13-42}{120} = - \frac{79}{24}; \quad c_3 = \frac53 \cdot \frac{19}{8} + \frac{80}{3} \frac{69-140}{320} = - \frac{47}{24}
\]
hence
\begin{equation}\label{eq:hessf3}
\nabla^2 F(I)[N,N] = \frac{95}{18} \Tr(N)^2 - \frac{79}{24} \norm{N_2}^2 - \frac{47}{24} \norm{N_3}^2 + \frac{95}{24} \norm{N_4}^2.
\end{equation}

We can now compare the Hessian with the formula \eqref{eq:hfinal} for the quadratic form $H$, verifying explicitly that the difference $D$ is null-Lagrangian, in fact a multiple of $Q_0$:
\begin{align*}
D(N) & = \nabla^2 F(I)[N,N] - H(N)
\\ & = \left( \frac{95}{18} - \frac{13}{6} \right) \Tr(N)^2 + \left( -\frac{79}{24} - \frac{11}{8} \right) \norm{N_2}^2 + \left( -\frac{47}{24} - \frac{65}{24} \right) \norm{N_3}^2 + \left( \frac{95}{24} + \frac{17}{24} \right) \norm{N_4}^2
\\ & = \frac{28}{9} \Tr(N)^2 - \frac{28}{6} \norm{N_2}^2 - \frac{28}{6} \norm{N_3}^2 + \frac{28}{6} \norm{N_4}^2
\\ & = \frac{14}{3} \left( \frac23 \Tr(N)^2 - \norm{N_2}^2 - \norm{N_3}^2 + \norm{N_4}^2 \right) = \frac{14}{3} Q_0(N).
\end{align*}
We also confirm that the normalized cost $\hat F$ is locally minimized at identity: by \eqref{eq:normhesssym}, the Hessian of $\hat F$ is given by
\begin{align*}
 \nabla^2 \hat F(I)[N, N] 
 & = \bigl(c_2 + \frac53 F(I)\bigr) \norm{N_2}^2 + \bigl(c_3 + \frac53 F(I)\bigr) \norm{N_3}^2
 \\ & =  \left( -\frac{79}{24} + \frac{95}{24} \right) \norm{N_2}^2 + \left( -\frac{47}{24} + \frac{95}{24} \right) \norm{N_3}^2
 \\ & = \frac23 \norm{N_2}^2 + 2 \norm{N_3}^2.
 \end{align*}

\section{Asymptotics of the continuous gradient flow}\label{sect:thm3}
We lastly consider the  $L^2$-gradient flow associated to the continuous functional
\[
u \mapsto \cF(\id + u)=\frac1{\abs\T} \int_\T F(I + \nabla u(x))\,dx
\]
over functions $u \in C^1(\T; \R^3)$. 

In order to apply Poincar\'e--Wirtinger estimates and eliminate the translation-invariance inherent to the quantization problem, we impose the further restriction that the function $u$ has zero mean:
\[ \int_\T u(x) \,dx = 0. \]
We denote the spaces of zero-mean $L^2$ functions by $L^2_{\per,0}(\T; \R^3)$, and likewise for the Sobolev spaces $H^k_{\per,0}$ and $W^{\sigma,p}_{\per,0}$, where in each case we work with the normalized Lebesgue measure $\abs{\T}^{-1} \,dx$. On $H^1_{\per,0}$ we use the norm $\norm{u}_{H^1_{\per,0}} := \norm{\nabla u}_{L^2}$.

We denote by $C_P$ the optimal constant in the Poincar\'e--Wirtinger inequality
\[ \frac1{\abs\T} \int_\T \norm{\nabla u(x)}^2 \,dx \geq C_P \frac1{\abs\T} \int_\T \abs{u(x)}^2 \,dx. \]

By \eqref{eq:ffuncgrad}, the formal $L^2$ gradient of $\cF$ at $X = \id + u$ is
\[ \nabla_{L^2} \cF(X) = -\div_x \nabla F(\nabla X). \]
Therefore, at least formally, the $L^2$ gradient flow induced by the functional $\cF$ is given by
\begin{equation}\label{eq:gflow}
\partial_t u_t = \div_x \nabla F(I + \nabla_x u_t).
\end{equation}
Note that the right-hand side, being a divergence, preserves the zero-mean condition on $u_t$.

We will show that \eqref{eq:gflow} is indeed the gradient flow of $\cF$ and is strongly parabolic on a small $W^{\sigma,p}$-neighborhood of the identity.

The precise statement we will prove is the following:
\begin{theorem}\label{thm3}
Let $3 < p < \infty$ and $1+3/p < \sigma < 2$. 

There exists $\delta_0 \ll 1$ such that, for all $u_0 \in W^{\sigma,p}_{\per,0}(\T; \R^3)$ satisfying $\norm{u_0}_{W^{\sigma,p}} \leq \delta_0$, there is a unique weak solution in $L^\infty_t (H^1_{\per,0})_x$ to the gradient flow \eqref{eq:gflow}, defined for all $t > 0$ and satisfying
\[
    \norm{u_t}_{L^2(\T)} \leq e^{-C_P \lambda t} \delta_0, \qquad \norm{\nabla u_t}_{L^\infty(\T)} \leq C \phi(\delta_0) \ll 1
\]
for some $\lambda, C > 0$, $\phi \colon [0,\infty) \to [0,\infty)$ continuous and increasing with $\phi(0) = 0$.

Moreover, the solution is globally bounded in $C^{1,\alpha}$ for some $\alpha \in (0,1)$, and is classical for $t > 0$.
\end{theorem}

Note that for any $u_0 \in W^{\sigma,p}_{\per,0}(\T; \R^3)$ with $\norm{u_0}_{W^{\sigma,p}} \leq \delta_0$, $X_0 = \id + u_0$ is also a $C^1$-diffeomorphism of $\T$. 

To prove this theorem, we construct a matrix function $\tilde G$ which is globally $(\Lambda,\lambda)$-convex, such that whenever $\norm{\nabla u}_\infty$ is sufficiently small, $\div_x \nabla \tilde G(\nabla u) = \div_x \nabla F(I+\nabla u)$ and the induced functional
\[ \cG(u) := \frac1{\abs\T} \int_\T \tilde G(\nabla u(x)) \,dx \]
coincides exactly with $\cF(\id+u)$. We then deduce the existence, uniqueness and contractivity of weak solutions to the induced gradient flow
\[ \partial_t u_t = \div_x \nabla \tilde G(\nabla_x u_t). \]
We then apply an $A$-caloric regularity argument in order to conclude that the solution remains bounded in $C^{1,\alpha}$ and small in $W^{1,\infty}$, so that the right-hand side coincides with that of \eqref{eq:gflow} at all times, and the solution is indeed a classical solution of \eqref{eq:gflow}.

\subsection{Uniform ellipticity on a neighborhood of identity}

In order to translate the results of Section~\ref{sect:hesscalc}, 
which hold a priori only on the unperturbed lattice $X = \id$, to a local optimality statement on a neighborhood, we can then modify the matrix functional $F(\nabla X)$ by constant null-Lagrangian terms. We define specifically
\begin{equation}\label{eq:f0def} F_0(I+M) := F(I+M) - \underbrace{\ip{\nabla F(I)}{M}}_{\frac53 F(I) \Tr(M)} - \frac12 Q(M) \end{equation}
with the null-Lagrangian quadratic form
\[ Q(M) = \nabla^2 F(I)[M,M] - H_\alpha(M) = D(M) - \frac{25}{24} Q_0(M) = \left( \frac{14}{3} - \frac{25}{24} \right) \bigl(\Tr(M)^2 - \Tr(M^2)\bigr), \]
so that 
\begin{align*}
\nabla F_0(I) & = \nabla F(I) - \nabla F(I) = 0; \\
\nabla^2 F_0(I)[M,M] & = \nabla^2 F(I)[M,M] - Q(M) = H_\alpha(M). 
\end{align*}
Therefore, $\frac13 \norm{\cdot}^2 \leq \nabla^2 F_0(I) \leq 9 \norm{\cdot}^2$, and $F_0$ expands around the identity as
\[ F_0(I+\varepsilon M) = F(I) + \frac12 \varepsilon^2 H_\alpha(M) + O(\varepsilon^3). \]
That is, taking $\rho_0 > 0$ sufficiently small, $\nabla^2 F_0(M_0)$ is also positive definite for all $\norm{M_0-I} < \rho_0$.

Verifying that
\begin{align*}
\int_\T F(I + \nabla u(x)) \,dx - \int_\T F_0(I + \nabla u(x)) \,dx
= \frac53 F(I) \int_\T  \underbrace{\Tr(\nabla u(x))}_{\div u} \,dx + \frac12 \int_\T Q(\nabla u(x)) \,dx = 0
\end{align*}
where the first term vanishes because $\T$ is a compact manifold without boundary (or simply because all the directional derivatives $(l \cdot \nabla) u$ integrate to zero by the periodicity of $u$), we can conclude

\begin{theorem}\label{thm:locellip}
For $\rho_0 > 0$ sufficiently small, there exist $\Lambda \geq 9 > \frac13 \geq \lambda > 0$ such that on the set of functions $u \in C^1(\T; \R^3)$ with $\norm{\nabla u}_\infty \leq \rho_0$, the functional 
\[ u \mapsto \cF(\id+u) = \frac{1}{\abs{\T}} \int_\T F(I + \nabla u(x)) \,dx = \frac{1}{\abs{\T}} \int_\T F_0(I + \nabla u(x)) \,dx  \]
is $(\Lambda,\lambda)$-convex with respect to the homogeneous Sobolev norm. In particular, among all such functions with zero mean, the functional is strictly minimized at $u = 0$.
\end{theorem}

In particular, since the two functionals are identical, replacing $F$ with $F_0$ leaves the gradient flow unchanged. Explicitly, we have
\[ \nabla F_0(I+\nabla u) = \nabla F(I+\nabla u) - \frac12 \nabla Q(\nabla u) - \nabla F(I),  \]
where the second term, being a multiple of $\frac12 \nabla Q_0(\nabla u) = \Tr(\nabla u) I - \nabla u^\top$, is divergence-free by the proof of Lemma~\ref{lem:q0nulllag}. The constant last term is also divergence-free. Hence from now on, we express the gradient flow as
\begin{equation}\label{eq:gflowf0} \partial_t u_t = \div_x \nabla F_0(I+\nabla_x u_t) \end{equation}
due to the more favorable ellipticity properties of $F_0$.

\subsection{Globally elliptic extension}

Having constructed a matrix function $F_0$, which induces the same integral functional and has a strict local minimum at $M = I$ with positive definite Hessian, we can take constants $\rho, \lambda, \Lambda > 0$ such that $F_0$ is uniformly elliptic around the identity:
\[ \norm{M-I} < \rho \quad \implies \quad \lambda \norm{N}^2 \leq D^2 F_0(M)[N,N] \leq \Lambda \norm{N}^2 \quad \forall N \in \R^{3\times3}. \]
For the existence and regularity of solutions to the gradient flow, we will need such an ellipticity condition to hold globally for all $M$:

\begin{lemma}
There exists a smooth convex function $\tilde G \colon \R^{3\times3} \to \R$ such that, for some $\rho' > 0$, $\Lambda' > 9 > \frac13 > \lambda' > 0$,
\begin{itemize}
\item $\tilde G(N) \equiv F_0(I+N)$ whenever $\norm{N} < \rho'$;
\item $\displaystyle \lambda' \norm{\cdot}^2 \leq \nabla^2 \tilde G(N) \leq \Lambda' \norm{\cdot}^2$ for all $N \in \R^{3\times3}$.
\end{itemize}
\end{lemma}

\begin{proof}
We cite the extension theorem \cite[Thm.~1.3]{AzagraMudarra} which states that, given two functions $f \colon E \to \R$ and $G \colon E \to H$ defined on a subset $E \subset H$ of a Hilbert space $H$, there exists a $C^{1,1}$ convex function $\tilde f \colon H \to \R$ such that $\tilde f|_E = f$ and $\nabla \tilde f|_E = G$ iff $(f,G)$ satisfies the following for some $M > 0$:
\[ f(y) - f(x) - \ip{G(x)}{y-x} \geq \frac{1}{2M} \norm{G(y)-G(x)}^2 \quad \forall x, y \in E. \]
In that case, the same condition is satisfied by $(\tilde f,\nabla \tilde f)$ for all $x, y \in H$, and moreover, $\Lip(\nabla \tilde f) \leq M$. 

In our case, we take the convex subset $E = B_\rho(I) \subset \R^{3\times3}$, $f = F_0 - \frac\lambda4 \norm{\cdot}^2$, $G = \nabla f$. Then $f$ is uniformly elliptic on $E$ with 
\[ \frac\lambda2 \norm{\cdot}^2 \leq D^2 f(M) \leq (\Lambda - \frac\lambda2) \norm{\cdot}^2 \quad \forall M \in E, \]
thus $G = \nabla f$ is $L$-Lipschitz for $L = \Lambda - \frac\lambda2$. These two statements imply the desired condition:
\[ f(N) - f(M) - \ip{G(M)}{N-M} \geq \frac12 \frac\lambda2 \norm{N-M}^2 \geq \frac{\lambda}{4 L^2} \norm{G(N) - G(M)}^2 \quad \forall M,N \in E.  \]
Therefore, there exists $\tilde f \in C^{1,1}(\R^{3\times3})$ which is convex, satisfies $(\tilde f, \nabla \tilde f)|_E = (f, G)$, and $\Lip(\nabla \tilde f) \leq \frac{2L^2}{\lambda}$. 

Hence with the choice 
\[ \lambda' = \frac\lambda2, \quad \Lambda' = \frac{2L^2}{\lambda} + \frac\lambda2 = \frac\lambda2 \left[ \left(\frac{2\Lambda}{\lambda}-1\right)^2 + 1 \right] \geq \Lambda, \] 
the function $\tilde F = \tilde f + \frac{\lambda'}2 \norm{\cdot}^2$ extends $F_0$, is globally $C^{1,1}$ and $\lambda'$-convex with $\Lip(\nabla \tilde F) \leq \Lambda'$. That is, the ellipticity bounds
\[ \lambda' \norm{\cdot}^2 \leq D^2 \tilde F(M) \leq \Lambda' \norm{\cdot}^2 \]
hold for almost every $M \in \R^{3\times3}$.

Finally, to render $\tilde F$ globally smooth, we take a smaller neighborhood $E' = B_{\rho'}(I) \subset E$, a fixed bump function $\varphi$ which is $0$ outside $E$ and identically $1$ on $E'$, and a nonnegative approximation of identity $\vartheta_\varepsilon$ for $\varepsilon > 0$ sufficiently small. Then 
\[ \tilde F_\varepsilon := (1-\varphi) (\vartheta_\varepsilon \ast \tilde F) + \varphi \tilde F \] 
is everywhere smooth, and for small $\varepsilon$ also uniformly elliptic, only with slightly worse constants. 
On the region $\bar E \setminus E'$ where both terms play a role, we have
\begin{align*} 
D^2 \tilde F_\varepsilon & = (1-\varphi) (\vartheta_\varepsilon \ast D^2 \tilde F) + \varphi D^2 \tilde F
\\ & + \nabla \varphi \otimes (\nabla \tilde F - \vartheta_\varepsilon \ast \nabla \tilde F) +  (\nabla \tilde F - \vartheta_\varepsilon \ast \nabla \tilde F) \otimes \nabla \varphi + (\tilde F - \vartheta_\varepsilon \ast \tilde F) D^2 \varphi.
\end{align*}
The first line is uniformly sandwiched between $\lambda' \norm{\cdot}^2$ and $\Lambda' \norm{\cdot}^2$, since both $D^2 \tilde F$ and thus its convolution with $\vartheta_\varepsilon$ satisfy the ellipticity bounds uniformly on $\R^{3\times3}$. In particular, $D^2 \tilde F_\varepsilon$ admits the same bounds inside $E'$ and outside $\bar E$.

The second line is instead uniformly controlled by $\norm{\nabla \varphi}_{C^1} \norm{\tilde F - {\vartheta_\varepsilon \ast \tilde F}}_{C^1} \norm{\cdot}^2$, which tends to zero as $\varepsilon \to 0$ since $\tilde F \in C^{1,1}$. With a bound of $\delta \norm{\cdot}^2$ for some $\delta < \lambda'$, we obtain that
\[ (\lambda' - \delta) \norm{\cdot}^2 \leq D^2 \tilde F_\varepsilon \leq ( \Lambda' + \delta) \norm{\cdot}^2. \]
Lastly, setting $\tilde G(N) := \tilde F_\varepsilon(I+N)$ yields the desired extension.
\end{proof}

\subsection{Existence of convexified flow}

Consider now the $L^2$ gradient flow associated to the functional
\[ \cG(u) := \frac{1}{\abs{\T}} \int_\T \tilde G(\nabla u) \,dx. \]
In this section, we will use the global ellipticity of $\tilde G$ to obtain weak solutions to the abstract gradient flow of $\cG$:
\[ \partial_t u_t \in - \partial_{L^2} \cG(u_t), \]
where $\partial_{L^2} \cG(u)$ is the subdifferential of $\cG$. 

We first remark that, along the same lines as \cite[Ex.~11.6]{Ambrosio:2024aa}, the subdifferential of $\cG$ is either empty or a singleton containing only the functional gradient $\nabla_{L^2} \cG(u) = - \div_x \nabla \tilde G(\nabla u)$, if the latter exists in $L^2$ in a weak sense:
\[ \partial_{L^2} \cG(u) \neq \emptyset \quad \iff \quad \partial_{L^2} \cG(u) = \{\xi\} \subset L^2, \quad \int_\T \ip{\nabla \tilde G(\nabla_x u)}{\nabla \varphi} = \int_\T \xi \cdot \varphi \quad \forall \varphi \in C^\infty. \]

\begin{lemma}
On $L^2_{\per,0}(\T; \R^3)$, the functional
\[ \cG(u) := \left\{ \begin{matrix} \frac{1}{\abs{\T}} \int_\T \tilde G(\nabla u) \,dx, & u \in H^1_{\per,0} \\ \infty, & u \not\in H^1_{\per,0} \end{matrix} \right. \]
is $C_P \lambda'$-convex and weakly lower semicontinuous. 
\end{lemma} 
\begin{proof}
It is straightforward to observe that $\cG$ is $\lambda'$-convex and weakly lsc in $H^1_{\per,0}$ (cf.~\cite[Sect.~8.2]{Evans:2010aa}, Thm.~1). 

We then apply the Poincar\'e--Wirtinger inequality to obtain $C_P \lambda'$-convexity in $L^2$:
\begin{align*} 
& \qquad \cG((1-t)u+tv) - (1-t) \cG(u) - t \cG(v)
\\ & = \frac{1}{\abs{\T}} \int_\T \left[ \tilde G((1-t) \nabla u+t \nabla v) - (1-t) \tilde G( \nabla u) - t \tilde G( \nabla v) \right] \,dx
\\ & \leq \frac{1}{\abs{\T}} \int_\T -\frac{\lambda'}2 t(1-t) \norm{\nabla u - \nabla v}^2 \,dx 
\\ & = - \frac{\lambda'}2 t (1-t) \norm{\nabla u - \nabla v}^2_{L^2} \leq  - \frac{C_P \lambda'}2 t (1-t) \norm{u -v}^2_{L^2}.
\end{align*}
For weak lower semicontinuity in $L^2$, take a weakly convergent sequence $u_n \rightharpoonup u$ and set $L = \liminf_{n\to\infty} \cG(u_n)$, such that wlog each $\cG(u_n) < \infty$, i.e.~$u_n \in H^1_{\per,0}$, and $L = \lim_{n\to\infty} \cG(u_n) < \infty$. 
By the $\lambda'$-convexity of $\tilde G$ and the boundedness of $\cG(u_n)$, observe that $(u_n)_n$ is uniformly bounded in $H^1_{\per,0}$:
\[ C \geq\cG(u_n) = \frac{1}{\abs{\T}} \int_\T \tilde G(\nabla u_n ) \,dx \geq  \frac{1}{\abs{\T}} \int_\T \left[ F(I) + \frac{\lambda'}{2} \norm{\nabla u_n}^2 \right] \,dx = \cG(0) + \frac{\lambda'}{2} \norm{u_n}_{H^1_{\per,0}}^2. \]
Therefore by the Banach-Alaoglu theorem, along a subsequence $u_{n_k} \rightharpoonup \bar u$ in $H^1_{\per,0}$. The uniqueness of the weak $L^2_{\per,0}$ limit implies that $\bar u = u \in H^1_{\per,0}$, and the $H^1_{\per,0}$-weak lower semicontinuity of $\cG$ then implies $\cG(u) \leq \liminf_{k\to\infty} \cG(u_{n_k}) = L$.
\end{proof}

Therefore by the Brezis-Komura theorem \cite[Thm.~11.7]{Ambrosio:2024aa}, there exists a unique $C_P\lambda'$-contractive semigroup on $L^2_{\per,0}$ generating the gradient flow
\[ \partial_t u_t = - \nabla_{L^2} \cG(u_t) = \div_x \nabla \tilde G(\nabla_x u_t), \]
strongly in time and weakly in space: $(u_t)_t \in AC([0,\infty); L^2_{\per,0})$ and for a.e.~$t \in (0,\infty)$, $\dot u_t = -\nabla_{L^2} \cG(u_t) \in L^2_{\per,0}$. In other words,
\[ \frac{d}{dt} \frac{1}{\abs{\T}} \int_\T \ip{u_t}{\phi} \,dx = \frac{1}{\abs{\T}} \int_\T \ip{\dot u_t}{\phi} \,dx = - \frac{1}{\abs{\T}} \int_\T \ip{\nabla \tilde G(\nabla_x u_t)}{\nabla_x \phi} \,dx \quad \forall \phi \in H^1_{\per,0}(\T; \R^3). \]

\subsection{A priori estimates}

Let $\{u_t\}_{t\geq0} \subset L^2_{\per,0}$ be a weak solution of the gradient flow 
\[ \dot u_t = \div_x \nabla \tilde G(\nabla_x u_t) = -\nabla_{L^2} \cG(u_t). \]
Then
\[ \frac{d}{dt} \cG(u_t) = \ip{\nabla_{L^2} \cG(u_t)}{\dot u_t} = - \norm{\dot u_t}_{L^2}^2 \leq 0  \quad \text{for a.e. } t \in (0,\infty)\]
implies that the energy $\cG(u_t)$ is nonincreasing wrt $t$. 

Since $\tilde G$ is $\lambda'$-convex and minimized at $0$, we obtain
\[ \cG(u_0) \geq \cG(u_t) = \frac1{\abs\T} \int_{\T} \tilde G(\nabla u_t) \,dx \geq \frac1{\abs\T} \int_{\T} \left(F(I) + \frac{\lambda'}2 \norm{\nabla u_t}^2 \right) \,dx = \cG(0) + \frac{\lambda'}2 \norm{u_t}_{H^1_{\per,0}}^2, \]
which shows that $(u_t)_t$ is uniformly bounded in $H^1_{\per,0}$.

Let now $(u_t)_t$, $(v_t)_t$ be two solutions of the gradient flow. Testing against $u_t - v_t \in H^1_{\per,0}$ gives
\begin{align*}
    \frac{d}{dt} \frac12 \norm{u_t-v_t}_{L^2}^2 & = \ip{\dot u_t - \dot v_t}{u_t - v_t}_{L^2}
    \\ & = - \frac1{\abs\T} \int_{\T} \ip{\nabla \tilde G(\nabla u_t) - \nabla \tilde G(\nabla v_t)}{\nabla u_t - \nabla v_t} \,dx
    \\ & \leq - \frac1{\abs\T} \int_{\T} \lambda' \norm{\nabla u_t - \nabla v_t}^2 \,dx
    = - \lambda' \norm{u_t - v_t}_{H^1_{\per,0}}^2 
    \\ & \leq - C_P \lambda' \norm{u_t - v_t}_{L^2}^2,
\end{align*}
where the first inequality follows from the $\lambda'$-convexity of $\tilde G$, and the second from the Poincar\'e--Wirtinger inequality. Applying Gr\"onwall then yields
\[ \norm{u_t-v_t}_{L^2}^2 \leq e^{- 2C_P\lambda' t} \norm{u_0 - v_0}_{L^2}^2 \quad \implies \quad \norm{u_t - v_t}_{L^2} \leq e^{-C_P\lambda' t} \norm{u_0 - v_0}_{L^2}. \]
This verifies the $C_P\lambda'$-contractivity of the gradient flow, from which the uniqueness of solutions follows.

Taking the equilibrium solution $v_t = 0$, we obtain the following $L^2$ stability statement:
\[ \norm{u_t}_{L^2} \leq e^{-C_P\lambda' t} \norm{u_0}_{L^2}. \]
Our goal is now to upgrade this $L^2$ stability into $C^{1,\alpha}$ stability.

\subsection{Regularity and trapping}

In order to obtain a gradient flow for the original functional $\cF$, it remains to show that the gradient flow of $\cG$ remains $C^{1,\alpha}$-close to the identity as long as the initial state is sufficiently close. To this end, we follow the argument in \cite{CGI_2D} closely, writing out each step in detail. We first show that the solution is $C^{1,\alpha}$-close to the identity up to a time threshold, and then obtain a uniform $C^{1,\beta}$ bound for some $0<\beta<\alpha$, while keeping $\norm{\nabla u_t}_{L^\infty}$ small for all times.

\subsubsection{Short-time regularity}

The first step in our argument is to show that, if $u_0$ is sufficiently close to zero with respect to an appropriate Sobolev norm, the solution remains $C^{1,\alpha}$-close to the identity up to some threshold $t_0 > 0$. For this, we make use of the Theorem of Amann \cite[p.~192]{Amann}. While this theorem is formulated for Dirichlet solutions on bounded domains, it follows from abstract semigroup arguments, which also apply to the case of periodic zero-mean solutions; see also \cite[Rem.~10.7]{Amann}. 

To verify the assumptions of Theorem \cite[p.~192]{Amann}, we first write out the gradient flow PDE in non-divergence form:
\[ \partial_t u_t + \mathscr{A}(\nabla_x u_t) u_t = 0, \quad \mathscr{A}(\nabla_x u) u = -\div_x \nabla \tilde G(\nabla_x u) = - \sum_{j,k} a_{jk}(\nabla_x u) \partial_j \partial_k u,  \]
where $a_{jk} \in C^2(\R^{3\times3}; \R^{3\times3})$. Indeed, we have
\[ - (\mathscr{A}(\nabla_x u) u)_i = \sum_j \partial_j (\partial_{M_{ij}} \tilde G(\nabla_x u)) = \sum_{j,k,l} (\partial_{M_{kl}} \partial_{M_{ij}} \tilde G)(\nabla_x u) \partial_j \underbrace{(\nabla_x u)_{kl}}_{\partial_l u_k} \overset{!}{=} \sum_{j,k,l} a_{jl}^{ik}(\nabla_x u) \partial_j \partial_l u_k, \]
so that
\[ a_{jl}^{ik}(\nabla_x u) = (\partial_{M_{kl}} \partial_{M_{ij}} \tilde G)(\nabla_x u). \]
We then verify the required ellipticity assumption
\[ \Re \Bigl( \sum_{i,j,k,l} a_{jl}^{ik} \xi^j \xi^l \alpha_i \bar \alpha_k  \Bigr) > 0 \quad \text{for all} \quad \xi \in \R^3 \setminus \{0\},\, \alpha \in \mathbb{C}^3 \setminus \{0\}. \]
First observe that since $a_{jl}^{ik} = a_{lj}^{ki} \in \R$ by construction, we can restrict to $\alpha \in \R^3 \setminus \{0\}$ by interchanging $i \leftrightarrow k$ and $j \leftrightarrow l$. 
Then considering the one-rank matrices $A_{ij} = \alpha_i \xi^j$, we have
\[ \sum_{i,j,k,l} a_{jl}^{ik} \xi^j \xi^l \alpha_i \alpha_k = \sum_{i,j,k,l} (\partial_{M_{kl}} \partial_{M_{ij}} \tilde G) A_{ij} A_{kl} \geq \lambda' \norm{A}^2 = \lambda' \abs{\xi}^2 \abs{\alpha}^2 > 0. \]

We can then apply Theorem \cite[p.~192]{Amann}. Given $3 < p < \infty$ and $1+3/p < \sigma < 2$, for any $u_0 \in W^{\sigma,p}_{\per,0}$, there exists a unique maximal classical solution $u \colon [0,T(u_0)) \to W^{\sigma,p}_{\per,0}(\T)$ to the PDE
\[ \partial_t u_t = \div_x \nabla \tilde G(\nabla_x u_t) \]
with initial state $u_0 \in W^{\sigma,p}_{\per,0}(\T)$. Moreover, the domain 
\[ \mathscr{D} = \{(t, u_0) \mid 0 \leq t < T(u_0), \, u_0 \in W^{\sigma,p}_{\per,0}\} \]
is open in $[0,\infty) \times W^{\sigma,p}_{\per,0}$, and $\Phi(t, u_0) = u_t$ is continuous as a map $\mathscr{D} \to W^{\sigma,p}_{\per,0}$.

In particular, since $\Phi(0,0) = 0$, for any $\varepsilon_0 > 0$ there exist $t_0 = t_0(\varepsilon_0),\ \delta_0 = \delta_0(\varepsilon_0) > 0$ such that
\[ 0 \leq t \leq t_0, \quad \norm{u_0}_{W^{\sigma,p}_{\per,0}} \leq \delta_0 \quad \implies \quad \norm{u_t}_{W^{\sigma,p}_{\per,0}} < \varepsilon_0. \]
Then by Sobolev embedding (noting $0<3/p<1$ and $\sigma > 1+3/p$), we have
\[ \norm{u_t}_{C^{1,\alpha}} < C \varepsilon_0 \quad \text{for all} \quad t \leq t_0 \]
for some constant $C$, where $0<\alpha<\sigma-1-3/p$. We have then obtained

\begin{lemma}\label{lem:shortreg}
Let $3 < p < \infty$ and $1+3/p < \sigma < 2$. For any $\varepsilon_0 > 0$, there exist $t_0, \ \delta_0 > 0$ such that
\[ \norm{u_0}_{W^{\sigma,p}_{\per,0}} \leq \delta_0 \quad \implies \quad \max_{0\leq t \leq t_0} \norm{u_t}_{C^{1,\alpha}} \leq \varepsilon_0. \]
\end{lemma}

\subsubsection{Global regularity}

In order to extend this short-time regularity result to all times, we now apply an $A$-caloric regularity argument following \cite{DuzaarMingione}. For a more general treatment of regularity for nonlinear
parabolic systems, we refer to \cite{DuzaarMingioneSteffen}.

Fix $\varepsilon_0>0$, and let $t_0,\delta_{\rm loc}>0$ be given by Lemma~\ref{lem:shortreg}. Assume $\norm{u_0}_{W^{\sigma,p}_{\per,0}}\leq\delta_0\leq\delta_{\rm loc}$. We will decrease $\delta_0$ as needed below, keeping $\varepsilon_0$ and $t_0$ fixed. Fix also $0<\beta<\alpha$.

\paragraph{Step 1: Parabolic averages.} On $\R_+ \times \R^3$, we define the parabolic metric
\[ d(z_1, z_2) := \max\{\abs{t_1-t_2}^{1/2}, \abs{x_1-x_2}\}; \quad z_i = (t_i, x_i) \in \R_+ \times \R^3. \]
This is bi-Lipschitz equivalent to the metric defined in \cite[(1.2)]{DuzaarMingione} (up to a factor of $\sqrt2$), but will be more suitable for our purposes. We use the backward parabolic cylinders
\[ Q_\delta(\bar z) = \{ (t,x) \mid \bar t-\delta^2 < t < \bar t,\ \abs{x-\bar x}<\delta \} \]
for $\bar z=(\bar t,\bar x)\in\R_+\times\T$ and $\delta>0$ smaller than the injectivity radius of $\T$. Their volume is $\abs{Q_\delta(\bar z)}=\delta^2\cdot\frac{4\pi}{3}\delta^3=\frac{4\pi}{3}\delta^5$.

Pick $\bar z=(\bar t,\bar x)$ with $\bar t\geq t_0/2$ and $0<\delta<\min\{1,\sqrt{t_0}/4,\operatorname{inj}(\T)/4\}$, so that the cylinders remain in positive time. We will use the resulting estimates for $t\geq t_0$. We define the parabolic averages 
\[ \Psi_2(\bar z, \delta, l) := \frac{1}{\abs{Q_\delta(\bar z)}} \int_{Q_\delta(\bar z)} \abs{\frac{u_t(x) - l(x)}{\delta}}^2 \,dt\,dx
\]
for an affine map $l(x)=Ax+b$ independent of time, as in \cite[Lem.~6.1]{DuzaarMingione}. We first estimate $\Psi_2(\bar z,\delta,0)$.

Recalling that the spatial $L^2$ norm is normalized and $\abs\T=4$, we enlarge the spatial integration domain to $\T$ and apply the $C_P\lambda'$-contractivity of the gradient flow to obtain
\[
\Psi_2(\bar z,\delta,0)\leq \frac{3}{\pi\delta^7}\int_{\bar t-\delta^2}^{\bar t}e^{-2C_P\lambda't}\norm{u_0}_{L^2_{\per,0}}^2\,dt \leq \frac{3}{\pi\delta^5}e^{-2C_P\lambda'(\bar t-\delta^2)}\norm{u_0}_{L^2_{\per,0}}^2.
\]
Since $3/\pi<1$, we can take $C_\delta:=\delta^{-5}$. The bound $\norm{u_0}_{L^2_{\per,0}}\leq\delta_0$ then gives $\Psi_2(\bar z,\delta,0)\leq C_\delta\delta_0^2$ uniformly in $\bar z$.

\paragraph{Step 2: Best affine approximations.}
Given $\bar z$, $\delta$ as above, there is a unique affine map $l_{\bar z,\delta}$ that minimizes $l \mapsto \Psi_2(\bar z, \delta, l)$. Following \cite[Sect.~3]{DuzaarMingione}, this minimizer is given by
\begin{align*} 
l_{\bar z,\delta}(x) & = m_{\bar z,\delta} + A_{\bar z,\delta}(x-\bar x); \\
\quad m_{\bar z,\delta} = \frac{1}{\abs{Q_\delta(\bar z)}} \int_{Q_\delta(\bar z)} u_t(x) \,dt \,dx; \quad &
A_{\bar z,\delta} = \frac{5}{\delta^2} \frac{1}{\abs{Q_\delta(\bar z)}} \int_{Q_\delta(\bar z)} u_t(x) (x - \bar x)^\top \,dt \,dx .
\end{align*}

Evidently for $u = 0$, the best affine approximation is $0$ itself. For $u$ close to $0$, we can bound
\begin{align*}
    \abs{m_{\bar z,\delta}}^2 & = \abs{\frac{1}{\abs{Q_\delta(\bar z)}} \int_{Q_\delta(\bar z)} u_t(x) \,dt \,dx}^2
    \leq \frac{1}{\abs{Q_\delta(\bar z)}} \int_{Q_\delta(\bar z)} \abs{u_t(x)}^2 \,dt \,dx  
    =\delta^2\Psi_2(\bar z,\delta,0); \\
    \norm{A_{\bar z,\delta}}^2 & = \frac{25}{\delta^4} \norm{\frac{1}{\abs{Q_\delta(\bar z)}} \int_{Q_\delta(\bar z)} u_t(x) (x-\bar x)^\top \,dt \,dx}^2
    \leq \frac{25}{\delta^4}  \frac{1}{\abs{Q_\delta(\bar z)}} \int_{Q_\delta(\bar z)} \norm{u_t(x) (x-\bar x)^\top}^2 \,dt \,dx
    \\ & = \frac{25}{\delta^4}  \frac{1}{\abs{Q_\delta(\bar z)}} \int_{Q_\delta(\bar z)} \abs{u_t(x)}^2 \underbrace{\abs{x-\bar x}^2}_{\leq \delta^2} \,dt \,dx
    \leq25\Psi_2(\bar z,\delta,0).
\end{align*}

\paragraph{Step 3: Campanato estimates $\implies$ global $C^{1,\beta}$ bound.}
We now apply \cite[Lem.~7.3]{DuzaarMingione}, which implies that for any $M>1$, there exist $\rho_0(M),\ \tilde\Psi_0(M)>0$ such that, if $\bar z$ and $\delta$ satisfy
\[ \abs{m_{\bar z,\delta}}+\norm{A_{\bar z,\delta}}\leq M,\quad \delta\leq\rho_0(M),\quad \Psi_2(\bar z,\delta,l_{\bar z,\delta})+\delta^{2\beta}\leq\tilde\Psi_0(M), \]
then there exists $\Gamma_{\bar z}\in\R^{3\times3}$ which satisfies the Campanato estimate
\begin{align*}
\frac1{\abs{Q_r(\bar z)}}\int_{Q_r(\bar z)}\norm{\nabla_xu-\Gamma_{\bar z}}^2\,dz
&\leq c\left[\left(\frac{2r}{\delta}\right)^{2\alpha}\Psi_2(\bar z,\delta,l_{\bar z,\delta})+r^{2\beta}\right]\\
&\leq c\left[\left(\frac2\delta\right)^{2\beta}\Psi_2(\bar z,\delta,l_{\bar z,\delta})+1\right]r^{2\beta}
\end{align*}
for all $0<r\leq\delta/2$, with $c>0$ depending on the fixed structural constants and exponents. The second inequality follows from $\beta<\alpha$ and $r^{2\alpha}\leq(\delta/2)^{2(\alpha-\beta)}r^{2\beta}$.

In our case, since $A_{\bar z,\delta}=0$ for $u=0$, we fix $M=2$ and take $\delta\leq\rho_0(M)$, subject to the preceding geometric restrictions, such that $\delta^{2\beta}\leq\tilde\Psi_0(M)/2$. Keeping $\varepsilon_0$, $t_0$, and $\delta$ fixed, we decrease $\delta_0\leq\delta_{\rm loc}$ until
\begin{align*}
\abs{m_{\bar z,\delta}}+\norm{A_{\bar z,\delta}}
&\leq(\delta+5)\sqrt{\Psi_2(\bar z,\delta,0)}\leq(\delta+5)\sqrt{C_\delta}\,\delta_0\leq M;\\
\Psi_2(\bar z,\delta,l_{\bar z,\delta})
&\leq\Psi_2(\bar z,\delta,0)\leq C_\delta\delta_0^2\leq\tilde\Psi_0(M)/2.
\end{align*}
Both conditions then hold uniformly in $\bar z$.

By the Campanato characterization of H\"older spaces and a finite covering of $\T$, we obtain a uniform spatial H\"older seminorm for $\nabla u_t$ for $t\geq t_0$. Since $\int_\T\nabla u_t=0$ by periodicity, this implies
\begin{equation}\label{eq:dxc0a}
\norm{\nabla u_t}_{C^{0,\beta}(\T)}\leq\tilde C\qquad\text{for all }t\geq t_0,
\end{equation}
where $\tilde C$ depends only on the fixed radius and the structural constants. Since $u_t$ has zero mean, this also implies a uniform $C^{1,\beta}$ bound for $u_t$.

\paragraph{Step 4: Sobolev interpolation $\implies$ global $L^\infty$ smallness.} Fix $t \geq t_0$. The fractional Gagliardo-Nirenberg inequality \cite[Thm.~1]{BrezisMironescu} (with the choices $s_1 = 0$, $p_1 = 2$, $s_2 = 1$, $p_2 = \infty$) implies that
\[ \norm{u_t}_{W^{s,p}(\T)} \leq C \norm{u_t}_{L^2(\T)}^\theta \norm{u_t}_{W^{1,\infty}(\T)}^{1-\theta}  \]
for $\theta \in (0,1)$ such that
\[ s = \theta s_1 + (1-\theta) s_2 = 1 - \theta, \quad \frac1p = \frac{\theta}{p_1} + \frac{1-\theta}{p_2} = \frac\theta2. \]
We choose $\theta$ so that $W^{s,p}$ embeds into $L^\infty$, i.e.
\[ s > \frac3p \quad \iff \quad 1 - \theta > \frac32 \theta \quad \iff \quad \theta < \frac25. \]
Note that the constant $C$ is independent of $\bar x$, since all balls $B_\delta(\bar x)$ with fixed radius $\delta$ are isometric. 

Then by the exponential decay of $\norm{u_t}_{L^2(\T)}$ and 
\[ \norm{u_t}_{W^{1,\infty}(\T)} \leq c_\infty \norm{\nabla u_t}_{L^{\infty}(\T)} \leq c_\infty \tilde C, \]
we conclude that
\begin{equation}\label{eq:xlinfty} \norm{u_t}_{L^\infty(B_\delta(\bar x))} \leq C' \delta_0^{\theta}, \quad t \geq t_0 \end{equation}
where $C'$ is independent of $\bar z$ and $\delta_0$.

\paragraph{Step 5: global Lipschitz smallness.} 
We close the argument by combining the uniform H\"older bound \eqref{eq:dxc0a} on $\nabla u_t$ with the $L^\infty$ bound \eqref{eq:xlinfty} on $u_t$ to obtain an estimate for $\norm{\nabla u_t}_{L^\infty}$ that tends to zero with $\delta_0$.

Let $f \in C^{1,\alpha}(\T)$, let $x \in \T$ and $h = t \nabla f(x)$ for some $t > 0$. Then by the Mean Value Theorem, there exists some $s \in [0,1]$ such that
\[ f(x+h) - f(x) = \nabla f(x) \cdot h + \bigl( \nabla f(x+sh) - \nabla f(x) \bigr) \cdot h \]
hence
\[ \abs{h} \abs{\nabla f(x)} = \nabla f(x) \cdot h \leq 2 \norm{f}_{L^\infty} + \norm{\nabla f}_{C^{0,\alpha}} \abs{h}^{1+\alpha}. \]
Choosing $t > 0$ so that
\[ \frac{2 \norm{f}_{L^\infty}}{\abs{h}} = \norm{\nabla f}_{C^{0,\alpha}} \abs{h}^{\alpha} \quad \iff \quad \abs{h}^{1+\alpha} = \frac{2 \norm{f}_{L^\infty}}{\norm{\nabla f}_{C^{0,\alpha}}}, \]
we obtain the bound
\[ \abs{\nabla f(x)} \leq 2 \norm{\nabla f}_{C^{0,\alpha}} \abs{h}^{\alpha} = 2^{1+\frac{\alpha}{1+\alpha}} \norm{f}_{L^\infty}^{\frac{\alpha}{1+\alpha}} \norm{\nabla f}_{C^{0,\alpha}}^{1-\frac{\alpha}{1+\alpha}}, \]
which holds for all $x \in \T$.

Now applying this scalar estimate to each component of $u_t$, with exponent $\beta$ in place of $\alpha$, we can combine \eqref{eq:dxc0a} and \eqref{eq:xlinfty} to obtain
\begin{equation}
    \norm{\nabla u_t}_{L^\infty(\T)} \leq C' \delta_0^{\frac{\theta\beta}{1+\beta}} \quad \text{for all} \quad t \geq t_0,
\end{equation}
where $C'$ is independent of $\delta_0$.

Finally, decrease $\delta_0$ so that $C'\delta_0^{\theta\beta/(1+\beta)}\leq\varepsilon_0$. Together with Lemma~\ref{lem:shortreg}, this gives $\norm{\nabla u_t}_{L^\infty}\leq\varepsilon_0$ for all $t>0$.

To conclude, in combination with Lemma~\ref{lem:shortreg}, we have proved the following:

\begin{lemma}\label{lem:globreg}
Let $3 < p < \infty$ and $1+3/p < \sigma < 2$. For any $\varepsilon_0 > 0$, there exists $\delta_0 > 0$ such that for any initial data $\norm{u_0}_{W^{\sigma,p}_{\per,0}} \leq \delta_0$, the unique solution $u_t$ to the gradient flow is uniformly bounded in $C^{1,\beta}$, and satisfies
\[  \norm{\nabla u_t}_{L^\infty} \leq \varepsilon_0\quad \text{for all} \quad t > 0. \]
\end{lemma}

In particular, there exists $\delta_0 > 0$ such that as long as $\norm{u_0}_{W^{\sigma,p}_{\per,0}} \leq \delta_0$, $\cG(u_t) = \cF(\id+u_t)$ for all $t \geq 0$ and the gradient flow for $\cG$ is actually a gradient flow for $\cF$. This completes the proof of Theorem~\ref{thm3}.

\subsection{Linearized equation and sharp decay}

We now confine ourselves to the trapped regime, and split the gradient flow into its linearization at $u = 0$ plus a remainder. First recall that the convexified cost $F_0$ expands around $I$ as follows:
\[ F_0(I+M) = F(I) + \frac12 H_\alpha(M) + O(\norm{M}^3). \]
Denoting by $G_0 \colon \R^{3\times 3} \to \R^{3\times3}$ the positive definite linear operator inducing the quadratic form $H_\alpha$, in the sense that $H_\alpha(M) = \ip{M}{G_0 M}$, the gradient of $F_0$ also expands as
\[ \nabla F_0(I+M) = G_0 M + O(\norm{M}^2). \]
The operator $G_0$ is given precisely by a linear combination of projections onto invariant subspaces:
\[ G_0 M = \frac{103}{12} \pi_1 M + \frac13 \pi_2 M + \frac53 \pi_3 M + \frac13 \pi_4 M. \]
Setting $N(M) := \nabla F_0(I+M) - G_0 M = O(\norm{M}^2)$ for the nonlinear remainder, we can rewrite the gradient flow \eqref{eq:gflow} as
\begin{align}
\label{eq:gflowlinrem}  \partial_t u_t & = \cL_0 u_t + \div_x N(\nabla u_t); \quad \cL_0 u := \div G_0 \nabla u, 
\end{align}
where the linear operator $\cL_0$ is elliptic with convexity constant $\frac13$. 

Consequently, the linearized equation
\[ \partial_t u_t = \cL_0 u_t \]
exhibits exponential decay at the rate $\frac13 C_P$. 

Our goal is to show that in the trapped regime, the linear term $\cL_0$ dominates, and the same optimal decay rate of $\lambda_\ast = \frac13 C_P$ is attained for the nonlinear flow:

\begin{theorem}\label{thm:sharpdecay}
Under the same assumptions as Theorem~\ref{thm3}, the solution $(u_t)_t$ to the gradient flow satisfies
\[ \norm{u_t}_{C^{m,\alpha}} \leq C_{m,\alpha} e^{-\lambda t}, \quad \lambda = \frac13 C_P, \quad t \geq 1 \]
for every $m \geq 0$ and $\alpha \in (0,1)$. 
\end{theorem}

\paragraph{Step 1: preliminary a priori estimates.}
The first observation is to make the defect in the convexity constant of $F_0$ at $I+M$ dependent on $M$ itself. 

By the smoothness of $F_0$, there exist constants $\rho, c_\ast > 0$ such that 
\begin{equation}\label{eq:rhocast} \norm{M} \leq \rho \quad \implies \quad \norm{N(M)} \leq c_\ast \norm{M}^2. \end{equation}

We now define the quantities
\[ a(t) := \norm{u_t}_{L^2}, \quad b(t) := \norm{\nabla_x u_t}_{L^\infty}. \]
In the trapped regime, we can choose $b(t) \leq \varepsilon_0 \leq \rho$ at all times, so that almost everywhere in space,
\[ \bigl\lvert\ip{N(\nabla u_t)}{\nabla u_t}\bigr\rvert \leq c_\ast \norm{\nabla u_t}^2 \norm{\nabla u_t} \leq c_\ast b(t) \norm{\nabla u_t}^2. \]
Then testing the gradient flow \eqref{eq:gflowlinrem} against $u$, we obtain
\[ \frac12 \frac{d}{dt} a(t)^2 = - \frac1{\abs\T} \int_\T \ip{G_0 \nabla u_t + N(\nabla u_t)}{\nabla u_t} \,dx \leq - \left( \frac13 - c_\ast b(t) \right) \norm{\nabla u_t }_{L^2}^2. \]
Applying Poincar\'e--Wirtinger then gives
\[  \frac12 \frac{d}{dt} a(t)^2 \leq -\left(\lambda_\ast - \kappa b(t) \right) a(t)^2, \qquad \kappa := c_\ast C_P. \]
Integrating, we obtain
\begin{equation} a(t) \leq a(0) \exp\left( - \lambda_\ast t + \int_0^t \kappa b(s) \,ds \right). \end{equation}
A priori, trapping only gives the boundedness of $b(t)$. If we can further show that $b \in L^1$, the integral can be absorbed into the constant, yielding the optimal rate
\[ a(t) \leq a(0) e^{\kappa \norm{b}_{L^1}} e^{-\lambda t}. \]

\paragraph{Step 2: parabolic smoothing gives integrable defect.} 
Recall the two inequalities used in the previous section, valid for periodic functions $f$ with zero mean:
\[ \norm{f}_{L^\infty} \leq C \norm{f}_{L^2}^\theta \norm{\nabla f}_{L^\infty}^{1-\theta}, \qquad 0 < \theta < \frac25, \]
and
\[ \norm{\nabla f}_{L^\infty} \leq C \norm{f}_{L^\infty}^{\frac{\beta}{1+\beta}} \norm{\nabla f}_{C^{0,\beta}}^{\frac{1}{1+\beta}}. \]
Combining them and using the uniform $C^{1,\beta}$ bound on $u_t$ in the trapped regime, we obtain, for $t\geq t_0$,
\[ b(t) = \norm{\nabla u_t}_{L^\infty} \leq C \norm{u_t}_{L^2}^{\frac{\beta\theta}{1+\beta}} \norm{\nabla u_t}_{C^{0,\beta}}^{\frac{1}{1+\beta}} \norm{\nabla u_t}_{L^\infty}^{\frac{\beta(1-\theta)}{1+\beta}} \leq C \tilde C^{\frac{1+\beta(1-\theta)}{1+\beta}} a(t)^{\frac{\beta\theta}{1+\beta}}. \]
Since we already know the $L^2$ norm $a(t)$ to decay exponentially, if at a suboptimal rate, we can conclude that $b(t) \leq C \exp(-\mu t)$ for some $\mu > 0$.

This concludes that $b \in L^1$ hence $a(t) \leq C \exp(-\lambda_\ast t)$.

\paragraph{Step 3: transfer to higher norms.}
The uniform $C^{1,\beta}$ bound and the smoothness of $F_0$ allow us to apply parabolic regularity estimates repeatedly, obtaining uniform bounds for all derivatives of $u_s$ for $s\geq1$. We write the equation as
\[
  \partial_s u_s=\sum_{i,j}B_{ij}(s,x)\,\partial_i\partial_j u_s,
\]
where $B_{ij}(s,x)$ are the $3\times3$ blocks of $D^2F_0(I+\nabla u_s(x))$. These coefficients are uniformly elliptic and have uniform derivative bounds. For Schauder estimates for linear parabolic systems,
we refer to \cite{Schlag1996}. Interior estimates on $[t-1,t]\times\T$, together with the $L^2$ decay proved above, therefore give
\[
  \norm{u_t}_{C^{m,\alpha}}\leq C_{m,\alpha}\left(\int_{t-1}^{t}\norm{u_s}_{L^2}^2\,ds\right)^{1/2}\leq C_{m,\alpha}e^{-\lambda_\ast t},\qquad t\geq2,
\]
where $C_{m,\alpha}$ is independent of $t$. Since the solution is smooth for positive times, increasing $C_{m,\alpha}$ gives the same estimate for $1\leq t\leq2$.

This completes the proof of Theorem~\ref{thm:sharpdecay}.

\appendix

\section{Calculations for the unperturbed BCC lattice}\label{app:unperturb}

We now compute the quantities $\lambda_l$ associated to each square and hexagonal pyramid in the Voronoi cell of the unperturbed BCC lattice, as well as the quantities $\xi_{ll'}$ associated to the intersections of pyramids.

\subsection{Square faces}

Each square pyramid of the Voronoi cell has a height of $1$, and a base square with side length $\frac{1}{\sqrt2}$. By symmetry under the action of $G$, we assume wlog that $l = (0,0,-2) \in N_{ax}$.

Denote the base square by $B = [-a,a]^2$, $a = \frac1{2\sqrt2}$ and the outward-facing height vector by $h_S = \frac{l}{2} = (0,0,-1)$, so that $\Delta \cong \Delta-h_S = \mathrm{conv}\{-h_S,B\}$. 

At a height $z \in [0,1]$, the cross section of the pyramid $\tilde\Delta = \Delta-h_S$ will be given by $-zh_S+(1-z)B$, with $h_S \bot B$. 
That is, the pyramid $\Delta-h_S$ is the image of the cube $[-a,a] \times [-a,a] \times [0,1]$ under the change of variables $(x,y,z) \mapsto ((1-z)x, (1-z)y, z)$.

The second moment of the pyramid around $-h_S$ is then
\[ \lambda_{ax} = \int_{\tilde\Delta} \abs{ x - (- h_S) }^2 \,dx = \int_{-a}^a \int_{-a}^a \int_0^1 \abs{((1-z)x, (1-z)y, z-1)}^2 (1-z)^2 \,dz \,dy \,dx, \]
where $(1-z)^2$ is the Jacobian determinant.

Expanding out the norm and integrating first over $x$ and $y$,
\begin{align*}
    \lambda_{ax} & = \int_{-a}^a \int_{-a}^a \int_0^1 \left[ x^2+y^2 + 1 \right] (1-z)^4 \,dz \,dy \,dx
    \\ & = \left( \frac{8a^4}{3} + 4a^2 \right) \int_0^1 (1-z)^4 \,dz 
    \\ & = \frac15 \left( \frac{1}{24} + \frac12 \right) = \frac{13}{120}.
\end{align*}

\subsection{Hexagonal faces}

Hexagonal pyramids have a base $H$ with side length again $2a = \frac{1}{\sqrt2}$, and height instead $h = \abs{ h_H }= \frac{\sqrt 3}{2}$ where $h_H = - h e_3$. 

We represent the hexagonal base $H$ by the change of variables $(x,y) \mapsto (\frac{2\sqrt3a - \abs{y}}{\sqrt3a} x, y)$ where $x \in [-a,a]$, $y \in [-\sqrt3 a, \sqrt3 a]$. Then $\tilde\Delta = \mathrm{conv}\{-h_H, H\}$ is the image of the cube $[-a,a] \times [-1, 1] \times [0,1]$ along the map
\[ (x,y,z) \mapsto ((1-z) (2-\abs{y})x, (1-z)\sqrt3 a y, h z ), \]
with Jacobian determinant $\sqrt3 a h (1-z)^2 (2-\abs{y})$ whenever $y \neq 0$. We then have
\[ \lambda_H = \int_{\tilde\Delta} \abs{ x+h_H }^2 \,dx = \sqrt3 a h \int_{-a}^a \int_{-1}^{1} \int_0^1 (1-z)^4 \left( (2- \abs{y})^2 x^2 + 3 a^2 y^2 + h^2 \right) (2- \abs{y}) \,dz \,dy \,dx \]
so that
\begin{align*} 
\lambda_H & = \frac{\sqrt3 a h}{5} \int_{-a}^a \int_{-1}^{1} \left( (2-\abs{y})^2 x^2 + 3 a^2 y^2 + h^2 \right) (2-\abs{y}) \,dy \,dx
\\ & = \frac{3 a}{5} \int_{-a}^a \int_0^1 \left( (2-y)^2 x^2 + 3 a^2 y^2 + h^2 \right) (2-y) \,dy \,dx
\\ & = \frac{3 a}{5} \int_0^1 \left( (2-y)^2 \cdot \frac{2a^3}{3} + 3 a^2 y^2 \cdot 2a + h^2 \cdot 2a \right) (2-y) \,dy
\\ & = \frac{3 a}{5} \left( \frac{5 a^3}{2} + \frac{5 a^3}{2} + 3 a h^2 \right) = 3 a^4+ \frac95 a^2 h^2 = \frac{3}{64} + \frac{27}{160} = \frac{69}{320}.
\end{align*}

\subsection{Triangular boundaries}

We now treat the interfaces $E_{ll'} = \Delta_l \cap \Delta_{l'}$, which are triangles with vertices $0$, $v_1$ and $v_2$ where $v_1$ and $v_2$ are vertices adjacent to $l$ and $l'$.

Since each vertex in the BCC Voronoi cell has the same norm $\frac{\sqrt5}{2}$ and each edge of the Voronoi cell has the same length $2a = \frac1{\sqrt2}$, it suffices to consider a single isosceles triangle in the plane:
\[ T = \{ (x,y) \mid 0 \leq x \leq h, \ -\frac{ax}{h} \leq y \leq \frac{ax}{h} \}, \quad a = \frac{1}{2\sqrt2}, \quad h = \frac{3}{2\sqrt2}. \]
Here the vertices $v_1$, $v_2$ correspond to $(h, \pm a)$, both with norm $\frac{\sqrt5}{2}$.
Then
\begin{align*} 
\xi_T & = \int_0^h \int_{-ax/h}^{ax/h} (x^2+y^2) (x,y,0) \,dy \,dx
\\ & = \int_0^h \left(x^3 \cdot \frac{2ax}{h} + x \cdot \frac{2a^3x^3}{3h^3}, 0, 0 \right) \,dx
\\ & = \left( \frac{2ah^5}{5h} + \frac{2a^3h^5}{15h^3} \right) (1,0,0)
\\ & = \frac{2ah^2}{15} (3h^2+a^2) (1,0,0) 
\\ & = \frac{3}{40\sqrt2} \frac{28}{8} (1,0,0) = \frac{21}{80\sqrt2} (1,0,0).
\end{align*}
Hence for any such interface $E_{ll'}$ we conclude that
\[ \xi_{ll'} = \beta \frac{v_1+v_2}{2}, \quad \beta = \frac{7}{40} \]
since $\frac{v_1+v_2}{2} = h (1,0,0)$.

\bibliographystyle{alpha}
\nocite{*}
\raggedright
\bibliography{refs}

\end{document}